\documentclass[12pt,reqno]{amsart}
\usepackage{todonotes,color} 
\usepackage[utf8]{inputenc}
\usepackage{epsfig}
\usepackage{wrapfig}
\usepackage[centertags]{amsmath}
\usepackage{amssymb, textcomp}
\usepackage{amsthm}
\usepackage{url}
\usepackage{comment}
\usepackage{color}

\numberwithin{equation}{section}

\usepackage{hyperref}

\usepackage{graphicx}
\usepackage{graphics}
\usepackage{overpic}
\usepackage{hyperref}
\usepackage{pst-node}
\usepackage{tikz-cd}

\newtheorem{lemma}{Lemma}[section]

\newtheorem{definition}[lemma]{Definition}
\newtheorem{remark}[lemma]{Remark}

\newtheorem{proposition}[lemma]{Proposition}
\newtheorem{theorem}[lemma]{Theorem}

\DeclareMathOperator{\Rp}{Re}
\DeclareMathOperator{\Ip}{Im}

\def\al{\alpha}

\def\eps{\varepsilon}

\def\R{\mathbb R}
\def\RR{\mathbb R}
\def\CC{\mathbb C}

\def\PP{\mathcal P}

\def\LL{\mathcal L}

\def\EE{\mathcal E}

\def\cD{\mathcal D}
\def\NNN{\mathcal N}
\def\ZZZ{\mathcal Z}

\def\FF{\mathcal F}

\def\GG{\mathcal G}
\def\II{\mathcal I}
\def\XX{\mathcal X}
\def\YY{\mathcal Y}
\def\RRR{\mathcal R}

\def\OO{\mathcal O}

\def\D{\mathcal D}

\def\~{\tilde}

\def\ga{\gamma}
\def\de{\delta}

\def\rr{\rho}

\def\wt{\widetilde}
\def\ol{\overline}
\def\wh{\widehat}
\def\kk{\kappa}
\def\eps{\varepsilon}
\def\pa{\partial}

\def\out{\mathrm{out}}
\def\inn{\mathrm{in}}
\def\exp{\mathrm{exp}}
\def\aux{\mathrm{aux}}
\def\match{\mathrm{match}}

\def\max{\mathrm{max}}
\def\Nin{\mathcal{J}^{\inn}}
\def\Nch{\mathcal{J}^{\match}}
\def\Gmatch{\mathcal{F}^{\match}}
\def\Ghmatch{\wh{\mathcal{F}}^{\match}}
\def\Gindif{\wt{\mathcal{G}}^\inn}

\newcommand{\us}{\star}

\newcommand{\uns}{\mathrm{u}}
\newcommand{\sta}{\mathrm{s}}

\newcommand{\rplus}{\rho_{+}}
\newcommand{\rminus}{\rho_{-}}

\title[Homoclinic orbits arising near a saddle-center point]{On a countable sequence of homoclinic orbits arising near a saddle-center point}

\author{Inmaculada Baldom\'a}
\address{Departament de Matem\`atiques \& IMTECH, Universitat Polit\`ecnica de Catalunya, Diagonal 647, 08028 Barcelona, Spain and Centre de Recerca Matem\`atica, Campus de Bellaterra, Edifici C, 08193 Barcelona, Spain}
\email{immaculada.baldoma@upc.edu}
\author{Marcel Guardia}
\address{Departament de Matem\`atiques i Inform\`atica, Universitat de Barcelona, Gran Via, 585, 08007 Barcelona, Spain and Centre de Recerca Matem\`atica, Campus de Bellaterra, Edifici C, 08193 Barcelona, Spain}
\email{guardia@ub.edu}
\author{Dmitry E. Pelinovsky}
\address{Department of Mathematics and Statistics, McMaster University, Hamilton, Ontario, Canada, L8S 4K1}
\email{dmpeli@math.mcmaster.ca}

\begin{document}

\maketitle

\begin{abstract}
Exponential small splitting of separatrices in the singular perturbation theory leads generally to nonvanishing oscillations near a saddle--center point and to nonexistence of a true homoclinic orbit. It was conjectured long ago that the oscillations may vanish at a countable set of small parameter values if there exist a quadruplet of singularities in the complex analytic extension of the limiting homoclinic orbit. The present paper gives a rigorous proof of this conjecture for a particular fourth-order equation relevant to the traveling wave reduction of the modified Korteweg--de Vries equation with the fifth-order dispersion term. 
\end{abstract}


\tableofcontents

\section{Introduction}

Homoclinic orbits arise in dynamical systems at the intersections of stable and unstable manifolds 
(also known as the separatrices) associated to a saddle equilibrium point. They represent spatial profiles of traveling solitary waves in nonlinear dispersive wave equations from which spatial dynamical systems are obtained in the traveling reference frame. Existence of a homoclinic orbit connected at a saddle point is a generic phenomena in a planar  Hamiltonian system if there exists a center point near the saddle point. 

The phase space of many spatial dynamical systems has the dimension higher than two, in which case the equilibrium point may admit a center manifold in addition to the stable and unstable manifolds. For such a saddle-center 
point, intersection of the separatrices is not generic 
and homoclinic orbits do not generally exist. The corresponding traveling solitary waves are not fully decaying since their spatial profiles approach the oscillatory tails spanned by orbits along the center manifold.

It is rather common in analysis of solitary waves to consider an asymptotic limit when a higher-dimensional spatial dynamical system with a saddle-center point formally reduces to the planar Hamiltonian dynamical system with a homoclinic orbit. This leads to the main question of the singular perturbation theory if the homoclinic orbit persists under the perturbation. The standard answer to this question is negative because the exponentially small splitting of the separatrices generally occurs due to the singular perturbations. 

First examples of the exponentially small (beyond-all-order) phenomena 
and the relevant asymptotic analysis can be found in~\cite{Chapman1,Daalhuis1,Gel1,GrimshawJoshi,Kruskal,PGR}. Rigorous mathematical analysis and the proof of the existence of oscillatory tails near the saddle-center point in four-dimensional spatial dynamical systems was later developed in \cite{Lombardi,Tovbis1}. The oscillatory tails are present 
if a certain constant (called the Stokes constant) is nonzero, the proof of which usually relies on numerical computations. The numerical data in \cite{Tovbis2} for a particular model of the fifth-order Korteweg--de Vries (KdV) equation suggest that the Stokes constant is generally nonzero but may vanish along bifurcations of co-dimension one if another parameter is present in the spatial dynamical system. 
 
Compared to the standard setting of the non-vanishing oscillatory tails in the beyond-all-order expansions, a rather novel mechanism of obtaining a countable number of true homoclinic orbits was proposed in \cite{prl-alfimov}. The mechanism is related to the location of singularities of the truncated homoclinic orbit in a complex plane. 
If there is only one symmetric pair of singularities in the complex plane nearest to the real line, then the Stokes constant is generally nonzero and no true homoclinic orbit persists in the singular perturbation theory.  
However, if there exist a quadruplet with two symmetric pairs of singularities at the same distance from the real line, then the singular perturbation theory 
exhibits a countable set of true homoclinic orbits as the small parameter goes to zero. 

The theory from \cite{prl-alfimov} was illustrated on a number of other mathematical models involving nonlocal integral equations \cite{Alf1}, lattice advance-delay equations \cite{Alf2,Lustri2}, and differential advance-delay equations for traveling waves in lattices \cite{Deng2,Deng1,Lustri1,Porter}. The spatial profiles of solitary waves in such models must generally exhibit oscillatory tails  (in which case, they are usually called generalized solitary waves or nanoptera), see analysis in \cite{Faver1,Faver2} and numerical results in \cite{Tim,Porter,Vain}. However, the tails miraculously vanish along a countable set of bifurcation points if the singular limit admits a real analytic solution with a quadruplet of complex singularities nearest to the real line. A similar idea for homoclinic orbits in symplectic discrete maps has been discussed in \cite{Simo} some time before \cite{prl-alfimov}. 

Despite a number of examples supporting the conjecture from \cite{prl-alfimov}, no mathematically rigorous proof was developed in the literature. The purpose of this paper is to give a proof of this conjecture for the simplest four-dimensional dynamical system with a saddle-center equilibrium point. 

\subsection{Main model}

Let $\gamma, \eps \in \R$ be parameters and consider the fourth-order equation for some $u \in C^{\infty}(\mathbb{R},\mathbb{R})$,
\begin{equation}\label{eq:mainfourthorder}
\eps^2 u''''+(1-\eps^2)u''-u+u^2+2\gamma u^3=0.
\end{equation}
If $\eps$ is a small parameter, then the formal limit $\eps \to 0$ yields the second-order equation
\begin{equation}\label{eq:mainsecondorder}
u''-u+u^2+2\gamma u^3=0
\end{equation}
with $(0,0)$ being a saddle point of the planar Hamiltonian system 
\begin{equation}\label{eq:planarHam}
\left\{ \begin{array}{l} 
u' = w, \\
w' = u-u^2-2\gamma u^3.
\end{array} \right.
\end{equation}

The second-order equation~\eqref{eq:mainsecondorder} appears in the traveling wave reduction of the modified Korteweg--de Vries (KdV) equation 
\begin{equation}
    \label{mKdV}
    \frac{\partial \eta}{\partial t} + 2 \eta \frac{\partial \eta}{\partial x} + 6 \beta \eta^2 \frac{\partial \eta}{\partial x} + \frac{\partial^3 \eta}{\partial x^3} = 0,
\end{equation}
where $\eta = \eta(x,t)$ is real and $\beta$ is a parameter. 
Traveling waves of the modified KdV equation~\eqref{mKdV} correspond to the form $\eta(x,t) = \eta_c(x-ct)$ with the wave speed $c$ and the wave profile $\eta_c$ found from the third-order equation 
\begin{equation}
    \label{KdV-ode}
\eta_c'''(x) - c \eta_c'(x) + 2 \eta_c \eta_c'(x) + 6 \beta \eta_c^2 \eta_c'(x) = 0.
\end{equation}
If $c > 0$, the scaling transformation $\eta_c(x) = c u(\sqrt{c} x)$ and integration of~\eqref{KdV-ode} with zero integration constant for solitary wave solutions yields equation~\eqref{eq:mainsecondorder} with $\gamma := \beta c$. 

If $\gamma > 0$, there exist two families of periodic solutions and two solitary wave solutions of equation~\eqref{eq:mainsecondorder}, see, e.g., \cite{ChenP,LeP}. If $\gamma < 0$, there exists only one family of periodic solutions and only one solitary wave solution of equation~\eqref{eq:mainsecondorder}, see, e.g., \cite{MucaP}. This also follows from the phase portraits for the dynamical system~\eqref{eq:planarHam} on the phase plane $(u,w)$ shown 
in Figure~\ref{fig:plane} for $\gamma = 1$ (left) and $\gamma = -0.1$ (right). 

\begin{figure}[htb!]	
	\centering
		\includegraphics[width=6.5cm,height=5.5cm]{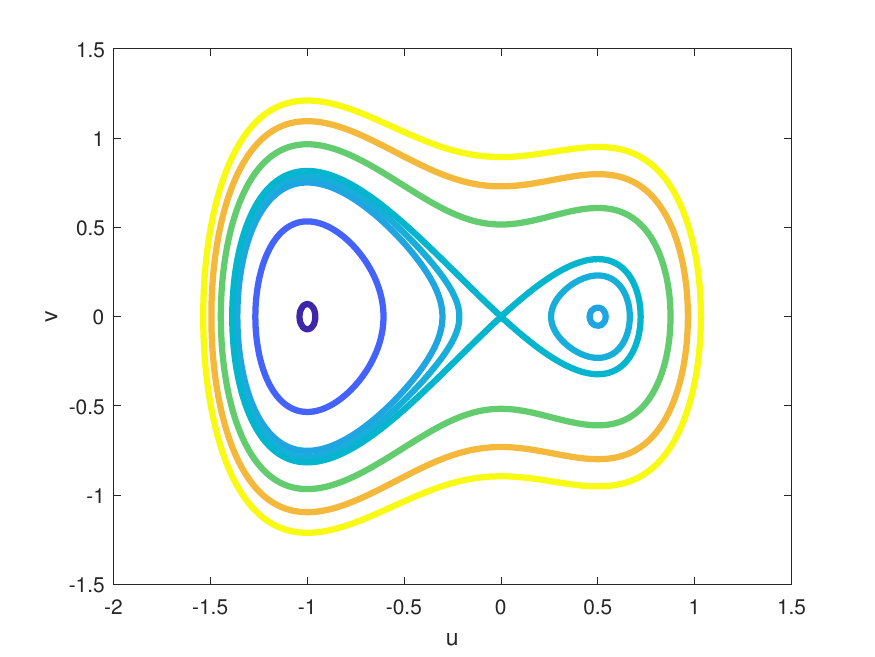}
		\includegraphics[width=6.5cm,height=5.5cm]{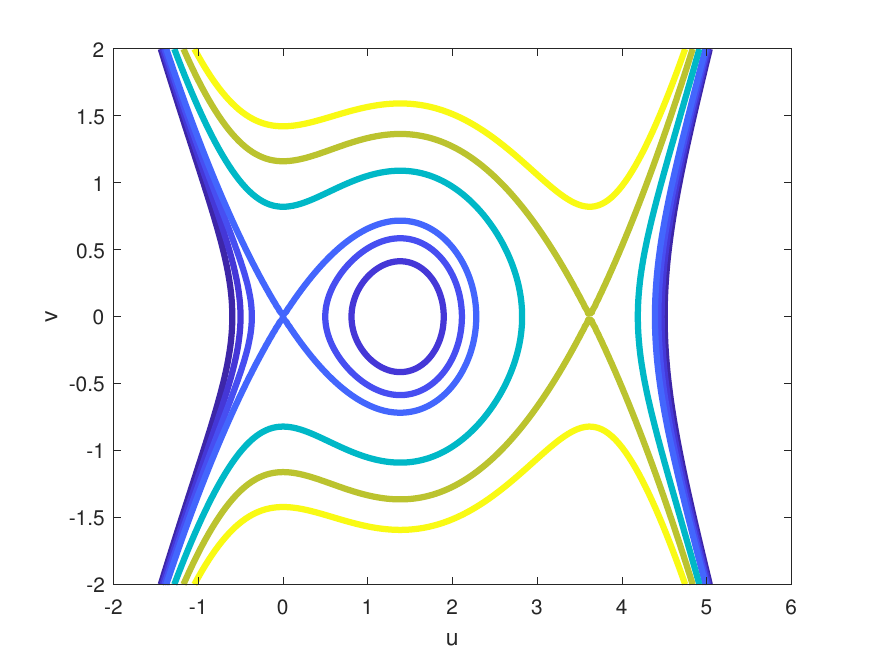}
	\caption{Phase portraits of~\eqref{eq:planarHam} for $\gamma = 1$ (left) and $\gamma = -0.1$ (right). }	
	\label{fig:plane}
\end{figure} 

The fourth-order equation~\eqref{eq:mainfourthorder} is the traveling wave reduction of the modified KdV equation with the fifth-order dispersion term, also known as the Kawahara equation~\cite{Kawahara},
\begin{equation}
    \label{Kawahara}
    \frac{\partial \eta}{\partial t} + 2 \eta \frac{\partial \eta}{\partial x} + 6 \beta \eta^2 \frac{\partial \eta}{\partial x} + \frac{\partial^3 \eta}{\partial x^3} + \alpha \frac{\partial^5 \eta}{\partial x^5} = 0,
\end{equation}
where $\alpha$ is another parameter. Traveling waves of the form $\eta(x,t) = \eta_c(x-ct)$ satisfy the fifth-order equation, which can be integrated once with the zero integration constant. The scaling transformation 
$\eta_c(x) = c u(\sqrt{c(1-\varepsilon^2)} x)$ yields~\eqref{eq:mainfourthorder} with $\gamma = \beta c$ and 
$\varepsilon^2$ found from the equation 
$$
\frac{\varepsilon^2}{(1-\varepsilon^2)^2} = \alpha c.
$$
This is always possible for small $\varepsilon$ if $\alpha c$ is small.

For $\beta = 0$, the Kawahara equation~\eqref{Kawahara} has been one of the main toy model of the shallow water wave theory to study periodic oscillations arising at the exponential tails of the solitary wave profiles, see recent works~\cite{Deconinck,Sprenger,Bridges}. Since the true homoclinic orbits are known not to exist for $\beta = 0$ \cite{GrimshawJoshi,PGR}, the main motivation for our study is to show the existence of a sequence of true homoclinic orbits in the modified Kawahara equation for $\beta \neq 0$.

The homoclinic orbit of the second-order system~\eqref{eq:planarHam} with $\gamma = 0$ is known in the exact analytical form:
\[
u_0(x) = \frac{3}{2} {\rm sech}^2\left(\frac{x}{2}\right). 
\]
It has double poles on the imaginary axis with the nearest singularities at $x = \pm \pi i$. If $\gamma \neq 0$, the double poles split into pairs of simple poles and the splitting is different for $\gamma > 0$ and $\gamma < 0$. 
The homoclinic orbit for $\gamma = 0$ is continued in the exact analytical form for every $1 + 9 \ga > 0$ as 
\begin{equation}\label{def:soliton}
u_0(x)=\frac{3}{\sqrt{1+ 9\ga}\cosh(x)+1}.
\end{equation}
For $\gamma > 0$, the double poles nearest to the real axis split along the imaginary axis as simple poles at 
\[
x = \pm i \pi \pm i \arccos\left(\frac{1}{\sqrt{1+9\ga}}\right),
\]
with four independent choices of signs. For $\ga \in (-\frac{1}{9},0)$ 
the double poles split off the imaginary axis as simple poles at 
\begin{equation}\label{def:singularities:intro}
x = \pm i \pi\pm \cosh^{-1}\frac{1}{\sqrt{1+9\ga}},
\end{equation}
again with four independent choices of signs. This is precisely the case which fits the theory from \cite{prl-alfimov} and coincides with Example 1 in \cite{prl-alfimov}. The numerical data on Figure 1 in \cite{prl-alfimov} already provide a convining evidence of the existence of a countable sequence 
$\{ \eps_n(\ga) \}_{n \in \mathbb{N}}$ for every $\ga \in (-\frac{1}{9},0)$ 
such that $\eps_n(\ga) \to 0$ as $n \to \infty$ with the homoclinic orbits persisting in the full equation~\eqref{eq:mainfourthorder} for $\eps = \eps_n(\ga)$ and with $u(x)$ being close to $u_0(x)$ in~\eqref{def:soliton}. 

Hence, in what follows we are only interested in the case $\ga \in (-\frac{1}{9},0)$, when the only homoclinic orbit with the profile $u_0$ is available in the form~\eqref{def:soliton}. For completeness, we mention that another homoclinic orbit exists for $\ga > 0$, see the left panel of Figure~\ref{fig:plane}, and its (negative) profile is given by 
\[
\tilde u_0(x)=-\frac{3}{\sqrt{1+ 9\ga}\cosh(x)-1}. 
\]
The simple poles of $\tilde{u}_0$ are located at the imaginary axis at 
\[
x = \pm i \arccos\left(\frac{1}{\sqrt{1+9\ga}}\right) + 2\pi i n, \quad n \in \mathbb{Z}.
\]
For $\gamma \leq 0$, $\tilde{u}_0$ is singular on real line and hence 
is neglected.

\subsection{Main result and the method of proof}

The main result of this paper is the following.
\begin{theorem}\label{thm:main:0}
For any $\gamma\in \left (-\frac{1}{9},0 \right )$, 
there exists $N_0\in\mathbb{N}$ large enough and a sequence $\{\eps_n\}_{n\geq N_0} $ of the form 
\begin{equation}
\label{eps-n-distribution}
\eps_n = \frac{\al}{n\pi} \left[ 1 + \frac{1}{n} \mathcal{O}\left(\frac{1}{\log n}\right) \right],\qquad \text{where}\quad \al=\cosh^{-1}\frac{1}{\sqrt{1+9\ga}},
\end{equation}
such that equation~\eqref{eq:mainfourthorder} with $\eps=\eps_n$ has a homoclinic orbit to the origin in $\R^4$.
\end{theorem}

We prove  this result by analyzing the stable and unstable invariant manifolds of the origin in $\R^4$ and measuring their distance at a suitable cross-section of $\R^4$. To this end, we rewrite the fourth-order equation \eqref{eq:mainfourthorder} as two second-order equations. By introducing
\begin{equation}\label{def:f}
f(u) := u^2 + 2 \ga u^3\quad\text{ and }\quad v := u''-u+f(u),
\end{equation}
 equation \eqref{eq:mainfourthorder} becomes the system
\begin{equation}\label{eq:system}
\left\{ \begin{array}{l}
  u''=u+v-f(u)\\
  v''=-\frac{1}{\eps^2}v+f'(u) (u+v-f(u))+f''(u)(u')^2.
 \end{array} \right.
\end{equation}
The phase space of system~\eqref{eq:system} is written in the variables $(u,u',v,v')\in\RR^4$. Moreover, this system has the first integral
\begin{equation}\label{def:firstintegral:2}
\begin{split}
G(u,u',v,v')=&(1-\eps^2)\frac{(u')^2}{2}-\frac{u^2}{2}+F(u)\\
&+\eps^2\left[u'(v'+u'-f'(u)u')-\frac{(u+v-f(u))^2}{2}\right],
\end{split}
\end{equation}
with 
\begin{equation*}
F(u)=\int_{0}^u f(v)dv =\frac{u^3}{3}+ \frac{\gamma u^4}{2}.
\end{equation*} 

We notice that the origin in $\R^4$ is a saddle-center equilibrium point of the second-order system \eqref{eq:system} with associated eigenvalues $\big \{-1,1, i \eps^{-2}, -i \eps^{-2}\big \}$ which are of different scales. Therefore, the stable and unstable manifold associated to the origin have dimension one and, thus, they are just trajectories of~\eqref{eq:system}.  

Since system~\eqref{eq:system} is autonomous, in order to find homoclinic connections, it is necessary that there exists a time parameterization of the stable and unstable invariant manifolds, denoted by 
\[
\big (u^{\star}(x), (u^{\star})'(x), v^{\star}(x), (v^{\star})'(x) \big ),
\quad \star=\uns,\sta
\]
(which also depend on the parameters $\eps$ and $\gamma$),  such that
\begin{equation*}
\big (u^{\uns}(0), (u^{\uns})'(0), v^{\uns}(0), (v^{\uns})'(0) \big )=  \big (u^{\sta}(0), (u^{\sta})'(0), v^{\sta}(0), (v^{\sta})'(0) \big).
\end{equation*} 

In a general setting two curves do not intersect in a four dimensional space, however system~\eqref{eq:system} is reversible with respect to the involution
\begin{equation}\label{def:involution}
\Psi:(u,u',v,v')\to (u,-u',v,-v')
\end{equation}
whose symmetry plane is 
\begin{equation}\label{def:symmetryaxis}
\Pi=\{ (u,u',v,v') \in \mathbb{R}^4 : \quad u'=0, \; v'=0\}.
\end{equation}
In other words, if $(u(x),u'(x),v(x),v'(x))$ is a solution of system~\eqref{eq:system}, then the function defined by $\Psi(u(-x),u'(-x),v(-x),v'(-x))$ is also a solution. In particular  
\[
u^{\sta}(x)=u^\uns(-x), \qquad v^{\sta}(x) = v^{\uns}(-x)
\]
and therefore $u^{\sta}(0)= u^{\uns}(0)$ and $v^{\sta}(0) = v^{\uns}(0)$.

As a consequence, a homoclinic orbit exists if the unstable curve to $(0,0,0,0)$ as $x \to -\infty$ intersects the symmetry plane $\Pi$. Indeed, if such intersection occurs, then the unstable curve to $(0,0,0,0)$ as $x \to -\infty$ is reflected by the involution to the stable curve to $(0,0,0,0)$ as $x \to +\infty$. 

It can be seen that the perturbed invariant manifolds can be approximated by the homoclinic orbit for the unperturbed problem~\eqref{eq:mainsecondorder},  $$
(u(x),u'(x),v(x),v'(x))=(u_0(x),u'_0(x), 0,0)
$$ 
with $u_0$ given in~\eqref{def:soliton}. 
Then, we define the section
\begin{equation}\label{def:section}
\Sigma=\{(u,u',v,v')\in\RR^4 : \quad u'=0\}.
\end{equation}
We observe that the homoclinic orbit  $(u(x),u'(x))=(u_0(x),0)$ 
of the second-order system~\eqref{eq:mainsecondorder} with $u_0$ computed in~\eqref{def:soliton}, satisfies $u_0'(0)=0$ and it intersects transversally the section $\Sigma$ with $(v,v') = (0,0)$.

Next theorem gives an asymptotic formula for the distance between the stable and unstable manifolds of the origin in $\R^4$ at $\Sigma$.

\begin{theorem}\label{thm:main}
There exist two unique solutions $(u^\uns, v^\uns)$ and $(u^\sta, v^\sta)$ of system~\eqref{eq:system} such that $(u^\uns)'(0)=(u^\sta)'(0)=0$ and 
\[
\lim_{x\to -\infty}(u^\uns(x), v^\uns(x))=0,\qquad \lim_{x\to +\infty}(u^\sta(x), v^\sta(x))=0.
\]
Moreover, there exists a constant $\Theta\in\RR$, $\Theta\neq 0$, such that
\[
\begin{split}
u^\uns(0)-u^\sta(0)&= 0\\
v^\uns(0)-v^\sta(0)&= 0\\
(v^\uns)'(0)-(v^\sta)'(0)&=-\frac{4\Theta  }{\sqrt{|\gamma|}\eps^3}e^{-\frac{\pi}{\eps}} \left(\sin\left(\frac{\alpha}{\eps}\right)   +\OO\left(\frac{1}{|\log\eps|}\right)\right).
\end{split}
\]
\end{theorem}

Theorem \ref{thm:main:0} is a direct consequence of Theorem \ref{thm:main}.

\begin{proof}[Proof of Theorem \ref{thm:main:0}]
Since the system~\eqref{eq:system} is reversible it is enough to obtain a point in the unstable manifold which intersects the symmetry plane $\Pi$ in \eqref{def:symmetryaxis}. Since 
\[(u^\uns(0),(u^\uns)'(0),v^\uns(0),(v^\uns)'(0))\in\Sigma
\]
it is enough to look for values of $\eps$ such that $(v^\uns)'(0)=0$.

By reversibility, 
\[
(u^\uns(0),(u^\uns)'(0),v^\uns(0),(v^\uns)'(0))=(u^\sta(0),-(u^\sta)'(0),v^\sta(0),-(v^\sta)'(0)).
\]
and therefore
\[
2 (v^\uns)'(0)=(v^\uns)'(0)-(v^\sta)'(0)=-\frac{4\Theta  }{\sqrt{|\gamma|}\eps^3}e^{-\frac{\pi}{\eps}} \left(\sin\left(\frac{\alpha}{\eps}\right)   +\OO\left(\frac{1}{|\log\eps|}\right)\right).
\]
Since $\Theta \neq 0$, the values of $\eps_n$  are found from roots of 
\[
\sin\left(\frac{\alpha}{\eps}\right)   +\OO\left(\frac{1}{|\log\eps|}\right) = 0,
\]
which yields \eqref{eps-n-distribution}.
\end{proof}

The main steps in the proof of Theorem \ref{thm:main} are explained in Section \ref{sec:outlineproof}. The proof of each step is deferred to Sections \ref{sec:outer}--\ref{sec:difference} and Appendices \ref{app-A}--\ref{app:g1:outer}.

\subsection{Exponentially small splitting of separatrices}\label{sec:splitting}

Theorem \ref{thm:main} fits into what is usually called \emph{exponentially small splitting of separatrices}. This phenomenon occurs in dynamical systems which have a hyperbolic behavior whose invariant manifolds are exponentially close with respect to a small parameter of the system. Here we review the literature on the topic and explain the main tools to deal with the exponentially small phenomenon.

The  exponentially small splitting of separatrices was first pointed out by Poincar\'e (see \cite{POI90}) and nowadays it is well known  that  appear in many analytic models with  multiple time scales and a conservative structure  (Hamiltonian, volume preserving) or reversibility.
The first rigorous analysis of this phenomenon was not achieved until the 80's in the seminal work by Lazutkin  on the standard map~\cite{Lazutkin84russian}, who proposed a scheme to prove the exponentially small transversality of the invariant manifolds of the saddle equilibrium point this map possesses. A full proof of this fact was obtained in 1999 by Gelfreich \cite{Gelfreich99}. 

The approach proposed by Lazutkin (detailed below in this section) has been implemented in multiple settings in the past decades such as area preserving maps \cite{DelshamsR98, MartinSS11a, MartinSS11b} and integrable Hamiltonian systems with a fast periodic or quasiperiodic forcing \cite{DelshamsS92, Gelfreich00, Sauzin01,BFGS12}. Note that the approach is extremely sensitive on the analyticity properties of the model and therefore ``implementing'' it in different settings is, by no means, straightforward. Strongly related to the present paper are those dealing with volume preserving or Hamiltonian Hopf-zero bifurcations. This was first addressed in \cite{BaldomaS06, BaldomaS08, BaldomaCS13, BaldomaCS18, BaldomaCS18b},  and has later been applied to the breakdown of breathers in the Klein-Gordon equation (which can be seen as an infinite dimensional Hopf-zero bifurcation) and in the invariant manifolds of $L_3$ in the restricted planar 3 body problem \cite{BaldomaGG22,BaldomaGG23}. Note that the exponentially small splitting of separatrices phenomena can be analyzed by other methods such as the so-called continuous averaging method \cite{Treshev97}.

Let us explain the main steps of the approach proposed by Lazutkin applied to Hopf-zero bifurcations. Note first that the unperturbed separatrix is analytic in a complex strip centered at the real line. Then, in all the mentioned works and in the approach explained below, one makes the strong assumption that, at each of the boundary lines of the strip, the separatrix has only one singularity. Then, an asymptotic formula for the distance between the perturbed invariant manifolds can be obtained following these steps.

\begin{enumerate}
\item Choose coordinates which  capture the slow-fast dynamics of the model so that it becomes a (fast) oscillator weakly coupled to an integrable system with a saddle point and a separatrix associated to it.
\item Prove the existence of the analytic continuation of suitable parametrizations of the perturbed invariant manifolds in appropriate complex domains. These domains contain a segment of
the real line and intersect a neighborhood sufficiently close to the singularities of the  separatrix.
	
\item Derive the inner equation, which gives the first order of the original system close to the singularities of the separatrix. This equation is independent of the perturbation parameter.
	
	\item Study two special solutions of the inner equation which are approximations of the perturbed invariant manifolds near the singularities and  provide an asymptotic formula for the difference between these two solutions of the inner equation.
\item By using complex matching techniques, compare the solutions of the inner equation with the parametrizations of the perturbed invariant manifolds.
\item Finally,  prove that the dominant term of the difference between manifolds is given by the term obtained from the difference of the solutions of the inner equation.
\end{enumerate}

This approach and all the aforementioned references rely on several hypotheses one has to assume on the model. In particular, as already said, one must assume that, at each of the boundary lines of its analyticity strip, the time-parameterization of the unperturbed separatrix has only one singularity. This assumption is rather strong and it is known to be non-generic (see \cite{prl-alfimov,Simo}). In particular, the model \eqref{eq:mainfourthorder} 
with $\gamma \in \left(-\frac{1}{9},0\right)$ we consider in this paper does not satisfy this hypothesis since  the separatrix has two singularities at each of these lines. 

As far as the authors know, no proof of exponentially small splitting of separatrices for  separatrices with multiple singularities with the same imaginary part existed until now. The reason is that to analytically extend the invariant manifolds to complex domains one needs to estimate quite sharply certain oscillatory integrals and this is not so straightforward when one has several singularities with the same imaginary part. In the present paper we propose a new approach which relies on considering ``auxiliary orbits'' of the model. The approach is rather flexible and we expect to be applicable to a wide set of models with\emph{any number} of singularities with the same imaginary part (see Section \ref{sec:furtherdirections} below).

Let us explain the main steps in the proof of Theorem \ref{thm:main}, comparing them with the classical Lazutkin's approach explained above. The singularities of the unperturbed separatrix closest to the real axis are those given in \eqref{def:singularities:intro}.

\begin{enumerate}
\item Choose coordinates which  capture the slow-fast dynamics of the model. In the present paper the coordinates in \eqref{eq:system} suffice. Note that this system possesses a first integral (see \eqref{def:firstintegral:2}).
\item Prove the existence of the analytic continuation of the time-parametrization of the perturbed unstable invariant manifolds
in  an appropriate complex domain (see \eqref{def:defdomainouter}). This domain contains a segment of
the real line and intersects a neighborhood sufficiently close to the  singularities of the  separatrix with negative real part (see \eqref{def:singularities:intro}). Analogously, extend the perturbed stable invariant manifold up to the singularities with positive real part. This is done in Theorem \ref{thm:outer:intro}.

\item Consider an auxiliary solution of \eqref{eq:system} which belongs to the same level  of the first integral and that can be defined in a  lozenge  shaped complex domain which contains a segment of the real line and domains $\eps$-close to all four singularities of the unperturbed separatrix (see \eqref{def:domainaux}). This is done in Theorem \ref{thm:aux:intro}. Note that this solution does not belong to neither the stable nor the unstable invariant manifold. Instead of measuring the distance between the stable and unstable invariant manifolds at a given section, we will measure the distance between the unstable manifold and the auxiliary solution and between the auxiliary solution and the stable manifold.
	
\item Derive the inner equation (see \eqref{eq:inner}), which gives the first order of the original system close to the singularities of the separatrix. Note that the same inner equation appears close to all four singularities in \eqref{def:singularities:intro}.
	
\item Study two special solutions of the inner equation and  provide an asymptotic formula for the difference between these two solutions of the inner equation. This is done in Theorem \ref{thm:inner}.
\item Close to the singularities with negative real part, by using complex matching techniques, compare the solutions of the inner equation with the parametrization of the perturbed unstable invariant manifold and the auxiliary solution (analogously close to the singularities with positive real part and the auxiliary solution and the parameterization of the stable invariant manifold). This is done in Theorem \ref{thm:matching:intro}.
\item Prove that the dominant term of the difference between the unstable manifold and the auxiliary solution is given by the term obtained from the difference of the solutions of the inner equation close to the singularities with negative real part (analogously for the stable manifold and the auxiliary solution close to the rightmost singularities). This is done in Propositions \ref{prop:difference:remainder} and \ref{prop:CsFormula}. Joining the two asymptotic formulas provides the difference between the stable and unstable invariant manifolds.
\end{enumerate}

\subsection{Further directions and applications}\label{sec:furtherdirections}

Although we have addressed a very particular model, the fourth-order equation \eqref{eq:mainfourthorder}, which is relevant for traveling waves of the modified Kawahara equation~\eqref{Kawahara}, the statement and proof of Theorem~\ref{thm:main} can be extended to other dynamical systems with the saddle-center points. 

One example where a sequence of homoclinic orbits appears in the singular perturbation theory was considered in \cite{Alf2}. The limiting second-order equation is given by 
\begin{equation}
    \label{NLS-sat}
    u'' - u + \frac{u^3}{1 + \gamma u^2} = 0,
\end{equation}
with a parameter $\gamma > 0$ and it appears as the standing wave reduction of the focusing nonlinear Schr\"{o}dinger (NLS) equation with a saturation term. If $\gamma = 0$, the homoclinic orbit is given by $u_0(x) = \sqrt{2} {\rm sech}(x)$
with the simple pole singularities along the imaginary axis at 
$$
x = \frac{i \pi (2n + 1)}{2}, \quad n \in \mathbb{Z}. 
$$
However, for every $\gamma > 0$ it was proven in \cite[Theorem 2.2]{Alf2} that 
the nearest singularities to the real line appear as a quadruplet in the complex plane. Hence, the numerical approximations in \cite[Section 3]{Alf2} showed the existence of a countable sequence of true homoclinic orbits when the limiting second-order equation~\eqref{NLS-sat} is perturbed by the fourth-order derivative term. 

This example is rather striking since the term $u^3/(1+\gamma u^2)$ with $\gamma > 0$ does not change the number and types of the critical points in the dynamical system on the real line, but only change the number and types of singularities in the complex plane. 

Another example appears in the cubic--quintic NLS equation 
\begin{equation}
\label{NLS-cubic-quintic}
    u'' - u + u^3 (1 + 3 \gamma u^2) = 0
\end{equation}
with another parameter $\gamma \in \mathbb{R}$. The homoclinic orbit is given by 
\[
u_0(x) = \frac{2}{\sqrt{1 + \sqrt{1 + 16 \gamma} \cosh(2x)}}.
\]
The simple pole singularity for $\gamma = 0$ at $x = \frac{i\pi}{2}$ splits vertically along the imaginary axis for $\gamma > 0$ and horizontally for $\gamma < 0$ with the square root branch point singularities. In the latter case, we have a quadruplet of square root singularities in the complex plane which lead to a sequence of homoclinic orbit if the second-order equation~\eqref{NLS-cubic-quintic} is perturbed by the fourth-order derivative term. 

For both models~\eqref{NLS-sat} and~\eqref{NLS-cubic-quintic}, the singularities in the complex plane are more complicated than poles and involve  branching points, see \cite{Alf2}.

The analytical proof of Theorem~\ref{thm:main} can be extended from fourth-order dynamical systems to other finite-dimensional dynamical systems. It is nevertheless an open direction to extend the analysis to the infinite-dimensional dynamical systems such as the differential advance-delay equations. Such situations with the saddle-center points and the quadruplets of singularities in the complex plane are well-known in the context of traveling solitary waves in diatomic Fermi--Pasta--Ulam (FPU) systems \cite{Deng2,Lustri1}. If the center manifold is still two-dimensional and the stable and unstable manifolds are infinite-dimensional, we conjecture that a similar sequence of true homoclinic orbits exist in the singular limit of the diatomic FPU system, in agreement with the numerical results in \cite{Tim,Porter,Vain}. However, the proof of this conjecture is left for further studies.

\vspace{0.25cm}

{\bf Acknowledgements.} This work is part of the grant
PID-2021-122954NB-100 funded by MCIN/AEI/10.13039/501100011033 and ``ERDF A
way of making Europe''.
M. Guardia has been supported by the European Research Council (ERC) under the
European Union’s Horizon 2020 research and innovation programme (grant agreement
No. 757802). M. Guardia is also supported by the Catalan Institution for Research and
Advanced Studies via an ICREA Academia Prizes 2018 and 2023. 
This work is also supported by the Spanish State Research Agency, through the
Severo Ochoa and Mar\'ia de Maeztu Program for Centers and Units of Excellence in
R\&D (CEX2020-001084-M). 

The authors have no relevant financial or non-financial interests to disclose.

Data sharing not applicable to this article as no datasets were generated or analyzed during the current study.

\section{Details of the proof}\label{sec:outlineproof}

We devote this section to prove Theorem \ref{thm:main}. First in Section \ref{sec:homo} we provide analytic properties of the unperturbed solution  \eqref{def:soliton}. Then, in Section \ref{sec:outer:description} we study the analytic continuation of the perturbed solutions in suitable complex domains and we also analyze the auxiliary solution. In Section~\ref{sec:expsmallestimates} we give exponential upper bounds for the difference between two solutions for the stable and unstable invariant manifolds at a given transverse cross-section. To provide an asymptotic formula for this difference we analyze the first order of the perturbed solutions close to the singularities of the unperturbed solution. This is done in Section \ref{sec:innerscale} by means of an inner equation and complex matching techniques. Finally, in Section \ref{sec:asympformula} we obtain the asymptotic formula for the difference between two solutions for the stable and unstable invariant manifolds.

We will use the notation $'$ and $\partial_x $ to indicate the derivative with respect to $x$. In addition, when defining functional operators, we usually omit the dependence of some known functions such as $u_0$ on $x$.

\subsection{Properties of the unperturbed solution}\label{sec:homo}

The first step in the proof of Theorem \ref{thm:main} is to analyze the analytic properties of the unperturbed solution $u_0$ introduced in \eqref{def:soliton}. This is contained in the following lemma, the proof of which can be found in Appendix \ref{app-A}.

\begin{lemma}\label{lemma:propertiesu0}
For $\ga \in (-\frac{1}{9},0)$, the function $u_0$  in \eqref{def:soliton} has the following properties:
\begin{itemize}
\item At the line $\Im x= \pi$ $u_0$ has exactly two singularities at
\begin{equation}\label{def:alpha}
x_\pm=\pm \al+\pi i, \quad\al=\cosh^{-1}\frac{1}{\sqrt{1+9\ga}} 
\end{equation}
and at $\Im x= -\pi$ $u_0$ has singularities at the conjugate points $\ol{x_\pm}$
\item $u_0$ is real analytic in $ \mathbb{C} \backslash\{x_\pm+i2k\pi, \ol{x_\pm}-i2k\pi\}_{k\in \mathbb{N}}$. 
\item In a neighborhood of $x_\pm $, $u_0$ satisfies
\[
u_0(x) = \frac{c_{\pm 1}}{x-x_\pm} + \mathcal{O}(1) \quad \mbox{\rm as} \;\; x \to x_\pm,
\]
with
\begin{equation}\label{def:cpm1}
c_{\pm1}= \mp\frac{1}{\sqrt{|\gamma|}}.
\end{equation}
\item The second derivative of $u_0$ has exactly eight zeros, $x_j^{\pm}$, $j=1,2,3,4$ with $|\Im x_j^{\pm}| \leq \pi$ of the form 
$$
x_1^{\pm} = \pm i b, \qquad x_2^{\pm}= \pm a, \qquad x_3^{\pm}= \pm \mathbf{a}+ i \pi, \qquad x_4^{\pm} = \pm \mathbf{a}- i \pi
$$ 
with 
$$
b\in \left (\frac{\pi}{2}, \pi\right ), \qquad a > \alpha , \qquad \mathbf{a} \in (0,\alpha).
$$
\end{itemize}
\end{lemma}

\subsection{The outer scale}\label{sec:outer:description}

The second step in the proof of Theorem \ref{thm:main} is to look for parameterizations of the one-dimensional stable and unstable invariant manifolds in the system \eqref{eq:system}. We parameterize them as solutions of equation \eqref{eq:system} by fixing the initial condition at $\Sigma$ defined in  \eqref{def:section}.

We analyze the invariant manifolds by a perturbative approach close to $(u_0,0)$ where $u_0$ is the solution of~\eqref{eq:mainsecondorder}  introduced in~\eqref{def:soliton} that satisfies $u_0'(0) = 0$. To this end, we write 
\[
 u=u_0+\xi,\quad v=\eta,
\]
which yields the following system 
\begin{equation}\label{eq:perturb}
\left\{ \begin{array}{l} 
\LL_1\xi =\FF_1  [\xi,\eta], \\
\LL_2\eta =\FF_2  [\xi,\eta], 
\end{array} \right.
\end{equation}
where the linear operators are defined by
\begin{equation}\label{def:diffoperators}  
\left\{ \begin{array}{l} 
\LL_1 =-\pa_x^2+1-2u_0(x)-6\ga u_0^2(x),\\
\LL_2 =\pa_x^2+\frac{1}{\eps^2},
\end{array} \right.
\end{equation}
and
\begin{equation}\label{eq:outer:RHSoperator} 
\left\{ \begin{array}{l} 
\FF_1[\xi,\eta] =-\eta+(1+6\ga u_0) \xi^2 + 2 \ga \xi^3,\\
\FF_2[\xi,\eta] =f'(u_0+\xi)\left(u_0+\xi+\eta-f(u_0+\xi)\right)+f''(u_0+\xi)(u_0'+\xi')^2,
\end{array} \right.
\end{equation}
with $f$ defined in~\eqref{def:f}. 
Now, since
\[
\eta'=u'''-u'+f'(u)u',
\]
the first integral \eqref{def:firstintegral:2} becomes
\begin{align}\label{defGfirstintegral}
\wt G(\xi,\xi',\eta,\eta',x) =& \frac{1}{2} (1-\eps^2)\left[ (u_0')^2 + 2 u_0'\xi' + (\xi')^2 \right] - \frac{1}{2} \left[  u_0^2 - 2 u_0 \xi - \xi^2 \right] + F(u_0+\xi)\notag\\
&+\eps^2\Big[\left(u_0'+\xi'\right)\left(\eta'+u_0'+\xi'-f'(u_0+\xi)(u_0'+\xi')\right)\\
&-\frac{1}{2} (\eta+u_0+\xi-f(u_0+\xi))^2 \Big], \notag 
\end{align}
which is constant along solutions of \eqref{eq:perturb}.

The following theorem, whose  proof is given in Section~\ref{sec:outer}, provides two solutions of~\eqref{eq:perturb} which decay exponentially as $\Re x\to+ \infty$ and $\Re x\to- \infty$ respectively. They correspond to the parameterizations of the invariant manifolds. Moreover, we prove that they can be analytically extended to the so-called outer domains defined as 
\begin{equation}\label{def:defdomainouter} 
\begin{split}
D^{\out,\uns}_{\kappa}=&\,\left\{ x\in\CC: \quad  \left|\Ip(x)\right|\leq -\tan\theta\Rp(x-x_-)+ \Ip x_--\kappa\eps \right\},\\
D^{\out,\sta}_{\kappa} = & \,\left\{ x\in\CC: \quad \left|\Ip(x)\right|\leq \tan\theta\Rp(x-x_+) + \Ip x_+-\kappa\eps \right\},
\end{split}
\end{equation}
where $0<\theta<\mathrm{atan} \left (\frac{\pi}{3\alpha}\right )$, with $\alpha$ defined in~\eqref{def:alpha}, is a fixed angle independent of $\eps$ and $\kappa\geq 1$ (see Figure~\ref{fig:outer}).  Observe that $D^{\out,\us}_{\kappa}$, $\us=\uns,\sta$, reach domains at a $\kappa\eps$--distance of the singularities $x=x_-$ and $x=x_+$ of $u_0$ respectively.

\begin{figure}[hbt!]	
	\centering
	\begin{overpic}[width=10cm]{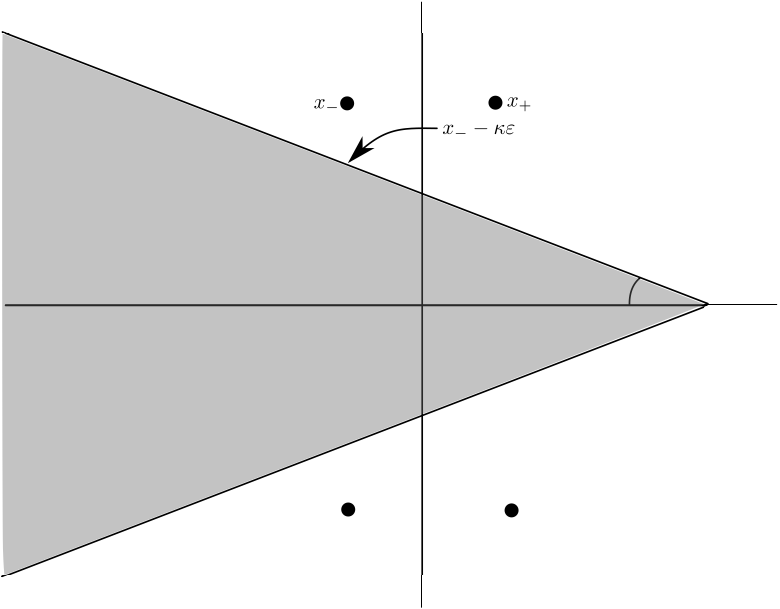}
		\end{overpic}
	\caption{The outer domain $D^{\out,\uns}_{\kappa}$ introduced in \eqref{def:defdomainouter}. }	
	\label{fig:outer}
\end{figure} 

\begin{theorem}\label{thm:outer:intro} Fix $0<\theta<\mathrm{atan} \left (\frac{\pi}{3\alpha}\right )$. 
 There exists $\kappa_0, \eps_0>0$, such that, if $\eps\in (0, \eps_0)$ and $\kappa> \kappa_0$, then there exist 
 real-analytic functions  $(\xi^\star,\eta^\star)$, $\star=\uns,\sta$, defined  in the domain $D^{\out,\us}_{\kappa}$ which  are solutions of \eqref{eq:perturb} satisfying 
\[
\begin{split}
\lim_{\Re x\to -\infty}(\xi^{\uns},\eta^{\uns})=(0,0),\qquad
\lim_{\Re x\to \infty}(\xi^{\sta},\eta^{\sta})=(0,0)
\end{split}
\]
 and
\[
\partial_x \xi^\star(0)=0,\qquad \wt G(\xi^\star,\partial_x \xi^\star,\eta^\star,\partial_x \eta^\star,x)=0,
\]
where $\wt G$ is the first integral introduced in~\eqref{defGfirstintegral}. 

Moreover, there exists $M_1>0$, depending only on $\theta,\kappa_0,\eps_0$, such that $\xi^\us$ and $\eta^\us$, $\us=\uns,\sta$, satisfy the following estimates.
\begin{itemize}
\item For $x\in D^{\out,\star}_{\kappa}\cap\{|\Rp(y)|\geq 2\alpha\}$, 
\[
 |\xi^\star(x)|\leq\, M_1\eps^2 e^{-|\Re x|},\qquad    |\eta^\star(x)|\leq \, M_1\eps^2 e^{-|\Re x|}
 \]
 and
\[
 |\pa_x\xi^\star(x)|\leq\, M_1\eps^2 e^{-|\Re x|},\qquad  |\pa_x\eta^\star(x)|\leq \, M_1\eps^2 e^{-|\Re x|}.
 \]
\item  For  $x\in D^{\out,\star}_{\kappa}\cap\{|\Rp(y)|\leq  2\alpha \}$,
\[
\begin{aligned}
  |\xi^\star(x)|&\leq\, \frac{M_1\eps^2}{|x-x_-|^3|x-\ol x_-|^3|x-x_+|^3|x-\ol x_+|^3},\\
  |\eta^\star(x)|&\leq \, \frac{M_1\eps^2}{|x-x_-|^5|x-\ol x_-|^5|x-x_+|^5|x-\ol x_+|^5},\\
  |\pa_x\xi^\star(x)|&\leq\, \frac{M_1\eps^2}{|x-x_-|^4|x-\ol x_-|^4|x-x_+|^4|x-\ol x_+|^4},\\
  |\pa_x\eta^\star(x)|&\leq \, \frac{M_1\eps}{|x-x_-|^5|x-\ol x_-|^5|x-x_+|^5|x-\ol x_+|^5}.
\end{aligned}
\]
\end{itemize}
Finally, 
\[
\xi^\sta(x)=\xi^\uns(-x), \qquad \eta^\sta(x) = \eta^\uns(-x)
\]   
or, in other words, the unstable curve is reflected by the involution $\Psi$ in~\eqref{def:involution} to the stable one. 
\end{theorem}

To prove Theorem \ref{thm:main}, we analyze the difference 
\begin{equation}\label{def:differenceDelta}
\Delta=(\Delta\xi,\Delta\eta)=(\xi^{\uns}-\xi^{\sta}, \eta^{\uns}-\eta^{\sta}).
\end{equation}
However, since its difference is exponentially small, to obtain an asymptotic formula, we would need to analyze this difference in $\eps$-neighborhoods of the singularities $x=x_\pm$. Note that Theorem~\ref{thm:outer:intro} does not provide the analytic continuation of $(\xi^{\sta}, \eta^{\sta})$ to points $\kappa\eps$-close to $x_-$ (and same happens for $(\xi^{\uns}, \eta^{\uns})$ and $x_+$). 

Instead of performing the analytic extension of the invariant manifolds in the $\kappa \eps$-neighborhood of the points $x_{\pm}$, we rely on auxiliary functions $(\xi^\aux, \eta^\aux)$. These functions will be solutions of the same equation \eqref{eq:perturb} and will also belong to the same energy level with respect to $\wt G$ as $(\xi^{\uns,\sta}, \partial_x \xi^{\uns,\sta}, \eta^{\uns,\sta}, \partial_x \eta^{\uns,\sta})$. Then, the analysis of the difference \eqref{def:differenceDelta} will be deduced by the differences 
\begin{equation}\label{def:differencesDeltaaux}
\begin{split}
\Delta^{\uns}&=(\Delta\xi^{\uns},\Delta\eta^{\uns})=(\xi^{\uns}-\xi^\aux, \eta^{\uns}-\eta^\aux),\\
\Delta^{\sta}&=(\Delta\xi^{\sta},\Delta\eta^{\sta})=(\xi^\aux-\xi^{\sta}, \eta^\aux-\eta^{\sta}).
\end{split}
\end{equation}

The following theorem, whose proof is given in Section \ref{sec:aux}, provides the existence of the functions $(\xi^\aux, \eta^\aux)$ in the domain
\begin{equation}\label{def:domainaux} 
\begin{split}
 D^\aux_\kk  =&\left\{ x\in\CC : \quad \left|\Ip(x)\right|\leq \tan\theta\Rp(x-x_-)+\pi-\kappa\eps \right\}\\
& \cap \left\{ x\in\CC : \quad  \left|\Ip(x)\right|\leq -\tan\theta\Rp(x-x_+)+\pi-\kappa\eps \right\}
\end{split}
\end{equation}
with $\kappa,\theta>0$. The domain is shown in Figure~\ref{fig:aux}.

\begin{figure}[htb!]	
	\centering
	\begin{overpic}[width=10cm]{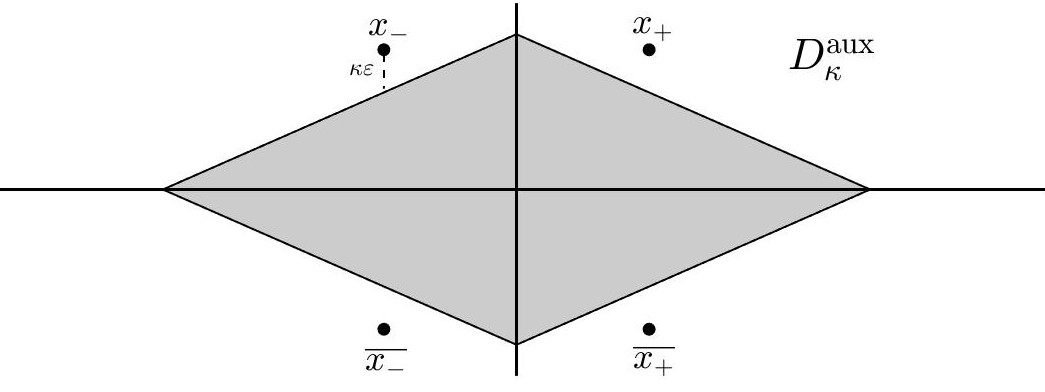}
	\end{overpic}
	\caption{The auxiliary domain $D^{\aux}_{\kappa}$ introduced in \eqref{def:domainaux}.}	
	\label{fig:aux}
\end{figure}

\begin{theorem}\label{thm:aux:intro} Let $0<\theta<\mathrm{arctan}\left (\frac{\pi}{\alpha} \right )$. 
 	There exists $\kappa_0, \eps_0>0$, such that, if $\eps\in (0, \eps_0)$ and $\kappa> \kappa_0$, then there exist 
 real-analytic functions  $(\xi^\aux,\eta^\aux)$ defined  in the domain $D^\aux_\kk$ which  are a solution of \eqref{eq:perturb} and satisfy
\[
\partial_x \xi^\aux(0)=0\qquad \text{and}\qquad \wt G(\xi^\aux,\partial \xi^\aux,\eta^\aux,\partial_x \eta^\aux,x)=0
\]
where $\wt G$ is the first integral introduced in~\eqref{defGfirstintegral}.

Moreover, there exists $M_2$, depending on $\theta,\kappa_0,\eps_0$ such that,  for $x\in D^\aux_\kk$,
\[
\begin{aligned}
  |\xi^\aux(x)|&\leq\, \frac{M_2\eps^2}{|x-x_-|^3|x-\ol x_-|^3|x-x_+|^3|x-\ol x_+|^3}\\
  |\eta^\aux(x)|&\leq \, \frac{M_2\eps^2}{|x-x_-|^5|x-\ol x_-|^5|x-x_+|^5|x-\ol x_+|^5}\\
  |\pa_x\xi^\aux(x)|&\leq\, \frac{M_2\eps^2}{|x-x_-|^4|x-\ol x_-|^4|x-x_+|^4|x-\ol x_+|^4}\\
  |\pa_x\eta^\aux(x)|&\leq \, \frac{M_2\eps}{|x-x_-|^5|x-\ol x_-|^5|x-x_+|^5|x-\ol x_+|^5}
\end{aligned}
\]
In addition $(\xi^\aux(x),\eta^\aux(x))= (\xi^\aux(-x),\eta^\aux(-x)$. 
\end{theorem}

\subsection{Exponentially small estimates}\label{sec:expsmallestimates}

The next step in the proof of Theorem  \ref{thm:main} is to 
analyze the differences $\Delta^{\uns}$,  $\Delta^{\sta}$ defined in \eqref{def:differencesDeltaaux}. Since $(\xi^\star,\eta^\star)$, $\star=\uns,\sta,\aux$ are all solutions of \eqref{eq:perturb}, we can conclude in the following lemma that the differences $\Delta^\star$ are solutions of a linear system in the following domains
\begin{equation}\label{def:domainsintersectionouter}
\begin{split}
E^{\out,\uns}_{\kappa}=&\,\left\{ x\in\CC : \quad  \left|\Ip(x)\right|\leq -\tan\theta\Rp(x-x_-)+ \Ip x_--\kappa\eps, \Re x\geq \Re x_- \right\},\\
E^{\out,\sta}_{\kappa} = & \,\left\{ x\in\CC : \quad  \left|\Ip(x)\right|\leq \tan\theta \Rp(x-x_+) + \Ip x_+-\kappa\eps, \Re x\leq \Re x_- \right\},
\end{split}
\end{equation}
(see Figure \ref{fig:intersection}).  Note that these domains, with $\theta$ such that 
$0 <\theta< \mathrm{atan} \left (\frac{\pi}{3\alpha}\right )$, satisfy $E^{\out,\star}_{\kappa}\subset D^{\out,\star}_{\kappa}\cap D^{\aux}_{\kappa}$, $\star=\uns,\sta$.

\begin{figure}[htb!]	
	\centering
	\begin{overpic}[width=10cm]{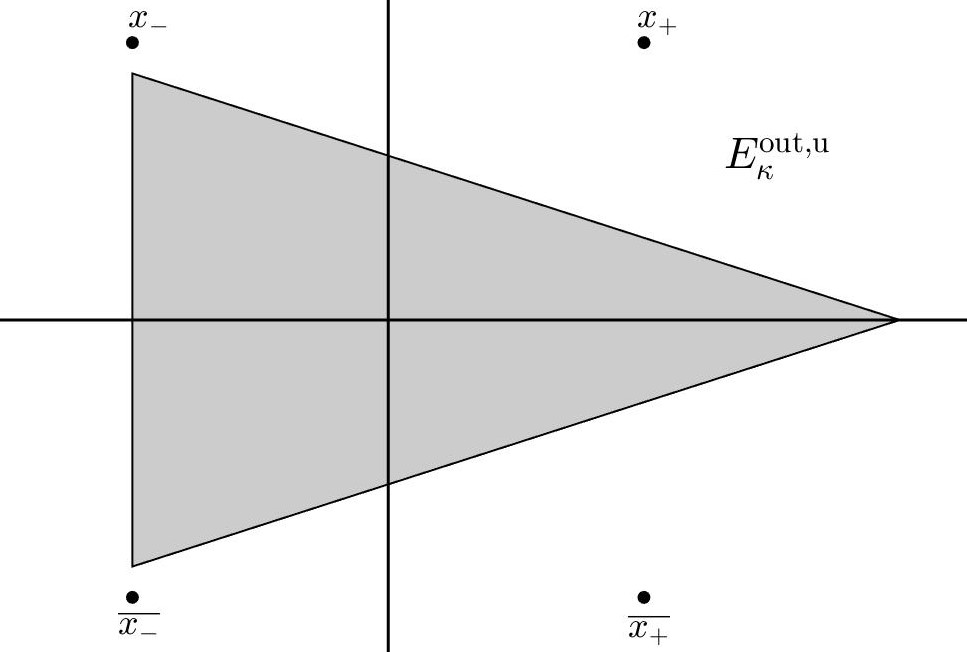}
		\end{overpic}
	\caption{The intersection domain $E^{\mathrm{out},\mathrm{u}}_{\kappa}$ introduced in \eqref{def:domainsintersectionouter}.}	
	\label{fig:intersection}
\end{figure}

\begin{lemma}\label{lemma:differenceequation}
The functions $\Delta^\star=(\Delta\xi^\star,\Delta\eta^\star)$, $\star=\uns,\sta$, in \eqref{def:differencesDeltaaux} are defined in the domains $E^{\out,\star}_{\kappa}$  in~\eqref{def:domainsintersectionouter} and are solutions of the linear system
\begin{equation}\label{eq:lineardiff}
\left\{ \begin{array}{l} 
 \LL_1\Delta\xi =\NNN_1[\Delta\xi,\Delta\eta], \\   
 \LL_2\Delta\eta =\NNN_2[\Delta\xi,\Delta\xi',\Delta\eta],
\end{array} \right.
\end{equation}
where 
\begin{equation}\label{eq:lineardiff:RHS}
\left\{ \begin{array}{l} 
\NNN_1[\Delta\xi,\Delta\eta](x) =-\Delta\eta(x)+a(x)\Delta\xi(x),\\   \NNN_2[\Delta\xi,\Delta\xi',\Delta\eta](x) = b(x)\Delta\xi(x)+c(x)\Delta\xi'+d(x)\Delta\eta(x),
\end{array} \right.
\end{equation}
for some functions $a$, $b$, $c$ and $d$, which  satisfy that, for $x\in E^{\out,\star}_{\kappa}$, 
\[
\begin{aligned}
  |a(x)|&\leq\, \frac{M_3\eps^2}{|x-x_-|^4|x-\ol x_-|^4|x-x_+|^4|x-\ol x_+|^4},\\
  |b(x)|&\leq \, \frac{M_3}{|x-x_-|^4|x-\ol x_-|^4|x-x_+|^4|x-\ol x_+|^4},\\
  |c(x)|&\leq\, \frac{M_3}{|x-x_-|^3|x-\ol x_-|^3|x-x_+|^3|x-\ol x_+|^3},\\
  |d(x)|&\leq \, \frac{M_3}{|x-x_-|^2|x-\ol x_-|^2|x-x_+|^2|x-\ol x_+|^2},
\end{aligned}
\]
for some constant $M_3$ independent of $\eps$ and $\kappa$.
\end{lemma}

To obtain the exponentially small estimates for the differences $\Delta^{\star}$ ($\star = \uns,\sta$), we use the existence of the first integral $\wt G(\xi,\xi',\eta,\eta',x)$. The first integral gives us an extra relation for the components of the difference $\Delta^{\star}$, which allows us to get rid of analyzing $\Delta\xi^{\star}$.

The following lemma is straightforward taking into account Lemma~\ref{lemma:propertiesu0} and  Theorems \ref{thm:outer:intro} and \ref{thm:aux:intro}.

\begin{lemma}\label{lemma:differencefirstintegral}
The functions $\Delta^\star=(\Delta\xi^\star,\Delta\eta^\star)$, $\star = \uns,\sta$, defined in \eqref{def:differencesDeltaaux} satisfy
\[
\left(-u''_0(x)+m(x)\right)\Delta\xi+\left(u'_0(x)+n(x)\right)\Delta\xi'+p(x)\Delta\eta+q(x)\Delta \eta'=0
\]
for some functions $m$, $n$, $p$ and $q$, which  satisfy that, for $x\in E^{\out,\star}_{\kappa}$,
\[
\begin{aligned}
  |m(x)|&\leq\, \frac{M_4\eps^2}{|x-x_-|^5|x-\ol x_-|^5|x-x_+|^5|x-\ol x_+|^5},\\
  |n(x)|&\leq \, \frac{M_4\eps^2}{|x-x_-|^4|x-\ol x_-|^4|x-x_+|^4|x-\ol x_+|^4},\\
  |p(x)|&\leq\, \frac{M_4\eps^2}{|x-x_-|^3|x-\ol x_-|^3|x-x_+|^3|x-\ol x_+|^3},\\
  |q(x)|&\leq \, \frac{M_4\eps^2}{|x-x_-|^2|x-\ol x_-|^2|x-x_+|^2|x-\ol x_+|^2},
\end{aligned}
\]
with $M_4>0$ a constant independent of $\eps$ and $\kappa$.
\end{lemma}

By using Lemma \ref{lemma:differencefirstintegral}, we reduce the system of two second-order equations \eqref{eq:lineardiff} to a third-order system imposed on $\Delta\zeta=\Delta \xi'$, $\Delta\eta$ and $\Delta\eta'$. The following lemma is obtained directly from Lemmas \ref{lemma:differenceequation} and \ref{lemma:differencefirstintegral}.

\begin{lemma}\label{lemma:differenceequation:2}
The functions  $\Delta\zeta^\star=\partial_x \Delta {\xi^\star}$, $\Delta\eta^\star$, $\star=\uns,\sta$,  are defined in $E^{\out,\star}$ in~\eqref{def:domainsintersectionouter} and are solutions of the linear equation
\begin{equation}\label{eq:lineardiff:2}
\left\{ \begin{array}{l} 
\wh \LL_1\Delta\zeta =\wh\NNN_1[\Delta\zeta,\Delta\eta, \Delta\eta'],\\   
 \LL_2\Delta\eta =\wh\NNN_2[\Delta\zeta,\Delta\eta,\Delta\eta'],
\end{array} \right.
\end{equation}
where 
\begin{equation}\label{def:L1hat}
\wh \LL_1=-\pa_x+\frac{u_0'''}{u_0''},
\end{equation}
and
\begin{equation*}
\left\{ \begin{array}{l} 
\wh\NNN_1[\Delta\zeta,\Delta\eta,\Delta\eta'] =-\Delta\eta+\wh r(x)\Delta\zeta+\wh s(x)\Delta\eta+\wh t(x)\Delta\eta', \vspace{0.2cm}\\   \wh\NNN_2[\Delta\zeta,\Delta\eta,\Delta\eta'] =\wh c(x)\Delta\zeta+\wh d(x)\Delta\eta+\wh e(x)\Delta\eta',
\end{array} \right.
\end{equation*}
for some functions $\wh r$, $\wh s$, $\wh t$, $\wh c$, $\wh d$ and $\wh e$, which  satisfy that, for $x\in E^{\out,\star}_{\kappa}$, 
\[
\begin{aligned}
|\wh r(x)|&\leq\, \frac{M_5\eps^2}{|x-x_-|^3|x-\ol x_-|^3|x-x_+|^3|x-\ol x_+|^3},\\
 |\wh s(x)|&\leq\, \frac{M_5\eps^2}{|x-x_-|^2|x-\ol x_-|^2|x-x_+|^2|x-\ol x_+|^2},\\
 |\wh t(x)|&\leq\, \frac{M_5\eps^2}{|x-x_-||x-\ol x_-||x-x_+||x-\ol x_+|},\\
  |\wh c(x)|&\leq\, \frac{M_5}{|x-x_-|^3|x-\ol x_-|^3|x-x_+|^3|x-\ol x_+|^3},\\
  |\wh d(x)|&\leq \, \frac{M_5}{|x-x_-|^2|x-\ol x_-|^2|x-x_+|^2|x-\ol x_+|^2},\\
  |\wh e(x)|&\leq \, \frac{M_5\eps^2}{|x-x_-|^3|x-\ol x_-|^3|x-x_+|^3|x-\ol x_+|^3},
\end{aligned}
\]
with $M_5$ a constant independent of $\eps$ and $\kappa$.
\end{lemma}  

By using Lemma \ref{lemma:differenceequation:2}, we provide an asymptotic formula for $\Delta^\star$ at $x=0$. Note that, by Theorem \ref{thm:outer:intro} and \ref{thm:aux:intro}, $\Delta\zeta^\star(0)= \partial_x \Delta{\xi}^\star(0)=0$ (and that $\Delta \xi(0)$ can be obtained by Lemma \ref{lemma:differencefirstintegral} once the other components are known). Therefore, in order to prove Theorem \ref{thm:main}, it is sufficient to look for an asymptotic formula for $\Delta\eta^\star(0)$ and $\partial_x \Delta{\eta^\star}(0)$.

Assume for a moment that $\Delta\eta^\star$ satisfy
\[
\LL_2\Delta\eta=0
\]
(that is, assume that $\wh c= \wh d=\wh e=0)$. Then, $\Delta\eta^\star$ would be of the form 
\begin{equation}\label{pre_exponentially_small}
\Delta\eta^\star(x)=C_1^\star e^{\frac{ix}{\eps}}+C_2^\star e^{-\frac{ix}{\eps}}.
\end{equation}
We introduce 
\begin{equation}\label{def:rho-}
\rminus = x_--i\kappa\eps\quad \text{ and }\quad {\rplus}=x_+-i\kappa\eps
\end{equation}
with $x_{\pm}= \pm \alpha + \pi i$ and $\alpha$ defined in Lemma~\ref{lemma:propertiesu0}. We observe that, by Theorems \ref{thm:outer:intro} and Theorem \ref{thm:aux:intro},  $\Delta \eta^\uns$ is defined at $\rminus, \ol{\rminus}$ and $\Delta \eta^\sta$ is defined at $\rplus$, $\ol{\rplus}$. Evaluating $\Delta \eta^\uns$ in~\eqref{pre_exponentially_small} at $x=\rminus$ and 
$x=\ol{\rminus}$, using that $e^{\frac{i\rminus}{\eps}}$ and $e^{-\frac{i\ol{\rminus}}{\eps}}$ are of size $e^{-\frac{\pi}{\eps}}$, one obtains that $C_1^{\uns}$ and $C_2^{\uns}$ must satisfy
\begin{equation}\label{formulaC12us}
C_1^{\uns}=\Delta\eta^{\uns}(\ol{\rminus})e^{-\frac{i\ol{\rminus}}{\eps}}+\text{h.o.t.} \qquad \text{and}\qquad C_2^{\uns}=\Delta\eta^{\uns}(\rminus)e^{\frac{i{\rminus}}{\eps}}+\text{h.o.t.}.
\end{equation}
An analogous formula follows for $C_{1,2}^{\sta}$ changing $\rminus$ by $\rplus$. 

Now, the equation for $\Delta\eta^\star$, $\star=\uns,\sta$, in \eqref{eq:lineardiff} has a right hand side \eqref{eq:lineardiff:RHS} 
with nonzero $\wh c$, $\wh d$, $\wh e$ and therefore one has to proceed more carefully than in the arguments above. The following proposition gives the needed result.

\begin{proposition}\label{prop:difference:remainder}
The functions $\Delta\eta^\star$, $\star=\uns,\sta$, introduced in \eqref{def:differencesDeltaaux} are defined in  $E^{\out,\star}$ given by \eqref{def:domainsintersectionouter} and are of the form 
\begin{equation}\label{eq:formulaDeltaEta}
\Delta\eta^\star(x)=C_1^\star e^{\frac{ix}{\eps}}+C_2^\star e^{-\frac{ix}{\eps}}+\RRR^\star(x)
\end{equation}
where 
\begin{itemize}
\item The constants $C_1^\star$ and $C_2^\star$ satisfy 
\begin{equation}\label{eq:linearsystemCs}
\begin{split}
\Delta\eta^{\uns}(\rminus)&=C_1^{\uns} e^{\frac{i\rminus}{\eps}}+C_2^{\uns} e^{-\frac{i\rminus}{\eps}}\\
\Delta\eta^{\uns}(\ol{\rminus})&=C_1^{\uns} e^{\frac{i\ol{\rminus}}{\eps}}+C_2^{\uns} e^{-\frac{i\ol{\rminus}}{\eps}}\\
\Delta\eta^{\sta}(\rplus)&=C_1^{\sta} e^{\frac{i\rplus}{\eps}}+C_2^{\sta} e^{-\frac{i\rplus}{\eps}}\\
\Delta\eta^{\sta}(\ol{\rplus})&=C_1^{\sta} e^{\frac{i\ol{\rplus}}{\eps}}+C_2^{\sta} e^{-\frac{i\ol{\rplus}}{\eps}}.
\end{split}
\end{equation}
\item The functions $\RRR^\star$ satisfy that
\begin{equation}\label{def:R:zeros}
\RRR^{\uns}(\rminus)=0, \qquad \RRR^{\uns}(\ol{\rminus})=0, \qquad \RRR^{\sta}(\rplus)=0,\qquad \RRR^{\sta}(\ol{\rplus})=0,
\end{equation}
and that, for $x\in E^{\out,\star}_{\kappa}$,
\begin{equation}\label{def:R:estimates}
\begin{split}
\left|\RRR^\star(x)\right|&\leq\frac{M_6}{\kappa}e^{\frac{1}{\eps}|\Im x|}\left(|C_1^{\uns}|+|C_2^{\uns}|\right)\\
\left|\pa_x\RRR^\star(x)\right|&\leq \frac{M_6}{\eps\kappa}e^{\frac{1}{\eps}|\Im x|}\left(|C_1^{\uns}|+|C_2^{\uns}|\right),
\end{split}
\end{equation}
for some constant independent $M_6>0$ independent of $\eps$ and $\kappa$.
\end{itemize}
\end{proposition}

Note that the properties of $C_j^\star$ are a direct consequence of evaluating~\eqref{eq:formulaDeltaEta} at $x=\rho^\pm$ and $x=\ol{\rho^\pm}$ and the properties of $\RRR^\star$. That is, to prove Proposition~\ref{prop:difference:remainder} boils down to prove the properties stated for the functions $\RRR^\star$. This is done in Section \ref{sec:difference}.

By Proposition~\ref{prop:difference:remainder}, proceeding as for \eqref{pre_exponentially_small}, we have that indeed, 
$C_{1,2}^{\uns}$ is of the form in~\eqref{formulaC12us} and analogous formula are also true for $C_{1,2}^{\sta}$. As a consequence, of this analysis and using also that, by Theorems~\ref{thm:outer:intro} and~\ref{thm:aux:intro}
$$
|\Delta \eta^{*} (\rho_\pm)| ,\, |\Delta \eta^* (\ol{\rho_{\pm}}) |\leq M \frac{1}{\kappa^5\eps^3} , 
$$
we have that 
$$
|C_{1,2}^\star |\leq M \frac{1}{\eps^{3}} e^{-\frac{\pi}{\eps}}. 
$$
However, in order to prove the asymptotic formula in Theorem~\ref{thm:main}, we need to perform a more accurate analysis of the functions $\eta^\star$ (and $\xi^\star$) around the points $\rho_\pm$ and $\ol{\rho_\pm}$. This is done in the following subsections by means of the inner equation (Theorem~\ref{thm:inner}) and complex matching techniques (Theorem~\ref{thm:matching:intro}).

\subsection{The inner scale}\label{sec:innerscale}

We perform the change of coordinates to the inner variables. We consider the new variables 
\begin{equation}\label{def:innervariable}
z=\eps^{-1}(x-x_{\pm})
\end{equation}
and, recalling the definition of $c_{\pm1}$ in \eqref{def:cpm1}, we define the functions
\begin{equation}\label{def:innerfunctions}
 \phi(z)=\frac{\eps}{c_{\pm 1}}\xi(x_\pm+\eps z),\qquad\psi(z)=\frac{\eps^3}{c_{\pm 1}}\eta(x_\pm+\eps z).
 \end{equation}
Recall that $\ga<0$ and therefore $c_{\pm1}^2\ga=-1$.
Applying the change of coordinates to equation \eqref{eq:perturb} and letting $\eps\to 0$ we obtain the  limiting inner equation,  
\begin{equation}\label{eq:inner}
\left\{ \begin{array}{l} 
\LL_1^\inn\phi =\Nin_1[\phi,\psi], \\
\LL_2^\inn\psi =\Nin_2[\phi,\psi],  
\end{array} \right.
\end{equation}
with
\begin{equation}\label{def:diffoperators:inner}
\left\{ \begin{array}{l} 
\LL_1^\inn =-\pa_z^2+\frac{6}{z^2},\\
\LL_2^\inn =\pa_z^2+1,
\end{array} \right.
\end{equation}
and
\begin{equation}\label{eq:inner:RHSoperator} 
\left\{ \begin{array}{l} 
\Nin_1[\phi,\psi] =-\psi-\frac{6}{z}\phi^2-2\phi^3,\\
\Nin_2[\phi,\psi] =-6\left(\frac{1}{z}+\phi\right)^2\left(\psi+2\left(\frac{1}{z}+\phi\right)^3\right) -12\left(\frac{1}{z}+\phi\right)\left(-\frac{1}{z^2}+\pa_z\phi\right)^2.
\end{array} \right.
\end{equation}
This equation is reversible with respect to the symmetry
\begin{equation}\label{def:symmetryinner}
(\phi,\psi) \to (-\phi,-\psi),\quad z\to -z.
\end{equation}
We analyze this equation in the \emph{inner domains} (see Figure~\ref{fig:inner})
\begin{equation}
\label{innerdomainsol}
\begin{array}{l}
\cD^{\uns,\inn}_{\theta,\kappa}=\{z\in\CC : \quad |\Im(z)|> \tan\theta \Re(z)+\kappa \},\vspace{0.2cm}\\
\cD^{\sta,\inn}_{\theta,\kappa}=\{z\in \CC : \quad -z\in \cD^{\uns,\inn}_{\theta,\kappa} \},
\end{array}
\end{equation}
for  $0<\theta<\pi/2$ and $\kappa>0$.

\begin{figure}[htb!]	
	\centering
	\begin{overpic}[width=10cm]{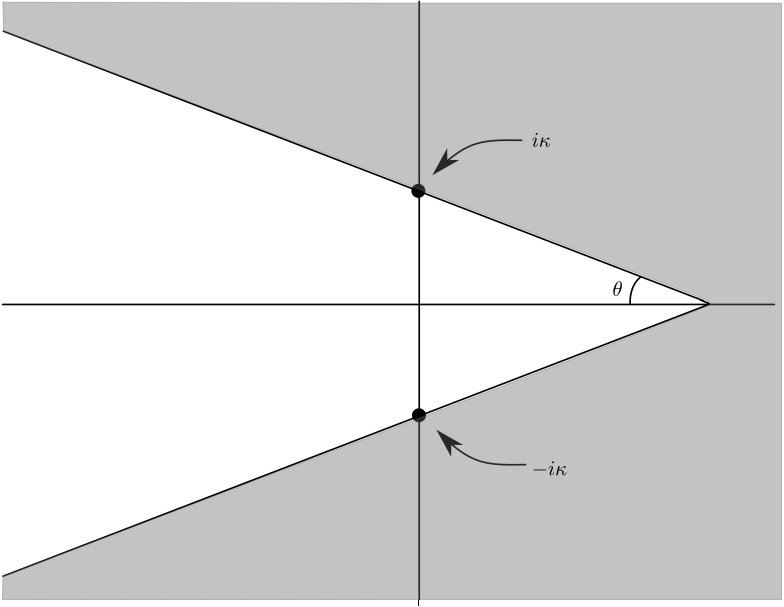}
		\end{overpic}
	\caption{The inner domain $D^{\uns,\inn}_{\theta,\kappa}$ introduced in \eqref{innerdomainsol}. }	
	\label{fig:inner}
\end{figure} 

The following theorem, which is proved in Section \ref{sec:inner}, provides an asymptotic formula for the difference between the two solutions of the inner equation.

\begin{theorem}\label{thm:inner}
Let $0<\theta<\frac{\pi}{2}$ be fixed. There exists $\kappa_0\geq 1$ big enough such that, for each $\kappa\geq\kappa_0$,
	\begin{enumerate} 
		\item\label{firstitem:thm:inner} Equation \eqref{eq:inner} has two real-analytic solutions $(\phi^{0,\star}, \psi^{0,\star}): \cD^{\star,\inn}_{\theta,\kappa}\rightarrow \CC^2$, $\star=\uns,\sta$, which, for every $z\in \cD^{\star,\inn}_{\theta,\kappa}$, satisfy
		\[
		\left|\phi^{0,\star}(z) \right|\leq \dfrac{M_7}{|z|^3}, \quad \left| \psi^{0,\star} (z)\right|\leq \dfrac{M_7}{|z|^5},  \]
   for some $M_7>0$ independent of $\kappa$.
   Moreover, they satisfy that, for $z\in \cD^{\uns,\inn}_{\theta,\kappa}$,
   \begin{equation}\label{def:symmetryinnersol}
(\phi^{0,\uns}(z),\psi^{0,\uns}(z))=(-\phi^{0,\sta}(-z),-\psi^{0,\sta}(-z)).
   \end{equation}
		\item\label{seconditem:thm:inner} The differences $\Delta\phi^0(z)= \phi^{0,\uns}(z)-\phi^{0,\sta}(z)$, $\Delta\psi^0(z)= \psi^{0,\uns}(z)-\psi^{0,\sta}(z)$ are given by 
		\begin{equation}\label{diffinnersol}
  \begin{split}
  \Delta\phi^0(z)&=\Theta e^{-i z}\left(-1+\chi_1(z)\right)\\
  \Delta\psi^0(z)&=\Theta e^{-i z}\left(1+ \chi_2(z)\right)\\
  \pa_z\Delta\phi^0(z)&=-i\Theta e^{-i z}\left(-1+\wh\chi_1(z)\right)\\
  \pa_z\Delta\psi^0(z)&=-i\Theta e^{-i z}\left(1+ \wh\chi_2(z)\right)
  \end{split}
  \end{equation}
  for $z\in \mathcal{R}^{\inn}_{\theta,\kappa}= \cD^{\uns,\inn}_{\theta,\kappa}\cap \cD^{\sta,\inn}_{\theta,\kappa}\cap \{z: i\RR, \Im z<0\}$,
		where $\Theta\in\RR$ is a constant, and $\chi_1$, $\chi_2$, $\wh\chi_1$, $\wh\chi_2$ are  analytic in $z$ and satisfy  that, for $z\in \mathcal{R}^{\inn}_{\theta,\kappa}$,
		\[
		|\chi_1(z)|\leq \dfrac{M_8}{|z|}, \quad |\chi_2(z)|\leq \dfrac{M_8}{|z|},\quad |\wh\chi_1(z)|\leq \dfrac{M_8}{|z|} \quad |\wh\chi_2(z)|\leq \dfrac{M_8}{|z|},\]
  for some $M_8>0$ independent of $\kk$.
 \item The constant $\Theta$ satisfies $\Theta\neq0$ if and only if there exists $z_0\in \mathcal{R}^{\inn}_{\theta,\kappa}$ such that $\Delta\phi^0(z_0)\neq 0$.
	\end{enumerate}
\end{theorem}

Theorem \ref{thm:inner} does not ensure that the first-order constant $\Theta$ is non-zero. This is stated in the next proposition, whose proof is deferred to Appendix \ref{app:nonvanishingStokes}.

\begin{proposition}\label{prop:nonvanishing}
The constant $\Theta\in\RR$ introduced  in Theorem \ref{thm:inner} satisfies
$\Theta\neq 0$.
\end{proposition}

Once we have obtained the solutions of the inner equation and analyzed their difference, the next step is to ``measure'' how well they approximate the functions obtained in Theorems \ref{thm:outer:intro} and \ref{thm:aux:intro}. This is done through what is usually called \emph{complex matching} techniques.

We first define the \emph{matching domains} where these differences are analyzed.
Let $0<\nu <1$ and  $0<\theta_2 < \theta <\theta_1<\frac{\pi}{2}$,  where $\theta$ is the angle introduced in \eqref{def:defdomainouter}. We denote
$$
\rminus= -i\kappa \eps + x_-, \qquad 
x_1^-=-i\kappa \eps - \eps^{\nu} e^{i \theta_1}+ x_-, \qquad x_2^- = -i\kappa \eps+ \eps^{\nu} e^{i\theta_2} +x_-.
$$
and 
$$
\rplus=-i\kappa \eps + x_+. \qquad x_1^+=-i\kappa \eps + \eps^{\nu} e^{-i \theta_1}+ x_+, \qquad x_2^+ = -i\kappa \eps- \eps^{\nu} e^{-i\theta_2} +x_+.
$$
Notice that $\rplus=-\overline{\rminus}$, $x_1^+=-\overline{x_1^-}$, $x_2^+ = -\overline{x_2^-}$, where we have denoted by 
$\overline{z}$ the complex conjugate of $z$. We define the  matching domains as 
\begin{equation}\label{def:domainsmatching}
D_{\theta_1,\theta_2, \nu}^{-, \match} = \widehat{\rminus,\, x_1^-\, ,x_2^-}, \qquad -D_{\theta_1,\theta_2,\nu}^{+, \match} = \widehat{\rplus,\, x_1^+\, ,x_2^+}
\end{equation}
that is, $D_{\theta_1,\theta_2,\nu}^{-,\match}$ 
as the triangle with vertexs $\rminus, x_1^-,x_2^-$ while $D_{\theta_1,\theta_2,\nu}^{-,\match}$ is the triangle with vertexs $-\rplus, x_1^+, x_2^+$ (see Figure~\ref{fig:matching}).

\begin{figure}[htb!]	
	\centering
	\begin{overpic}[width=10cm]{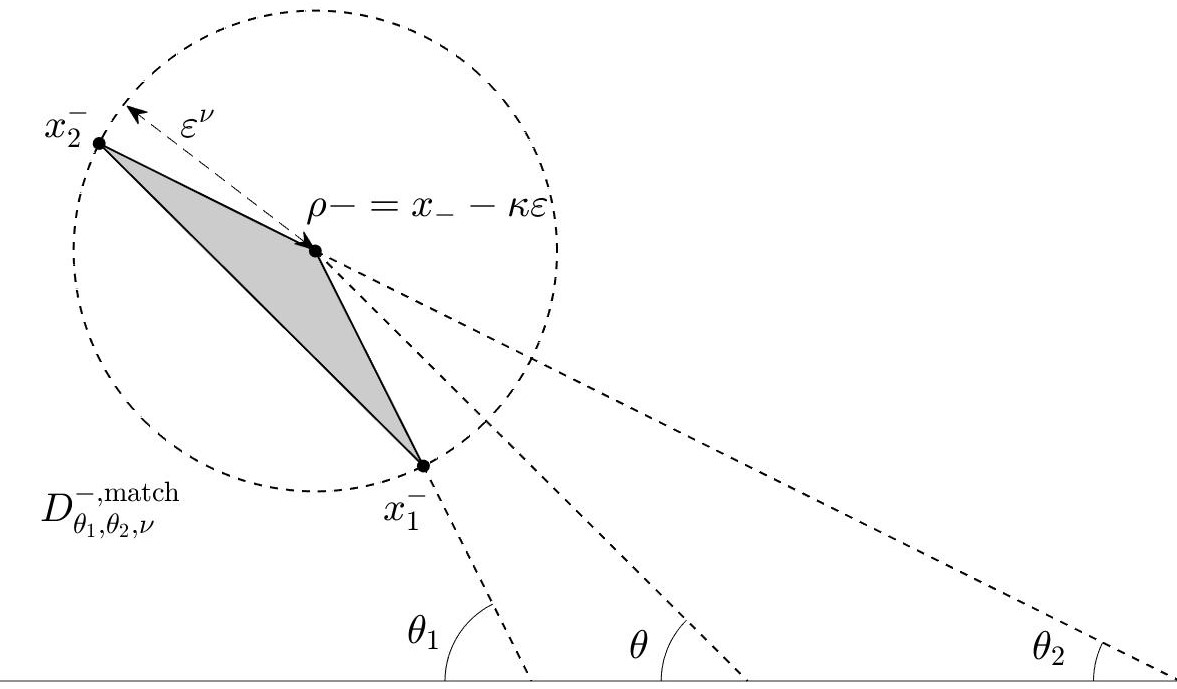}
		\end{overpic}
	\caption{The matching domain $D^{-,\match}_{\theta_1,\theta_2,\nu}$ introduced in \eqref{def:domainsmatching}. }	
	\label{fig:matching}
\end{figure} 

We also introduce
\begin{equation}\label{defxietaunspm}
\begin{aligned}
&\xi^{0,\uns}_-(x) =   \frac{c_{-1}}{\eps }\phi^{0,\uns}\big (\eps^{-1} (x-x_-) \big ), &\eta^{0,\uns}_-(x)=\frac{c_{-1}}{\eps^3} \psi^{0,\uns}\big (\eps^{-1} (x-x_-)\big ), \\
&\xi^{0,\sta}_+(x) = \frac{c_{1}}{\eps }\phi^{0,\sta}\big (\eps^{-1} (x-x_+ )\big ), &\eta^{0,\sta}_+ (x)=\frac{c_{1}}{\eps^3 }\psi^{0,\sta}\big (\eps^{-1} (x-x_+ )\big )
\end{aligned}
\end{equation}
and 
\begin{equation}\label{defxietaunspmaux}
\begin{aligned}
&\xi^{0,\aux}_-(x) = \frac{c_{-1}}{\eps } \phi^{0,\sta}\big (\eps^{-1} (x-x_-) \big ), &\eta^{0,\aux}_-(x)=\frac{c_{-1}}{\eps^3 } \psi^{0,\sta}\big (\eps^{-1} (x-x_- )\big ), \\
&\xi^{0,\aux}_+(x) = \frac{c_{1}}{\eps }\phi^{0,\uns}\big (\eps^{-1} (x-x_+ ) \big ), &\eta^{0,\aux}_+ (x)=\frac{c_{1}}{\eps^3 } \psi^{0,\uns}\big (\eps^{-1} (x-x_+  )\big ).
\end{aligned}
\end{equation}

The following theorem, which is proved in Section~\ref{sec:matching},
provides estimates between $(\xi^{0,*}_\pm,\eta^{0,*}_\pm)$ and $(\xi^*,\eta^*)$ with $\us=\uns,\sta,\aux$ in the corresponding matching domains.

\begin{theorem}\label{thm:matching:intro}
Let $\theta>0, \kappa_0$ be fixed as in Theorems~\ref{thm:inner}, \ref{thm:aux:intro} and $\theta$ as in Theorem~~\ref{thm:outer:intro}.
Take $0<\theta_2<\theta <\theta_1< \mathrm{atan}\left (\frac{\pi}{3\alpha}\right )$ and $\nu\in (0,1)$. 
 
We introduce the functions
\begin{align*}
    &\big (\delta \xi^{\uns}_-, \delta \eta^{\uns}_- \big )= \big (\xi^{\uns}- \xi^{0,\uns}_-, \eta^{\uns} - \eta^{0,\uns}_-\big ),  \\
 &\big (\delta \xi^{\sta}_+ , \delta \eta^{\sta}_+ \big )= \big (\xi^{\sta}- \xi^{0,\sta}_+, \eta^{\sta} - \eta^{0,\sta}_+\big ), \\ 
    & \big (\delta \xi_ \pm ^{\aux}, \delta \eta_\pm ^{\aux} \big )= \big (\xi^\aux - \xi^{0,\aux}_\pm, \eta^\aux - \eta^{0,\aux}_\pm \big ).
\end{align*}
Then there exist $\kappa_1\geq \kappa_0$ and a constant $M_9>0$ such that for all $\kappa \geq \kappa_1$ and $x \in D_{\theta_1,\theta_2, \nu}^{\pm, \match}$
\begin{align*}
& \big |\delta \xi^{\uns}_- (x)\big | ,\, 
\big | \delta \xi^{\sta}_+  (x)\big |, 
\big |\delta \xi^\aux _ \pm  (x) \big | \leq   M_9 |\log \eps|\frac{\eps^{2-\nu} }{|x-x_\pm|^2}, \\ & \big | \partial_ x \delta \xi^{\uns}_- (x)\big | ,\, 
\big | \partial_ x \delta \xi^{\sta}_+(x) \big |,\,
\big | \partial_ x\delta \xi^\aux _\pm (x) \big | \leq   M_9 |\log \eps|\frac{\eps^{2-\nu}  }{|x-x_\pm|^3} \\
& \big |\delta \eta^{\uns}_- (x)\big | ,\, 
\big |\delta \eta^{\sta}_+(x)(x) \big |,\,
\big |\delta \eta^\aux _\pm (x) \big | \leq   M_9 |\log \eps|\frac{\eps^{2-\nu}}{|x-x_\pm|^4}, \\ & \big | \partial_ x \delta \eta^{\uns}_- (x)\big | ,\, 
\big |\partial_x \delta \eta^{\sta}_+(x) \big |,\, 
\big | \partial_ x\delta \eta^\aux _\pm (x) \big | \leq   M_9 |\log \eps|\frac{\eps^{1-\nu}}{|x-x_\pm|^4}.
\end{align*}
\end{theorem}

\subsection{The asymptotic formula}\label{sec:asympformula}

Now, to prove Theorem \ref{thm:main} it only remains to provide an asymptotic formula for the constants $C_1^\star$ and $C_2^\star$. This is done in the following proposition, which is proved in Section \ref{sec-proof-prop}. From now on 
we take 
\begin{equation}\label{def:kappachoice}
\kappa=c|\log\eps|
\end{equation}
for some suitable constant $c>0$ to be chosen later. 

\begin{proposition}\label{prop:CsFormula}
The constants $C_1^\star$ and $C_2^\star$ introduced in Proposition \ref{prop:difference:remainder} satisfy
\begin{equation*}
\begin{split}
C_1^{\uns}&= \frac{1}{\sqrt{|\gamma|}\eps^3}e^{-\frac{i\ol{x_-}}{\eps}}\left(\Theta +\OO\left(\frac{1}{|\log\eps|}\right)\right)\\
C_2^{\uns}&=\frac{1}{\sqrt{|\gamma|}\eps^3}e^{\frac{ i {x_-}}{\eps}}\left( \Theta+\OO\left(\frac{1}{|\log\eps|}\right)\right) \\
C_1^{\sta}&= -\frac{1}{\sqrt{|\gamma|}\eps^3}e^{-\frac{i\ol{x_+}}{\eps}}\left(\Theta +\OO\left(\frac{1}{|\log\eps|}\right)\right)\\
C_2^{\sta}&=-\frac{1}{\sqrt{|\gamma|}\eps^3}e^{\frac{i{x_+}}{\eps}}\left(\Theta+\OO\left(\frac{1}{|\log\eps|}\right)\right).
\end{split}
\end{equation*}
\end{proposition}

Evaluating at $x=0$ the formula for $\Delta^\star$ in~\eqref{eq:formulaDeltaEta} together with Propositions~\ref{prop:difference:remainder} and~\ref{prop:CsFormula} lead to the asymptotic formulas
\[
\begin{split}
\Delta\eta^{\uns}(0)&=\frac{1}{\sqrt{|\gamma|}\eps^3}e^{-\frac{\pi}{\eps}}\left(2\Theta\cos\left(\frac{\alpha}{\eps}\right)  +\OO\left(\frac{1}{|\log\eps|}\right)\right)\\
\pa_x\Delta\eta^{\uns}(0)&=\frac{1}{\sqrt{|\gamma|}\eps^4}e^{-\frac{\pi}{\eps}}\left( -2\Theta\sin\left(\frac{\alpha}{\eps}\right)    +\OO\left(\frac{1}{|\log\eps|}\right)\right)\\
\Delta\eta^{\sta}(0)&=-\frac{1}{\sqrt{|\gamma|}\eps^3}e^{-\frac{\pi}{\eps}} \left(2\Theta\cos\left(\frac{\alpha}{\eps}\right)   +\OO\left(\frac{1}{|\log\eps|}\right)\right)\\
\pa_x\Delta\eta^{\sta}(0)&= -\frac{1}{\sqrt{|\gamma|}\eps^4}e^{-\frac{\pi}{\eps}} \left( 2\Theta\sin\left(\frac{\alpha}{\eps}\right)   +\OO\left(\frac{1}{|\log\eps|}\right)\right),
\end{split}
\]
where $\alpha$ is the constant introduced in \eqref{def:alpha}. 

To complete the proof of Theorem~\ref{thm:main} we 
recall that $\Delta \eta =\Delta\eta^{\uns}+\Delta\eta^{\sta}$ and that by the symmetry properties in Theorem~\ref{thm:outer:intro} and~\ref{thm:aux:intro} of $\eta^{\uns},\eta^\sta, \eta^\aux$ one has that, for $x\in D^\aux_\kappa \cap \mathbb{R}$ 
\[
\Delta \eta^{\uns}(x) = \eta^\uns (x) - \eta^\aux(x)= \eta^{\sta}(-x)- \eta^\aux(-x) = -\Delta \eta^\sta(-x)
\]
and therefore $\Delta \eta^\uns (0)= - \Delta\eta^\sta(0)$.  
This completes the proof of Theorem \ref{thm:main}. 

\begin{remark}
Notice that we could argue by symmetry that $\Delta \eta^\sta (x) = - \Delta \eta^\uns(-x)$ and skip the constants $C_{1,2}^\sta$ of our analysis. However we have preferred to keep all constants in order to emphasize that the method does not depend on the symmetries of the system. 
\end{remark}

\subsection{Proof of Proposition \ref{prop:CsFormula}}  
\label{sec-proof-prop}

To prove Proposition \ref{prop:CsFormula}, the first step is  to provide an asymptotic formula for $\Delta\eta^{\uns}(\rminus), \Delta\eta^{\uns}(\ol{\rminus})$ and  $\Delta\eta^{\sta}(\rplus), \Delta\eta^{\sta}(\ol{\rplus})$. 

\begin{lemma}\label{lem:delta:eta:uns} Let $\nu \in (0,1)$ and 
consider the points $x=\rminus$ and  $x=\ol{\rminus}$ introduced in~\eqref{def:rho-} with $\kappa$ as in \eqref{def:kappachoice} and $c\in (0, 1-\nu)$ . 

Then, the functions $\Delta\eta^{\uns}, \Delta \eta^\sta$ in~\eqref{def:differencesDeltaaux} satisfy
\[  
\begin{split}
\Delta\eta^{\uns}(\rminus)&= \frac{c_{-1}}{\eps^3}
{e^{-\kappa}}
\left(
\Theta  +\OO\left(\frac{1}{|\log\eps|}\right)\right)\\
\Delta\eta^{\uns}(\ol{\rminus})&=\frac{{c_{-1}}}{\eps^3} 
{e^{-\kappa}}
\left(
{ \Theta}
+\OO\left(\frac{1}{|\log\eps|}\right)\right),   
\end{split}
\]
and  
\[  
\begin{split}
\Delta\eta^{\sta}(\rplus)&= \frac{c_{+1}}{\eps^3}
{e^{-\kappa}}
\left(
\Theta  +\OO\left(\frac{1}{|\log\eps|}\right)\right)\\
\Delta\eta^{\sta}(\ol{\rplus})&=\frac{{c_{+1}}}{\eps^3} 
{e^{-\kappa}}
\left(
{\Theta}
+\OO\left(\frac{1}{|\log\eps|}\right)\right),   
\end{split}
\]
where $c_{\pm 1}$ and $\Theta$ are the constants introduced in~\eqref{def:cpm1} and Theorem~\ref{thm:inner} respectively.
\end{lemma}

\begin{proof}
We provide the proof for $\Delta\eta^{\uns}(\rminus)$. The other formula can be proven analogously. Note that $\Delta\eta^{\uns}$ can be written as 
\begin{align*}
\Delta\eta^{\uns}(x) & = 
  \eta^\uns(x)- \eta^{0,\uns}_-(x) + \eta^{0,\uns}_-(x) -\eta^{0,\aux}_- (x) + \eta^{0,\aux}_-(x)-\eta^\aux(x) 
\\ &= \frac{c_{-1}}{\eps^3}\Delta \psi^0\left(\frac{x-x_-}{\eps}\right) +\de\eta_-^{\uns}(x)-\de\eta_-^\aux(x) 
\end{align*} 
where $\eta^{0,\star}_-$, $\star=\uns,\aux$ are defined in~\eqref{defxietaunspm}, \eqref{defxietaunspmaux}$, \Delta\psi^0$ is the function analyzed in Theorem~\ref{thm:inner} (recall the inner change of variables \eqref{def:innervariable}) and $\de\eta_-^{\uns}$, $\de\eta_-^\aux$ are the functions introduced in Theorem \ref{thm:matching:intro}. 
Then, it is enough to use the asymptotic formula \eqref{diffinnersol} and the estimates in Theorem \ref{thm:matching:intro}. Indeed, using that $\rminus - x_- = -i \kappa \eps$, we obtain
\begin{align*}
\Delta \eta^\uns(\rminus) &=\frac{c_{-1}}{\eps^3 } \left (
 \Theta e^{-\kappa} \big (1 + \chi(-i\kappa)\big )  + \mathcal{O}\left (\frac{\eps^{1-\nu}}{|\log \eps|^3}\right )
\right ) \\
&= \frac{c_{-1}}{\eps^3 } e^{-\kappa} \left ( \Theta + \mathcal{O}\left (\frac{1}{|\log \eps|} \right ) + e^{\kappa}\mathcal{O}\left (\frac{\eps^{1-\nu} }{|\log \eps|^3}\right )\right ) 
\end{align*}
and therefore, from $e^{\kappa} = \eps^{-c} \leq \eps^{\nu-1}$, we obtain the result. Notice that  
\begin{align*}
\Delta\eta^{\sta}(x) & = 
\eta^{\aux}(x)- \eta^{0,\aux}_+(x) + \eta^{\aux,0}_+(x) -\eta^{0,\sta}_+ (x) + \eta^{0,\sta}_+(x)-\eta^\sta(x) 
\\ &= \frac{c_{+1}}{\eps^3}\Delta \psi^0\left(\frac{x-x_+}{\eps}\right) +\de\eta_+^{\aux}(x)-\de\eta_+^\sta(x) 
\end{align*} 
so the result for $\Delta\eta^\sta$ follows analogously as the one for $\Delta \eta^\uns$.
\end{proof}

To complete the proof of Proposition \ref{prop:CsFormula}, it suffices
to solve the linear system~\eqref{eq:linearsystemCs}. Indeed, we have that, the linear system for $C_{1,2}^\uns$ can be rewritten as 
\[
\begin{pmatrix}
    1 & e^{-\frac{2i \ol{\rminus}}{\eps}} \\ 
    e^{\frac{2i \rminus}{\eps}} & 1 
\end{pmatrix}
\begin{pmatrix}
    C_1^\uns \\ C_2^\uns
\end{pmatrix}
= \begin{pmatrix}
    e^{-\frac{i \ol{\rminus}}{\eps}} \Delta \eta^\uns(\ol{\rminus}) \\
    e^{\frac{i\rminus}{\eps}} \Delta \eta^\uns (\rminus)
\end{pmatrix}  .
\]
Thus, using that $\eps^{-1} i (\rminus - \ol{\rminus})=-\eps^{-1} 2\pi$, that $\rminus=x_- - i\kappa \eps$ and Lemma~\ref{lem:delta:eta:uns}  
\begin{align*}
    C_1^\uns &=  \frac{c_{-1}}{\eps^3} e^{-\frac{i\ol{x_-}}{\eps}} \left ( \Theta + \mathcal{O}\left (\frac{1}{|\log \eps| }\right )\right ) \\
    C_2^\uns &=  \frac{c_{-1}}{\eps^3} e^{\frac{i{x_-}}{\eps}} \left (  \Theta + \mathcal{O}\left (\frac{1}{|\log \eps| }\right )\right ).
\end{align*}
Proceeding analogously for $C_{1,2}^\sta$ we obtain
\begin{align*}
    C_1^\sta &=  \frac{c_{+1}}{\eps^3} e^{-\frac{i\ol{x_+}}{\eps}} \left ( \Theta + \mathcal{O}\left (\frac{1}{|\log \eps| }\right )\right ) \\
    C_2^\sta &=  \frac{c_{+1}}{\eps^3} e^{\frac{i{x_+}}{\eps}} \left (  \Theta + \mathcal{O}\left (\frac{1}{|\log \eps| }\right )\right ).
\end{align*}
Since $c_{\mp 1} = \pm(\sqrt{|\gamma|})^{-1}$ is given in ~\eqref{def:cpm1}, this completes the proof of Proposition \ref{prop:CsFormula}.

\subsection{Notation and preliminaries}

The rest of the paper is devoted to prove the intermediate results in the previous sections. In order to do so, 
here, we set some standard notations used in our work and to provide (and prove) a general result improving the classical fixed point theorem. 
We will use the following notation and conventions:

\begin{itemize}
\item For $g,h: \Omega\subset \mathbb{C}\to \mathbb{C}$, a function defined in a complex set $\Omega$, we will say that $|g(x)|\lesssim |h(x)|$ if there exists a constant $M$ such that for all $x\in \Omega$, $|g(x)|\leq M |h(x)|$. 
\item 
Let $X$ be a Banach space endowed with the norm $\| \cdot \|_X$. We will use the notation $B(\varrho) \subset X$ for the closed ball of radius $\varrho$ centered at the origin of $X$, namely
\[
B(\varrho)=\{ \mathbf{x} \in X : \| \mathbf{x}\|_X \leq \varrho\}.
\]
\item 
From now on, $\kappa_0,\varepsilon>0$ will be fixed; $\kappa_0$ is as large and we need and $\varepsilon_0>0$ is as small as necessary. All the constants appearing in the results are uniform with respect to $\varepsilon \in (0,\varepsilon_0]$ and $\kappa \geq \kappa_0$. Moreover, when we say in the statement of a result, that $\varepsilon$ is small enough (resp. $\kappa$ is big enough) we mean that we are choosing $\varepsilon_0>0$ small enough (resp. $\kappa_0$ big enough) such that the statement hold for $\varepsilon \in (0,\varepsilon_0]$ (resp. $\kappa \geq \kappa_0$). 
\end{itemize}

We present now a result which is a consequence of the Banach fixed point theorem. We will use it several times along the work. 
\begin{theorem}\label{thm:fixedpoint}
    Let $(X \| \cdot \|_X ), (Y,\| \cdot \|_Y)$ be Banach spaces and take any $(\mathbf{x}_0,\mathbf{y}_0)\in X\times Y$. Consider $\mathbf{F}: X \times Y \to X\times Y$ an operator, $\mathbf{F}= (\mathbf{F}_X, \mathbf{F}_Y)$, satisfying that, 
    there exist positive constants $\mathbf{c}$, 
    \[\varrho\geq 3(\mathbf{c}+1) \max \{ \|{\mathbf{F}}_X [\mathbf{x}_0, \mathbf{y}_0]-\mathbf{x}_0\|_X , \|{\mathbf{F}}_Y [\mathbf{x}_0, \mathbf{y}_0]-\mathbf{y}_0)\|_Y \},
    \]
    $L_1,L_2$ and $L_3$ such that 
    \begin{equation}\label{cond1:thm:fixedpoint}
    \begin{aligned}
    \|\mathbf{F}_X[\mathbf{x}, \mathbf{y}]- \mathbf{F}_X[\wt{\mathbf{x}},\wt{\mathbf{y}}] \|_X \leq \mathbf{c}\|\mathbf{y}-\wt{\mathbf{y}}\|_Y + L_1 \|\mathbf{x} - \wt{\mathbf{x}}\|_X \\ 
    \|\mathbf{F}_Y[\mathbf{x},\mathbf{y}]- \mathbf{F}_Y[\wt{\mathbf{x}},\wt{\mathbf{y}}] \|_Y \leq L_2 \|\mathbf{x}-\wt{\mathbf{x}}\|_X+L_3\|\mathbf{y}-\wt{\mathbf{y}}\|_Y   
    \end{aligned}
    \end{equation}
    for any $(\mathbf{x}-\mathbf{x}_0,\mathbf{y}- \mathbf{y}_0), (\wt{\mathbf{x}}-\mathbf{x}_0, \wt{\mathbf{y}}- {\mathbf{y}}_0) \in B(\varrho) \times B(\varrho)\subset X\times Y$.
    Then, if 
    \begin{equation}\label{cond:thm:fixepoin}
    L_1 + \mathbf{c}(L_2+L_3),L_2+L_3\leq \frac{1}{3},
    \end{equation}
    the fixed point equation $(\mathbf{x},\mathbf{y}) = \mathbf{F}[\mathbf{x},\mathbf{y}]$ restricted to $B(\varrho) \times B(\varrho)$ has a unique solution.  
\end{theorem}
\begin{proof}
   We endow $X\times Y$ with the norm $\|(\mathbf{x}, \mathbf{y})\|_\times = \max\{\|\mathbf{x}\|_X , \|\mathbf{y}\|_Y\}$.
   We notice that $B(\varrho) \times B(\varrho) \subset X\times Y$ is indeed the ball of radius $\varrho$ centered at the origin. 
   
   We first claim that, if $(\mathbf{x}-\mathbf{x}_0, \mathbf{y}-\mathbf{y}_0) \in B(\varrho) \subset X\times Y$, then 
   \[(\mathbf{x}-\mathbf{x}_0,\mathbf{F}_Y[\mathbf{x},\mathbf{y}]- \mathbf{y}_0) \in B(\varrho)\subset X\times Y.\]
   Indeed, it is clear that 
   \begin{align*}
   \|\mathbf{F}_Y[\mathbf{x},\mathbf{y}]-\mathbf{y}_0\|_Y &\leq \|\mathbf{F}_Y[\mathbf{x},\mathbf{y}]-\mathbf{F}_Y[\mathbf{x}_0,\mathbf{y}_0]\|_Y + \|\mathbf{F}_Y[\mathbf{x}_0,\mathbf{y}_0]-\mathbf{y}_0\|_Y \\
   & \leq \varrho \left ( L_2 + L_3 + \frac{1}{3(\mathbf{c}+1)} \right ) \leq \varrho 
   \end{align*}
   where we have used that $L_2+L_3 \leq \frac{1}{3}$.

   Consider the  operator
   \[
   \wh{\mathbf{F}}[\mathbf{x},\mathbf{y}]= \big ( \mathbf{F}_X(\mathbf{x}, \mathbf{F}_Y[\mathbf{x}, \mathbf{y}] ), \mathbf{F}_Y[\mathbf{x},\mathbf{y}]\big ), 
   \]
   which has the same fixed points that $\mathbf{F}$, 
   and we compute the Lipschitz constant of 
   the operator $\wh{\mathbf{F}}$. By hypothesis we have that 
   \begin{align*}
   \|\wh{\mathbf{F}}_X[\mathbf{x}, \mathbf{y}]- \wh{\mathbf{F}}_X[\wt{\mathbf{x}},\wt{\mathbf{y}}] \|_X & \leq \mathbf{c}\|\mathbf{F}_Y[\mathbf{x},\mathbf{y}]- \mathbf{F}_Y[\wt{\mathbf{x}},\wt{\mathbf{y}}] \|_Y + L_1 \|\mathbf{x} - \wt{\mathbf{x}}\|_X \\
   & \leq \mathbf{c}L_3 \|\mathbf{y}-\wt{\mathbf{y}}\|_Y + (L_1+ \mathbf{c}L_2) \|\mathbf{x} - \wt{\mathbf{x}}\|_X.
   \end{align*}
   Then, denoting $L=\max\{ L_1+\mathbf{c}L_2+\mathbf{c}L_3, L_2+L_3\} $
   \[
   \|\wh{\mathbf{F}} [\mathbf{x}, \mathbf{y}]- \wh{\mathbf{F}}[\wt{\mathbf{x}},\wt{\mathbf{y}}] \|_\times \leq L \|(\mathbf{x},\mathbf{y})\|_\times 
   \]
   and hence, the Lipschitz constant of $\wh{\mathbf{F}}$ is $L\leq \frac{1}{3}$ by hypothesis.

   In addition, for $(\mathbf{x}-\mathbf{x}_0,\mathbf{y}-\mathbf{y}_0) \in B(\varrho) \subset X\times Y$, 
   \begin{align*}
   \|\wh{\mathbf{F}} [\mathbf{x}, \mathbf{y}]-(\mathbf{x}_0, \mathbf{y}_0)\|\times \leq &    
   \|\wh{\mathbf{F}} [\mathbf{x}, \mathbf{y}]- \wh{\mathbf{F}} [\mathbf{x}_0, \mathbf{y}_0]\|_\times +\|\wh{\mathbf{F}} [\mathbf{x}_0, \mathbf{y}_0] -(\mathbf{x}_0, \mathbf{y}_0)\|_\times\\ \leq& 
L \|(\mathbf{x}, \mathbf{y})-(\mathbf{x}_0, \mathbf{y}_0)\|_\times +
    \|\wh{\mathbf{F}} [\mathbf{x}_0, \mathbf{y}_0]-  {\mathbf{F}} [\mathbf{x}_0, \mathbf{y}_0] \|_\times \\ & + \|{\mathbf{F}} [\mathbf{x}_0, \mathbf{y}_0] - (\mathbf{x}_0, \mathbf{y}_0) \|_\times       \\ \leq &  L\varrho +
    \|\mathbf{F}_X[\mathbf{x}_0, \mathbf{F}_Y[\mathbf{x}_0, \mathbf{y}_0] ]-\mathbf{F}_X[\mathbf{x}_0,\mathbf{y}_0]\|_X+ \frac{\varrho}{3(\mathbf{c}+1)} \\\leq &  \varrho \left (L+ \frac{1}{3(\mathbf{c}+1)} +\frac{\mathbf{c}}{3(\mathbf{c}+1)} \right )  \leq \varrho.
   \end{align*}
Therefore, the map $\wh{\mathbf{F}}$ is a contraction from $B(\varrho)\subset X\times Y$ to itself and the fixed point theorem implies the existence of an unique fixed point belonging to $B(\varrho) \subset X\times Y$.
   \end{proof}

\section{The invariant manifolds in the outer domain}
\label{sec:outer}

Here we prove Theorem \ref{thm:outer:intro} with a fixed point argument. Then the first step of the proof, done in Section~\ref{sec:fixedpointapproach:outer}, is to reformulate Theorem~\ref{thm:outer:intro} as a fixed point problem. In Section~\ref{sec:contractingmapping:outer} we prove that the fixed point operator is a contraction in a suitable closed ball of a Banach space. 

We prove Theorem~\ref{thm:outer:intro} for the unstable manifold and we obtain the corresponding result for the stable manifold taking advantage of the symmetries of the system. Indeed, by definition~\eqref{def:defdomainouter} of $D^{\out,\star}_{\kappa}$, $x\in D^{\out,\sta}_\kappa$ if and only if $-x \in D^{\out,\uns}_\kappa$ and using that the system is reversible with respect to the involution $\Psi$ in~\eqref{def:involution}, we deduce that, if $(\xi^\uns, \eta^\uns)$ satisfy the conditions in Theorem~\ref{thm:outer:intro}, then
\[
\xi^\sta(x) :=\xi^\uns(-x), \qquad \eta^\sta(x) : = \eta^\uns(-x)
\]
satisfy the corresponding properties.

\subsection{The fixed point approach}\label{sec:fixedpointapproach:outer}

For given $\kappa> 0$ and  $\theta \in \left (0, \arctan \left (\frac{\pi}{3 \alpha}\right )\right )$ we recall definition~\eqref {def:defdomainouter} (see also Figure~\ref{fig:outer}) of the complex domains $D^{\out,\uns}_{\kappa}$. From now on we fix $\theta$ and we do not write explicitly the dependence of the domains on $\theta$. The role of $\kappa$, as we will see, is completely different. 

We introduce, for a real-analytic function $h:D^{\uns,\out}_{\kappa}\rightarrow \CC$, the norm
\begin{align*}
\|h\|_{m,\ell} & = \displaystyle\sup_{x\in D^{\uns,\out}_{\kappa}\cap\{ \Re(x)\leq -2 \alpha \}}|\cosh x|^m |h(x)| \\
& \quad + \displaystyle\sup_{x\in D^{\uns,\out}_{\kappa}\cap\{ \Re(x)\geq -2\alpha\}}|x-x_-|^\ell |x-\ol{x}_-|^\ell |h(x)|
\end{align*}
with $\ell,m\in \mathbb{R}$.  
Then, we define the associated Banach space 
\begin{align*}
\mathcal{E}_{m,\ell}& =\{h: D^{\uns,\out}_{\kappa}\rightarrow \CC \textrm{ is real-analytic with } \|h\|_{m,\ell}<\infty \},\\
\mathcal{DE}_{m,\ell}&=\{h:D^{\uns,\out}_{\kappa}\rightarrow \CC \textrm{ is real-analytic with } \|h\|_{m,\ell}+\|h'\|_{m,\ell+1}<\infty \},
\end{align*}
and the product Banach space 
\[ \EE_\times=\mathcal{DE}_{1,3}\times\mathcal{E}_{1,5},  
\]
with the product norm
\[
\| (h_1,h_2)\|_{\times } = \max\big \{ \|h_1\|_{1,3} + \|h_1'\|_{1,4}, \|h_2\|_{1,5} \big \}.
\] 

We have the following lemma, whose proof is straightforward.

\begin{lemma}\label{propertiesnorm1}
	There exists $M>0$ depending only on $\theta$, such that, for any $\kappa>0$ and $g,h:D^{\out,\uns}_{\kappa} \rightarrow\CC$, it holds
		\begin{enumerate}
		\item If $\ell_2\geq\ell_1\geq 0$, then
		\[\|h\|_{m,\ell_2}\leq M\|h\|_{m,\ell_1}\quad \text{and} \quad \|h\|_{m,\ell_1}\leq \dfrac{M}{(\kappa\eps)^{\ell_2-\ell_1}}\|h\|_{m,\ell_2}.\]
		\item If $\ell_1,\ell_2\geq 0$  and $\|g\|_{m_1,\ell_1},\|h\|_{m_2,\ell_2}<\infty$, then
		$$\|gh\|_{m_1 + m_2,\ell_1+\ell_2}\leq \|g\|_{m_1,\ell_1}\|h\|_{ m_2,\ell_2}.$$	
  \item If $m_2 \geq m_1$, $\ell \geq 0$ and $\| g \|_{m_2,\ell} < \infty$ then
  \[
  \|g\|_{m_1,\ell} \leq M \| g \|_{m_2,\ell}. 
  \]
  \end{enumerate}	
\end{lemma}

In this functional setting, Theorem~\ref{thm:outer:intro} (for the unstable solution) is a straightforward consequence of the following result.
\begin{proposition}\label{prop:outer:proof}
Consider the system~\eqref{eq:perturb}, that is
\begin{equation}\label{eq:perturb:proof}
\LL_1\xi =\FF_1 [\xi ,\eta ] ,\qquad  
\LL_2\eta =\FF_2[\xi ,\eta ] 
\end{equation}
with $\LL_1,\LL_2$ and $\FF=(\FF_1,\FF_2)$ defined in~\eqref{def:diffoperators} and~\eqref{eq:outer:RHSoperator} respectively.
There exists $\kappa_0,\eps_0$ and a constant $M_1$ such that for $\eps \in (0,\eps_0)$ and $\kappa >\kappa_0$, system~\eqref{eq:perturb:proof} has solutions $(\xi^\uns,\eta^\uns) \in \mathcal{E}_\times$ satisfying $\|(\xi^\uns,\eta^\uns)\|_\times \leq M_1\eps^2$ and $\partial_x \xi^\uns(0)=0$.
\end{proposition}

\begin{remark} 
By definition of the Banach space $\mathcal{E}_\times$, since $(\xi^\uns,\eta^\uns) \in \mathcal{E}_\times$, it satisfies the boundary conditions 
\begin{equation}\label{eq:perturb:conditions:inf}
\lim_{\Re x \to -\infty} (\xi^{\uns}(x),\eta^{\uns}(x)) = (0,0). 
\end{equation}
Therefore, by Cauchy's theorem, it is also true for $x$ on $\mathbb{R}$ that 
\[
\lim_{x \to -\infty} (\partial_x\xi^{\uns}(x), \partial_x \eta^{\uns}(x)) = (0,0). 
\]
Then, 
\[
\lim_{x\to -\infty} \wt G(\xi^\uns(x), \partial_x \xi^\uns(x), \eta^\uns (x), \partial_x \eta^{\uns} (x),x) = \wt G(0,0,0,0)=0,
\]
with $\wt G$ the first integral defined in~\eqref{defGfirstintegral}, and therefore, for $x\in D_{\kappa}^{\out,\uns}$, \[\wt G(\xi^\uns(x), \partial_x \xi^\uns(x), \eta^\uns(x), \partial_x \eta^{\uns}(x),x)=0.\] 
In addition, for $x\in D^{\out,\uns}_\kappa$, we have $|x-x_+|,|x-\ol{x_+}|\geq M$ for some constant $M>0$ and hence the estimates in Theorem~\ref{thm:outer:intro} in the domain $D_\kappa^{\out,\uns} \cap \{\Re x \geq -2 \alpha\}$ hold trivially.
\end{remark}

The remaining part of this section is devoted to prove Proposition~\ref{prop:outer:proof}. In order to do so, we seek a fixed point formulation of~\eqref{eq:perturb:proof} in a suitable ball of $\mathcal{E}_\times$. Therefore, the next step in our analysis is to look for suitable right inverses of the operators $\LL_1$ and $\LL_2$. 

We start with $\LL_1$. The homogeneous equation $\LL_1\xi=0$ has two linearly independent solutions 
$\zeta_1$ and $\zeta_2$, where the odd function $\zeta_1(x) = u_0'(x)$ is a solution due to the translation symmetry and the even function $\zeta_2(x)$ is uniquely defined by the normalization 
\begin{equation}\label{wronskian}
\zeta_1(x) \zeta_2'(x) - \zeta_1'(x) \zeta_2(x) = 1, \quad x \in \R,
\end{equation}
which follows from the Wronskian identity.  The following lemma gives the second solution $\zeta_2$ and it is proved in  Appendix~\ref{app:g1:outer}.

\begin{lemma}\label{lem:zeta2:welldefined} 
For a given $\kappa>0$, 
there exists a unique real analytic even function $\zeta_2:D^{\out,\uns}_{\kappa} \to \mathbb{C}$ satisfying~\eqref{wronskian}. In addition, $\zeta_2(0)\neq 0$ and   
$\|\zeta_2 \|_{-1,2} + \|\zeta_2'\|_{-1,3} \leq M$ for some constant independent of $\kappa\geq 1$.
\end{lemma}

\begin{remark}
    We notice that $\zeta_1 = u'_0 \in D\mathcal{E}_{1,2}$. 
\end{remark}

The classical theory of second-order differential equations implies that we can construct right inverses of the operator $\mathcal{L}_1$ as
\begin{equation}\label{invG1:general}
\mathcal{L}_1^{-1}[h] (x) = \zeta_1(x) \left [ C_1+ \int_{x_1}^x\zeta_2(s)h(s)ds \right ] +  \zeta_2(x) \left [ C_2-\int_{x_2}^x\zeta_1(s)h(s) ds
\right ]
\end{equation}
for any given $x_1,x_2,C_1,C_2 \in \mathbb{R}$. However, we are interested in solutions $(\xi^\uns,\eta^\uns)$ satisfying the boundary conditions $\partial_x \xi (0)=0$ and the decay behavior ~\eqref{eq:perturb:conditions:inf}.
Therefore, we impose the same conditions on the solutions of $\mathcal{L}_1 \xi = h$ and we easily obtain that the right inverse is formally given by  
\begin{equation}\label{def:intoperators-1}
\GG_1^\out[h] (x) = \zeta_1(x)\int_0^x\zeta_2(s)h(s)ds - \zeta_2(x)\int_{-\infty}^x\zeta_1(s)h(s) ds
\end{equation}
where the (complex) integration path is, in the first integral, the segment between $0$ and $x$  and, in the second integral, corresponds to the path parameterized by $s= x+ t$, with $t\in (-\infty,0]$.  

In addition, it is straightforward to check that a right inverse of 
the operator $\mathcal{L}_2$ can be formally expressed as
\begin{equation}\label{def:intoperators-2}
\GG_2^\out[h] =
-\frac{i\eps}{2} e^{i\eps^{-1}x}\int_{-\infty}^x  e^{-i\eps^{-1}s} h(s)ds+\frac{i\eps}{2} e^{-i\eps^{-1}x}\int_{-\infty}^x  e^{i\eps^{-1}s} h(s)ds,
\end{equation} 
where the integration path is the horizontal line $s= x+ t$, $t\in (-\infty,0]$. 

The following lemma describes how the operators $\GG_1^\out$ and $\GG_2^\out$ act on functions belonging to $D\mathcal{E}_{1,3}$ and $\mathcal{E}_{1,5}$ respectively.  Its proof follows the same lines as the ones of Proposition 4.3 in~\cite{GomideGSZ22} and we sketch the main steps of the proof in Appendix~\ref{app:g1:outer}.

\begin{lemma}\label{prop_operators}
	The operators $\GG_1^\out$ and $\GG_2^\out$ introduced in \eqref{def:intoperators-1} and \eqref{def:intoperators-2} have the following properties. 
	\begin{enumerate}
		\item $\mathcal{G}_i^\out\circ\mathcal{L}_i=\mathcal{L}_i\circ\mathcal{G}_i^\out=\mathrm{Id}$. 
		\item\label{item2prop_operators} For any $m>1$ and $\ell\ge 5$, there exists a constant $M>0$ independent of $\eps$ and $\kappa$ such that, for every $h \in\mathcal{E}_{m,\ell}$,
		\begin{equation*}\label{desig}
		\left\|\mathcal{G}_1^\out [h]\right\|_{1,\ell-2} \leq M\|h \|_{m,\ell}\qquad \text{and}\qquad \left\|\pa_x\mathcal{G}_1^\out [h]\right\|_{1,\ell-1} \leq M\|h \|_{m,\ell} 
		\end{equation*}
and 
  \[
\pa_x \mathcal{G}_1^\out [h](0)=0.
  \]
  In addition, if $h $ is real analytic, then $\GG_1[h]$ is also real analytic. 
		\item For any $m\geq 1$, $\ell\ge0$, there exists $M>0$ such that for $h \in\mathcal{E}_{m,\ell}$,
\begin{equation*}
\label{desig2}
		\left\|\GG_2^\out [h]\right\|_{m,\ell}\leq M\eps^2\|h \|_{m,\ell}
		\end{equation*}
Moreover, when $h$ is real analytic, $\GG_2^\out[h]$ is also real analytic. 
		\end{enumerate}
		\end{lemma}

In order to prove Proposition~\ref{prop:outer:proof}, we use Lemma \ref{prop_operators} and look for solutions of~\eqref{eq:perturb:proof} belonging to $\mathcal{E}_\times$, satisfying $\partial_x \xi(0)=0$ as fixed points of the operator 
\begin{equation}\label{def:operatorFhat}
\FF^\out = \big (\GG_1^\out \circ\FF_1,\GG_2^\out \circ\FF_2 \big )
\end{equation}
where $\FF_i$ are the operators defined in~\eqref{eq:outer:RHSoperator}.

\subsection{The contraction mapping}\label{sec:contractingmapping:outer}

We prove Proposition~\ref{prop:outer:proof} using Theorem~\ref{thm:fixedpoint}. To do so, we study $\FF^\out[0,0]$ (Lemma~\ref{lemma:Fhat0}) and the Lipschitz constant of $\FF^\out$ in a suitable ball $B(R\eps^2) \subset \mathcal{E}_\times$ (Lemma~\ref{lemma:Fhat}).
  
\begin{lemma}\label{lemma:Fhat0}
There exists a constant $b_1>0$ independent of $\eps$ and $\kappa$ such that 
\[
 \|\FF^\out[0,0]\|_\times\leq b_1\eps^2.
\]
\end{lemma}

\begin{proof} 
	From definition~\eqref{eq:outer:RHSoperator} of $\mathcal{F}$, 
\[
\FF[0,0] = (0,f'(u_0) (u_0 - f(u_0)) + f''(u_0) (u'_0)^2).
\]
Since $u_0 \in \mathcal{E}_{1,1}$, see~\eqref{def:soliton} and Lemma~\ref{lemma:propertiesu0}, and $f(u)= u^2 + 2 \gamma u^3 $, $\FF_2[0,0] \in \mathcal{E}_{2,5} \subset \mathcal{E}_{1,5}$ with $\|\FF_2[0,0]\|_{1,5} \lesssim 1$ and from Lemma~\ref{prop_operators} the result holds true. 
\end{proof}

\begin{lemma}\label{lemma:Fhat}
There exists $C_1>0$ such that for all $R>0$,  
if $(\xi,\eta),(\wt\xi,\wt\eta)\in B(R\eps^2)\subset \EE_\times$, then 
the operator $\FF^\out$ in \eqref{def:operatorFhat} satisfies
\[
\begin{split}
\left\|\FF_1^\out[\xi,\eta]-\FF_1^\out(\wt\xi,\wt\eta)\right\|_{1,3}&\leq C_1 \|\eta-\wt\eta\|_{1,5}+\frac{C}{\kk^2}\|(\xi,\eta)-(\wt\xi,\wt\eta)\|_\times\\
\left\|\pa_x\FF_1^\out[\xi,\eta]-\pa_x\FF_1^\out[\wt\xi,\wt\eta]\right\|_{1,4}&\leq C_1 \|\eta-\wt\eta\|_{1,5}+\frac{C}{\kk^2}\|(\xi,\eta)-(\wt\xi,\wt\eta)\|_\times \\
\left\|\FF_2^\out[\xi,\eta]-\FF_2^\out[\wt\xi,\wt\eta]\right\|_{1,5}&\leq \frac{C}{{\kk^2}}\|(\xi,\eta)-(\wt\xi,\wt\eta)\|_\times .
\end{split} 
\]
for some constant $C=C(R)>0$ independent of $\eps$ and $\kk$.  
\end{lemma}

\begin{proof}
Let $(\xi,\eta), (\wt \xi , \wt \eta) \in B(R \eps^2)$. 
	We define $\zeta_\lambda = (\xi_\lambda, \eta_\lambda)=(\wt \xi , \wt \eta) + \lambda \big (( \xi,\eta) - (\wt \xi , \wt \eta)\big )$. 
	Then, using the mean value theorem
	\[
	\FF_1 [\xi , \eta] (x) - \FF_1 [\wt \xi ,\wt \eta] (x) = 
	\int_{0}^1  D \FF_1 [\zeta_\lambda](x)  \big ( \xi(x)- \wt \xi(x),\eta(x)- \wt \eta(x) \big )^\top \,d\lambda
	\]
	with 
	\[
	D \FF_1[\zeta_\lambda] (x) = \big (\partial_\xi \FF_1[\zeta_\lambda](x) , \partial_{\eta} \FF_1 [\zeta_\lambda](x) \big )=\big (12 \gamma u_0(x) \xi_\lambda(x) + 2 \xi_\lambda(x) + 6 \gamma \xi_\lambda^2(x) , -1 \big ) 
	\]
	and satisfying
	\[
	\| \partial_\xi \FF_1[\zeta_\lambda] \|_{1,2} \lesssim \frac{\eps^2}{(\eps \kappa)^2} + \frac{\eps^2}{\eps \kappa} + \frac{\eps^4}{(\eps \kappa)^4} \lesssim \frac{1}{\kappa^2},
	\]
	where we have used Lemma~\ref{propertiesnorm1}, that $\kappa$ is big enough and that $\eps$ is small enough. Then by the second item in Lemma~\ref{prop_operators} and recalling that 
	$\FF_1^\out = \GG_1^\out \circ \FF_1$
	\begin{align*}
	\|\FF_1^\out[\xi,\eta]- \FF_1^\out[\wt \xi, \wt \eta]\|_{1,3} & \leq M \| \FF_1 [\xi,\eta]- \FF_1[\wt \xi, \wt \eta] \|_{1,5}\\ 
	&\leq  M \|\wt \eta - \eta \|_{1,5} +  \frac{C}{\kappa^2} \|\xi - \wt \xi\|_{1,3}  
	\end{align*}
	where $M$ is the constant provided in item~\ref{item2prop_operators} of Lemma~\ref{prop_operators}, which is independent on $R$.  
	In addition, using again item~\ref{item2prop_operators} in Lemma~\ref{prop_operators}
	\[
	\|\partial_x \FF_1^\out[\xi,\eta]- \partial_x  \FF_1^\out[\wt \xi, \wt \eta]\|_{1,4}  \leq M\|\wt \eta - \eta \|_{1,5} +  \frac{C}{\kappa^2} \|\xi - \wt \xi\|_{1,3}.
	\]
	With respect to the second component, we 
	define
	\[
	\mathcal{M}[\xi,\eta,\xi'] = f'(u_0+\xi)\left(u_0+\xi+\eta-f(u_0+\xi)\right)+f''(u_0+\xi)(u_0'+\xi')^2 
	\]
	which satisfies $\mathcal{M}[\xi,\eta,\xi'] = \FF_2[\xi,\eta]$. We note that 
	$\| u_0 + \xi_\lambda\|_{1,1} ,\|u_0'+\xi_\lambda'\|_{1,2}\lesssim 1$. Then, computing 
	\[
	\partial_{\xi'} \mathcal{M}[\xi_\lambda,\eta_\lambda,\xi_\lambda'] =2 f''(u_0+\xi_\lambda) (u_0'+ \xi_{\lambda}') ,
	\]
	we have that 
	\[
	\|\partial_{\xi'} \mathcal{M}[\xi_\lambda,\eta_\lambda,\xi_\lambda'] \|_{2,1} \lesssim \frac{1}{(\kappa \eps)^{2}}.
	\]
	In addition
	\[
	\| \partial_{\eta} \mathcal{M}[\xi_\lambda,\eta_\lambda,\xi_\lambda'] \|_{2,0}  \lesssim \frac{1}{(\kappa \eps)^{2}},\qquad \| \partial_{\xi} \mathcal{M}[\xi_\lambda,\eta_\lambda,\xi_\lambda'] \|_{2,2}  \lesssim \frac{1}{(\kappa \eps)^{2}}.
	\]
	Then, using the mean's value theorem as Lemma~\ref{propertiesnorm1}, we obtain
	\[
	\| \FF_2 [\xi,\eta] -  \FF_2[\wt \xi,\wt \eta] \|_{1,5} \lesssim 
	\frac{1}{( \kappa \eps)^2} \| (\xi, \eta) - (\wt \xi, \wt \eta)\|_\times,
	\]
from which the last bound in Lemma~\ref{lemma:Fhat} follows by recalling that $\FF^{\out}_2 = \mathcal{G}_2^{\out} \circ \FF_2$ and applying the third item of Lemma~\ref{prop_operators}. 
\end{proof}

\begin{proof}[End of the proof of Proposition~\ref{prop:outer:proof}]
We apply now Theorem~\ref{thm:fixedpoint} to the operator $\mathcal{F}^\out$. Indeed, using Lemmas~\ref{lemma:Fhat0} and~\ref{lemma:Fhat}, we take (with the notation in Theorem~\ref{thm:fixedpoint}) $(\mathbf{x}_0,\mathbf{y}_0)=(0,0)$, 
$\mathbf{c}=C_1$, 
\[
\varrho=3(C_1+1) b_1 \eps^2\geq 3(C_1+1) \|\mathcal{F}^\out[0,0]\|_\times 
\] 
and $L_1=L_2=L_3= \frac{C}{\kappa^2}$.
Hence the conditions~\eqref{cond1:thm:fixedpoint} and~\eqref{cond:thm:fixepoin} in Theorem~\ref{thm:fixedpoint} are trivially satisfied taking $\kappa$ big enough. Therefore, $\mathcal{F}^\out$ has a unique fixed point which belongs to $B(3(C_1+1) b_1\eps^2)$. This completes the proof of Proposition~\ref{prop:outer:proof}.
\end{proof}

\section{An auxiliary solution}\label{sec:aux}
 
Here we prove Theorem~\ref{thm:aux:intro} by constructing a real-analytic solution $(\xi^\aux,\eta^\aux)$ of equation~\eqref{eq:perturb} defined in the domain $D^\aux_\kk$, see~\eqref{def:domainaux} and Figure~\ref{fig:aux}. As we have done in Section~\ref{sec:outer}, we fix $\theta\in\left (0, \mathrm{arctan}\left (\frac{\pi}{\alpha} \right) \right )$ and we omit the dependence on it along the proof.  We will run the fixed point argument similar to that of Section~\ref{sec:outer}. Note however that we have to modify some arguments in a suitable way so that 
\begin{itemize}
    \item The integrals defining the right inverse of the linear operators $\mathcal{L}_1,\mathcal{L}_2$ have to be over paths within the new domain $D^\aux_\kk$, see~\eqref{def:intoperators-1} and~\eqref{def:intoperators-2}.
    \item We have to ensure that the solutions belongs to the $0$ level curve of the first integral $\wt G$ given by~\eqref{defGfirstintegral}.
\end{itemize}  

\subsection{The fixed point approach}
We first define the Banach space where the fixed point argument is carried out. 
Given $\kappa > 0$, we define for a real-analytic function $h:D^{\aux}_{\kappa}\rightarrow \CC$ the norm
\begin{equation}
\label{normcosh:aux}
\|h\|_{\ell}  = \displaystyle\sup_{x\in D^{\aux}_{\kappa}}|(x-x_-)^\ell (x-\ol{x}_-)^\ell (x-x_+)^\ell (x-\ol{x}_+)^\ell h(x)|,
\end{equation}
with the associated Banach spaces 
\begin{equation}\label{def:Banach:aux}
\begin{aligned}
\mathcal{Y}_{\ell}&=\{h:D^{\aux}_{\kappa}\rightarrow \CC \textrm{ is real-analytic with } \|h\|_{\ell}<\infty \},\\
\mathcal{DY}^1_{\ell}&=\{h:D^{\aux}_{\kappa}\rightarrow \CC \textrm{ is real-analytic with } \|h\|_{\ell}+\|h'\|_{\ell+1}<\infty \},\\
\mathcal{DY}^2_{\ell}&=\{h:D^{\aux}_{\kappa}\rightarrow \CC \textrm{ is real-analytic with } \|h\|_{\ell}+\eps\|h'\|_{\ell}<\infty \}.
\end{aligned}
\end{equation}
Then, we define the product Banach space
\[
\YY_\times=\mathcal{DY}^1_{3}\times \mathcal{DY}^2_{5}
\]
with the norm
\[
\|(\xi,\eta)\|_\times= \max\big \{\|\xi\|_{3}+\|\xi'\|_{4}, \|\eta\|_{5}+\eps\|\eta'\|_{5}\big \}.
\]
The counterpart of Lemma~\ref{propertiesnorm1} for the Banach spaces $\mathcal{Y}_\ell$ is the following result whose proof is left to the reader.

\begin{lemma}\label{propertiesnorm1:aux}
	There exists $M>0$, such that, for any $\kappa>0$ and $g,h:D^{\aux}_{\kappa} \rightarrow\CC$, it holds that
		\begin{enumerate}
		\item If $\ell_2\geq\ell_1\geq 0$, then
		\[\|h\|_{\ell_2}\leq M\|h\|_{\ell_1}\quad \text{and} \quad \|h\|_{\ell_1}\leq \dfrac{M}{(\kappa\eps)^{\ell_2-\ell_1}}\|h\|_{\ell_2}.\]
		\item If $\ell_1,\ell_2\geq 0$  and $\|g\|_{\ell_1},\|h\|_{\ell_2}<\infty$, then
		$$\|gh\|_{\ell_1+\ell_2}\leq \|g\|_{\ell_1}\|h\|_{\ell_2}.$$	
  \end{enumerate}	
\end{lemma}

We rephrase Theorem~\ref{thm:aux:intro} as the following proposition.

\begin{proposition}\label{prop:aux:proof}
 There exist $\kappa_0, \eps_0>0$ and $M_2>0$, such that, if $\eps\in (0, \eps_0)$ and $\kappa> \kappa_0$, the system~\eqref{eq:perturb} has  
 real-analytic solutions $(\xi^\aux,\eta^\aux) \in \mathcal{Y}_\times$ satisfying
\[
\wt G(\xi^\aux,\partial_x \xi^\aux,\eta^\aux,\partial_x \eta^\aux,x)=0, \qquad \partial_x \xi^\aux (0)=0, 
\]
where $\wt G$ is the first integral introduced in \eqref{defGfirstintegral}, and $\|(\xi^\aux,\eta^\aux)\|_\times \leq M_2 \eps^2$. In addition, $\xi^\aux(x)=\xi^\aux(-x)$ and $\eta^\aux(x)= \eta^\aux(-x)$. 
\end{proposition}

To prove Proposition \ref{prop:aux:proof}, we recall that system~\eqref{eq:perturb} is 
\[
\mathcal{L}_1 \xi = \FF_1[\xi,\eta], \qquad \mathcal{L}_2 \eta = \FF_2[\xi,\eta]
\]
with $\mathcal{L}_1, \mathcal{L}_2$ and $\FF=(\FF_1,\FF_2)$ defined in~\eqref{def:diffoperators} and~\eqref{eq:outer:RHSoperator} respectively.  Therefore, in order to set up the fixed point equation, we first introduce the suitable right inverses of the linear operators $\mathcal{L}_1, \mathcal{L}_2$. 
We use the fundamental set of solutions $\zeta_1=u_0'$ and the analytic continuation of $\zeta_2$ (see Lemma~\ref{lem:zeta2:welldefined}). 
The following lemma specifies another suitable property for $\zeta_2$, and it is proved in  Appendix~\ref{app:g1:outer}.

\begin{lemma}\label{lem:zeta2:welldefined:aux}
    The even function $\zeta_2$ in Lemma~\ref{lem:zeta2:welldefined} has an even analytic continuation to $D_\kappa^\aux$. In addition, $\zeta_2 \in D\mathcal{Y}^1_2$.
\end{lemma}

We define now the linear operators 
\begin{equation}\label{def:intoperators-aux} 
\begin{split}
{\GG}^\aux_1[h](x) &= \zeta_1(x)\int_0^x\zeta_2(s)h(s)ds - \zeta_2(x)\int_{0}^x\zeta_1(s)h(s) ds,\\
{\GG}_2^\aux[h](x) &=
-\frac{i\eps}{2} e^{i\eps^{-1}x}\int_{-i\rho}^x  e^{-i\eps^{-1}s} h(s)ds+\frac{i\eps}{2} e^{-i\eps^{-1}x}\int_{i\rho}^x  e^{i\eps^{-1}s} h(s)ds,
\end{split}
\end{equation}
where $\rho=\rho(\theta)= \alpha_+ \tan \theta + \pi - \kappa \eps$ with $\alpha_+ = \Re x_+$, the superior vertex of $D_\kappa^\aux$. 

The following lemma gives estimates for the linear operators ${\GG}_1^\aux, {\GG}_2^\aux$. Its proof follows the same lines as the one of Lemma~\ref{prop_operators} and it is deferred to Appendix~\ref{app:g1:outer}. 

\begin{lemma}\label{lemma:operators:aux}
The operators ${\GG}_1^\aux$ and ${\GG}_2^\aux$ introduced in~\eqref{def:intoperators-aux} have the following properties.
	\begin{enumerate}
		\item $\mathcal{L}_i\circ{\GG}_i^\aux[\xi]=\xi$.
		\item For any  $\ell\ge 5$, there exists a constant $M>0$ independent of $\eps$ and $\kappa$ such that, for every $h \in\YY_{\ell}$,
		\begin{equation*}	 
		\left\|{\GG}_1^\aux [h]\right\|_{\ell-2} \leq M\|h \|_{\ell}\qquad \text{and}\qquad \left\|\pa_x{\GG}_1^\aux [h]\right\|_{\ell-1} \leq M\|h \|_{\ell}.
		\end{equation*}
  In addition, if $h$ is real analytic, ${\GG}_1^\aux[h]$ is real analytic. 
		\item For any  $\ell\ge0$, there exists $M>0$ such that for every  $h \in\mathcal{Y}_{\ell}$,
\begin{equation*}
\label{desig2:aux}
		\left\|{\GG}_2^\aux [h]\right\|_{\ell}\leq M\eps^2\|h \|_{\ell}, \quad \left\|\partial_x {\GG}_2^\aux [h]\right\|_{\ell}\leq M\eps \| h \|_{ \ell}
		\end{equation*}
When $h$ is real analytic, ${\GG}_2^\aux[h]$ is also real analytic. 
		\end{enumerate}
		\end{lemma}

Now, to set up the fixed point argument we proceed in two steps so that we fix the $\wt G$ level curve. For the $\eta$ component, we just impose that it satisfies
\[
\eta={\GG}_2^\aux \circ \FF_2[\xi,\eta].
\]
Note that, then in particular, 
\[
\eta(0)={\GG}_2^\aux \circ \FF_2[\xi,\eta](0).
\]
We use this equality to fix  $\wt G$ at $x=0$. 
Indeed, as we claimed in~\eqref{invG1:general}, $\LL_1$ in~\eqref{def:diffoperators} has several right inverses 
\[
\mathcal{L}_1^{-1} [h] = \zeta_1(x)\left[C_1+\int_0^x\zeta_2(s)h(s)ds\right] - \zeta_2(x)\left[C_2+\int_{0}^x\zeta_1(s)h(s) ds\right].
\]
The condition $\xi'(0)=0$ implies that one has to impose $C_1=0$ (recall that $\zeta_2$ is an even function, see Lemma~\ref{lem:zeta2:welldefined}). We choose a suitable $C_2$ so that the solution lies in $\wt G=0$. Indeed, we have
\[
\wt G\left(-\zeta_2(0)C_2, 0,\eta(0),\eta'(0),0\right)=0.
\]
 
The following lemma ensures that, for a given $\eta$ and $\xi$ in a suitable Banach space,  there exists a unique $C_2$ satisfying this equality.

\begin{lemma}\label{lemma:initialcondition}
Fix $R>0$. There exists $\eps_0$ such that for $\eps \in(0,\eps_0)$, there is a function 
$
\II: B(R\eps^2)\subset\cD\YY^2_5\to \CC
$
such that, for any $\eta \in B(R\eps^2)$, 
\begin{equation}\label{Gimplicit:aux}
\wt G\left(-\zeta_2(0)\II[\eta], 0,\eta(0),\eta'(0), 0\right)=0
\end{equation}
and
\[
\big |\II[\eta] \big | \lesssim \eps^2.
\]
Moreover, for any $\eta,\wt\eta\in B(R\eps^2)\subset\cD\YY^2_5$,
\[
|\II[\eta]-\II[\wt\eta]|\lesssim \eps^2 \|\eta-\wt\eta \|_5.
\]
\end{lemma} 

\begin{proof}
The proof follows by an implicit function theorem. Take $\eta\in B(R\eps^2)\subset\YY_5$ and denote $\eta_0=\eta(0)$ which satisfies $|\eta_0|\lesssim \eps^2$. Then, since $u_0'(0)=0$, see~\eqref{def:soliton}, equation~\eqref{Gimplicit:aux} is equivalent  
\[
0=\mathbf{G}(\xi_0,\eps; \eta_0)= -u_0''(0) \xi_0 - \frac{\eps^2}{2} (\eta_0 +u_0(0) + \xi_0 -f(u_0(0)+ \xi_0) + \wt{\mathbf{G}} (\xi_0) 
\]
with $ |\wt{\mathbf{G}}(\xi_0) |\lesssim |\xi_0|^2$. It is clear that $\mathbf{G}(0,0;\eta)=0$, then, recalling that $u_0''(0)\neq$ (see Lemma~\ref{lemma:propertiesu0}), the implicit function theorem assures, for $\eps$ small enough, the existence of $\xi_0=\xi_0(\eps; \eta_0)$, satisfying $|\xi_0|\lesssim \eps^2$. In addition, since $|\partial_{\eta_0} \xi_0(\eps; \eta_0)| \lesssim \eps^2$, $|\xi_0(\eps;\eta_0)- \xi_0(\eps;\wt{\eta}_0)| \lesssim \eps^2 |\eta_0- \wt{\eta}_0|$ for any $|\eta_0|, |\wt{\eta}_0| \lesssim \eps^2$. Taking $\mathcal{I}[\eta] = -\xi_0(\eps;\eta(0)) (\zeta_2(0))^{-1}$, the result follows provided $|\eta(0)|\lesssim \|\eta\|_5$.
\end{proof}

Based on the results of Lemmas \ref{lemma:operators:aux} and \ref{lemma:initialcondition}, we look for the functions  $(\xi^\aux,\eta^\aux)$ in Proposition \ref{prop:aux:proof} as fixed points of the operator 
\begin{equation}\label{def:operatorFaux}
\FF^\aux [\xi,\eta]=
\begin{pmatrix}
\FF_1^\aux[\xi,\eta]\\\FF_2^\aux [\xi,\eta]
\end{pmatrix}=
\begin{pmatrix}
-\zeta_2 \cdot \II [\eta]+{\GG}_1^\aux \circ \FF_1[\xi,\eta]\\
{\GG}_2^\aux \circ \FF_2 [\xi,\eta]
\end{pmatrix}
\end{equation}
with ${\GG}_1^\aux, {\GG}_2^\aux$ defined in~\eqref{def:intoperators-aux}, $\zeta_2$ defined by Lemma~\ref{lem:zeta2:welldefined:aux} and $\mathcal{F}=(\mathcal{F}_1,\mathcal{F}_2)$ is given in~\eqref{eq:outer:RHSoperator}.

\subsection{The contraction mapping}

The following two lemmas analyze the operator $\FF^\aux$ defined in~\eqref{def:operatorFaux}. 

\begin{lemma}\label{lemma:Fhat0:aux}
There exists a constant $b_2>0$ independent of $\eps$ and $\kappa$ such that 
\[
 \|\FF^\aux [0,0]\|_\times\leq b_2\eps^2.
\]
\end{lemma}
 
\begin{lemma}\label{lemma:Fhat:aux}
There exists $C_2$ such that for all $R>0$, if $(\xi,\eta),(\wt\xi,\wt\eta)\in B(R\eps^2)\subset \YY_\times$, the operator $\FF^\aux$ in \eqref{def:operatorFaux} satisfies
\[
\begin{split}
 \left\|\FF^\aux_1[\xi,\eta]-\FF^\aux_1[\wt\xi,\wt\eta]\right\|_{3}&\leq C_2 \|\eta-\wt\eta\|_{5}+\frac{C}{\kk^2}\|(\xi,\eta)-(\wt\xi,\wt\eta)\|_\times, \\
 \left\|\pa_x\FF^\aux_1 [\xi,\eta]-\pa_x\FF^\aux_1[\wt\xi,\wt\eta]\right\|_{4}&\leq C_2 \|\eta-\wt\eta\|_{5}+\frac{C}{\kk^2}\|(\xi,\eta)-(\wt\xi,\wt\eta)\|_\times, \\
\left\|\FF^\aux_2 [\xi,\eta]-\wh\FF^\aux_2[\wt\xi,\wt\eta]\right\|_{5}&\leq \frac{C}{\kk^2}\|(\xi,\eta)-(\wt\xi,\wt\eta)\|_\times, \\
\left\|\pa_x\FF^\aux_2 [\xi,\eta]-\pa_x\FF^\aux_2[\wt\xi,\wt\eta]\right\|_{5}&\leq \frac{C}{\eps\kk^2}\|(\xi,\eta)-(\wt\xi,\wt\eta)\|_\times,
 \end{split} 
 \]
 for some constant $C=C(R)>0$ independent of $\eps$ and $\kk$.
\end{lemma}
The proofs of Lemmas \ref{lemma:Fhat0:aux} and \ref{lemma:Fhat:aux}, using Lemmas~\ref{lemma:operators:aux} and~\ref{lemma:initialcondition} follow exactly the same lines as Lemma~\ref{lemma:Fhat0} and~\ref{lemma:Fhat} and are left to the reader. 

As in Section~\ref{sec:outer}, the Lipschitz constant for $\FF^\aux$ obtained in Lemma \ref{lemma:Fhat:aux} is not smaller than one. To overcome this problem we use Theorem~\ref{thm:fixedpoint} to establish that $\FF^\aux$ has a unique fixed point $(\xi^\aux,\eta^\aux)$ belonging to the ball $B(3(C_2+1) b_2 \eps^2)$.  

Let $\wt \xi^\aux, \wt \eta^\aux$ be such that 
\[
\wt \xi^\aux(x)=\xi^{\aux}(-x), \qquad \wt \eta^\aux(x)=\eta^\aux(-x).
\]
It is clear that 
$(\wt \xi^\aux, \wt \eta^\aux) \in B(3(C_2+1) b_2 \eps^2)$ provided the auxiliary domain $D^\aux_\kappa$ is symmetric with respect to $\{\Re x =0\}$ and $\{\Im x=0\}$. Therefore, by uniqueness of the solution of the fixed point equation $(\xi, \eta)= \mathcal{F}^\aux[\xi,\eta]$, in the ball 
$B(3(C_2+1) b_2 \eps^2)$, 
in order to finish the proof of Proposition~\ref{prop:aux:proof}, we only need to argue that $(\wt \xi^\aux , \wt \eta^\aux )$ is also a solution of this fixed point equation.  
For that we emphasize that  
\[
\FF^\aux_1[\wt \xi^\aux,\wt \eta^\aux](x)=\FF^{\aux}_1[\xi^\aux, \eta^\aux](-x).
\] 
Indeed, from definition~\eqref{eq:outer:RHSoperator}, 
\[
\FF_1[\wt \xi^\aux, \wt \eta^\aux](x) = \FF_1[\xi^\aux,\eta^\aux](-x), \qquad 
\FF_2[\wt \xi^\aux, \wt \eta^\aux](x) = \FF_2[\xi^\aux,\eta^\aux](-x)
\]
and from definition~\eqref{def:intoperators-aux} of $\GG_1^\aux,\GG_2^\aux$ and Lemma~\ref{lem:zeta2:welldefined:aux}, denoting $\wt h(x) = h(-x)$, we easily prove that 
\[
\GG_1^\aux[\wt h](x) = \GG_1^\aux[h](-x), \qquad \GG_2^\aux[\wt h](x) = \GG_2^\aux[h](-x).
\]  
In addition, it follows from definition~\eqref{defGfirstintegral} of $\wt G$ and Lemma~\ref{lemma:initialcondition} that $\mathcal{I}[\wt \eta^\aux] = \mathcal{I}[\eta^{\aux}]$
provided
\[
\wt G (\xi_0,0,\eta_0, \eta_0',0)=\wt G(\xi_0,0,\eta_0,-\eta_0',0), \qquad \forall \xi_0,\eta_0,\eta_0' \in \mathbb{R}.  
\]
This completes the proof of Proposition \ref{prop:aux:proof}.
 
\section{The  inner equation}\label{sec:inner}

Here we prove Theorem \ref{thm:inner} with item~\ref{firstitem:thm:inner} proved in Section \ref{sec:proofinner1} and item~\ref{seconditem:thm:inner} proved in Section \ref{sec:proofinner2}.

\subsection{The solutions of the inner equation}\label{sec:proofinner1}

Given $\ell\geq0$ and an analytic function $f: D^{\uns,\inn}_{\theta,\kappa}\rightarrow \CC$, where $D^{\uns,\inn}_{\theta,\kappa}$ is given in~\eqref{innerdomainsol}, consider the norm
\begin{equation}\label{def:innernorm}
\|f\|_{\ell}=\displaystyle\sup_{z\in D^{\uns,\inn}_{\theta,\kappa}}|z^{\ell}f(z)|,
\end{equation}
and the Banach spaces 
\[
\begin{split}
\mathcal{X}_{\ell}&=\{f:D^{\uns,\inn}_{\theta,\kappa}\rightarrow \CC;\ f \textrm{ is a real-analytic function and } \|f\|_{\ell}<\infty \},\\
\mathcal{DX}_{\ell}&=\{f:D^{\uns,\inn}_{\theta,\kappa}\rightarrow \CC;\ f \textrm{ is a real-analytic function and } \|f\|_{\ell}+\|f'\|_{\ell+1}<\infty \}.
\end{split}
\]
We also define the product space
\[
\mathcal{X}_\times=\mathcal{D}\mathcal{X}_{3} \times \mathcal{X}_5
\]
endowed with the norm
\[
\|(\phi,\psi)\|_\times = \max\big \{\|\phi\|_3+ \|\phi'\|_4, \|\psi\|_5\big \}.
\]
The proof of the following lemma can be found in~\cite{Baldoma06}.

\begin{lemma}\label{lemma:propertiesnorm2}
Given  analytic functions $g,h:D^{\uns,\inn}_{\theta,\kappa}\rightarrow\CC$, the following statements hold for some constant $M > 0$ depending only on $\theta$, 
	\begin{enumerate}
		\item If $\ell_1\geq\ell_2\geq 0$, then 		$$\|h\|_{\ell_1-\ell_2}\leq \dfrac{M}{\kappa^{\ell_2}}\|h\|_{\ell_1}.$$			
		\item If $\ell_1,\ell_2\geq 0$, and $\|g\|_{\ell_1},\|h\|_{\ell_2}<\infty$, then
		$$\|g h\|_{\ell_1+\ell_2}\leq \|g\|_{\ell_1}\|h\|_{\ell_2}.$$	
				\item If $h\in\mathcal{X}_{\ell}$ (with respect to the inner domain $ D^{\uns,\inn}_{\theta,\kappa}$), then $\partial_zh\in\mathcal{X}_{\ell+1}$ (with respect to the inner domain $ D^{\uns,\inn}_{2\theta,4\kappa}$), and
		$$\|\partial_zh\|_{\ell+1}\leq M\|h\|_{\ell}.$$		
		\end{enumerate}	
\end{lemma}

The first item in Theorem~\ref{thm:inner} is now rewritten as the following proposition.

\begin{proposition}
    Consider system~\eqref{eq:inner}, namely
    \begin{equation}\label{eq:inner:proof}
    \mathcal{L}_1^\inn[\phi]=\Nin_1[\phi,\psi], \qquad \mathcal{L}_2^\inn[\psi]=\Nin_2[\phi,\psi]
    \end{equation}
    with $\mathcal{L}_1^\inn, \mathcal{L}_2^\inn$ defined in~\eqref{def:diffoperators:inner} and $\Nin_1,\Nin_2$ in~\eqref{eq:inner:RHSoperator}. There exists $\kappa_0$ big enough and a constant $M_7 > 0$ such that for $\kappa>\kappa_0$, equations~\eqref{eq:inner:proof} have solutions $(\phi^{0,\uns},\psi^{0,\uns}) \in \mathcal{X}_\times$ with $\|(\phi^{0,\uns},\psi^{0,\uns})\|_\times \leq M_7 $.
\end{proposition}

As in Sections~\ref{sec:outer} and~\ref{sec:aux}, the suitable right inverse of the linear operators $\mathcal{L}_1^\inn,\mathcal{L}_2^\inn$ are given by the linear operators
\begin{equation}\label{def:J}
\begin{split}
\mathcal{G}_1^\inn[h](z)=&\,\dfrac{z^3}{5}\displaystyle\int_{-\infty}^z\dfrac{h(s)}{s^2}ds-\dfrac{1}{5z^2}\displaystyle\int_{-\infty}^z s^3h(s)ds
\\
\mathcal{G}_2^\inn[h](z)= &\,\dfrac{1}{2i}\displaystyle\int_{-\infty}^z e^{-i(s-z)}h(s)ds-\dfrac{1}{2i}\displaystyle\int_{-\infty}^z e^{i(s-z)}h(s)ds.
\end{split}
\end{equation}
The following lemma provides bounds for the linear operator $\mathcal{G}_{1,2}^{\mathrm{in}}$. Its proof is straightforward
from Proposition 5.2 in \cite{GomideGSZ22} (see also \cite{Baldoma06,BaldomaGG22,BaldomaS08}).

\begin{lemma}\label{lemma:operatorsinner}
	Consider $\kappa\geq 1$ big enough. Given $\ell> 2$, the operators $\mathcal{G}_1^\inn:\mathcal{X}_{\ell+2}\rightarrow \mathcal{X}_{\ell}$ and $\mathcal{G}_2^\inn:\mathcal{X}_{\ell}\rightarrow \mathcal{X}_{\ell}$ are well defined and the following statements hold.
	\begin{enumerate}
		\item $\mathcal{G}_i^\inn \circ\LL_i^\inn[h]=\LL_i^\inn\circ\mathcal{G}_i^\inn[h]=h$, $i=1,2$.
		\item For any $\ell >4$, there exists a constant $M>0$ independent of $\kappa$ such that, for every $h \in\mathcal{X}_{\ell}$,
		\begin{equation*}
            \begin{split}
  		\left\| \mathcal{G}_1^\inn[h]\right\|_{\ell-2}&\leq M\|h\|_{\ell},\\
              \left\|\pa_z\mathcal{G}_1^\inn[h]\right\|_{\ell-1}&\leq M\|h\|_{\ell}.
            \end{split}
            \end{equation*}
		\item For any $\ell >1$, there exists a constant $M>0$ independent of $\kappa$ such that, for every $h \in\mathcal{X}_{\ell}$,
		\begin{equation*}
            \left\|\mathcal{G}_2^\inn[h]\right\|_{\ell}\leq M\|h\|_{\ell}.	
  \end{equation*}		
		\end{enumerate}
	\end{lemma}

We use the integral operators in \eqref{def:J} in order to obtain solutions of \eqref{eq:inner:proof} with certain decay as $|z|\to \infty$ (within $D^{\us,\inn}_{\theta,\kappa}$, $\us=\uns,\sta$). Indeed, such solutions must be fixed points of the operator
\begin{equation}\label{def:operatorFhatInner}
  \FF^\inn= \big(\GG_1^\inn\circ\Nin_1 ,\GG_2^\inn\circ\Nin_2 \big ),
\end{equation}
where the operators $\Nin_1, \Nin_2$ are those introduced in~\eqref{eq:inner:RHSoperator}.

The following two lemmas give properties of the operator $ \FF^\inn$ when analyzed in the Banach space $\XX_\times=\mathcal{DX}_{3}\times\mathcal{X}_{5}$. 
The proofs of these two lemmas are straightforward using the definition of $\Nin_1$ and $\Nin_2$ in \eqref{eq:inner:RHSoperator}, see  \eqref{def:operatorFhatInner}, and Lemmas \ref{lemma:operatorsinner} and \ref{lemma:propertiesnorm2}.

\begin{lemma}\label{lemma:Fhat0:inner}
There exists a constant $b_3>0$ independent of $\kappa$ such that 
\[
 \|\FF^\inn[0,0]\|_\times\leq b_3.
\]
\end{lemma}
 
\begin{lemma}\label{lemma:Fhat:inner}
There exists $C_3>0$ such that for all $R>0$, if $(\phi,\psi),(\wt \phi, \wt \psi)\in B(R)\subset \XX_\times$, the operator $\FF^\inn$ in \eqref{def:operatorFhatInner} satisfies
\[
\begin{split}
 \left\|\FF_1^\inn[\phi,\psi]-\FF_1^\inn[\wt \phi,\wt \psi]\right\|_{3}&\leq C_3\|\psi-\psi'\|_{5}+\frac{C}{\kk^2}\|(\phi,\psi)-(\wt \phi, \wt \psi)\|_\inn,\\
 \left\|\pa_z\FF_1^\inn[\phi,\psi]-\pa_z\FF_1^\inn[\wt \phi, \wt \psi]\right\|_{4}&\leq C_3\|\psi-\psi'\|_{5}+\frac{C}{\kk^2}\|(\phi,\psi)-(\wt \phi,\wt \psi)\|_\inn,\\
\left\|\FF_2^\inn[\phi,\psi]-\FF_2^\inn[\wt \phi, \wt \psi]\right\|_{5}&\leq \frac{C}{\kk^2}\|(\phi,\psi)-(\wt \phi, \wt \psi)\|_\inn,
 \end{split} 
 \]
 for some constant $C=C(R)>0$ independent of $\kk$.
\end{lemma}

We use again Theorem~\ref{thm:fixedpoint} to conclude the existence of a fixed point of $(\phi,\psi)=\mathcal{F}^\inn[\phi,\psi]$ belonging to $B(3(C_3+1)b_3) \subset \mathcal{X}_\times$. This fixed point is the function given in item~\ref{firstitem:thm:inner} of Theorem \ref{thm:inner}. Moreover, by construction it satisfies the stated estimates and they are real analytic functions. The symmetry is a consequence of the reversibility of equation  \eqref{eq:inner} with respect to \eqref{def:symmetryinner}

\subsection{The difference between the solutions of the inner equation}\label{sec:proofinner2}

To complete the proof of Theorem \ref{thm:inner}, we analyze the differences 
$$
\Delta\phi^0(z)= \phi^{0,\uns}(z)-\phi^{0,\sta}(z),\qquad \Delta\psi^0(z)= \psi^{0,\uns}(z)-\psi^{0,\sta}(z),
$$
for $z\in \mathcal{R}^{\inn,+}_{\theta,\kappa}$ with 
$$\mathcal{R}^{\inn,+}_{\theta,\kappa}= D^{\uns,\inn}_{\theta,\kappa}\cap D^{\sta, \inn}_{\theta,\kappa}\cap\{ z\in i\R \textrm{ and }\Im(z)<0 \}.
$$

Given an analytic function $f:\mathcal{R}^{\inn,+}_{\theta,\kappa}\rightarrow\CC$, we define the norm
\begin{equation*}
\label{expnorm}
\|f\|_{\ell,\exp}=\displaystyle\sup_{z\in\mathcal{R}^{\inn,+}_{\theta,\kappa}}|z^{\ell}e^{iz} f(z)|
\end{equation*}
and the Banach spaces
\begin{equation*}
\begin{split}
\mathcal{Z}_{\ell,\exp}&=\left\{f:\mathcal{R}^{\inn,+}_{\theta,\kappa}\rightarrow\CC;\ \|f\|_{\ell,\exp}<\infty\right\},\\
\mathcal{DZ}_{\ell,\exp}&=\left\{f:\mathcal{R}^{\inn,+}_{\theta,\kappa}\rightarrow\CC;\ \|f\|_{\ell,\exp} +\|f'\|_{\ell,\exp}<\infty\right\}.
\end{split}
\end{equation*}
We will consider the product Banach space 
\begin{equation*}
\mathcal{Z}_{\times,\exp}=\mathcal{DZ}_{0,\exp}\times \mathcal{Z}_{0,\exp}
\end{equation*}
and denote by $\|\cdot \|_{\times,\exp}$ the associated norm:
\[
\|(\phi,\psi)\|_{\times ,\eps} = \max\{\|\phi\|_{0,\exp} + \|\phi'\|_{0,\exp}, \|\psi\|_{0,\exp}\}.
\]
It can be easily seen that, if $f\in\XX_{\ell_1}$ and $g\in\ZZZ_{\ell_2,\exp}$, then $fg\in\ZZZ_{\ell_1+\ell_2,\exp}$ and $\|fg\|_{\ell_1+\ell_2,\exp}\leq \|f\|_{\ell_1} \|g\|_{\ell_2,\exp}$.

The second item in Theorem~\ref{thm:inner} can be rewritten as the following proposition, which will be proved in the rest of this section.

\begin{proposition}\label{prop:difinner}
There exist $\Theta \in \mathbb{R}$ and $\kappa_0, M_8>0$  such that for $\kappa>\kappa_0$, $\Delta\phi^0, \Delta\psi^0 \in \mathcal{D}\mathcal{Z}_{0,\exp}$ and they satisfy
\begin{align*}
 \|\Delta \phi^0 +\Theta e^{-iz}\|_{1,\exp} + \|\partial_z \Delta \phi^0-i\Theta e^{-iz}\|_{1,\exp}&\leq M_8 |\Theta|, \\ 
 \|\Delta \psi^0 -\Theta e^{-iz}\|_{1,\exp} + \|\partial_z \Delta \psi^0+i \Theta e^{-iz}\|_{1,\exp}& \leq M_8 |\Theta|.
\end{align*} 
\end{proposition}

Since both the stable and unstable solutions satisfy equation~\eqref{eq:inner:proof}, applying the mean value theorem, one can see that the functions $\Delta\phi^0$, $\Delta\psi^0$ satisfy a linear homogeneous equation of the form
\begin{equation}\label{eq:inner:diff}
\left\{ \begin{array}{l}  
\wt\LL_1^\inn\Delta\phi^0 =\PP_1[\Delta\phi^0,\Delta\psi^0], \\
\LL_2^\inn\Delta\psi^0 =\PP_2[\Delta\phi^0,\Delta\psi^0 ],  
\end{array} \right.
\end{equation}
where $\wt\LL_1^\inn=-\pa^2_z$, $\LL_2^\inn$ is the operator introduced in \eqref{def:diffoperators:inner}  and $\PP_1$, $\PP_2$ are defined by
\begin{align}\label{def:innerdiff:operatorRHS}
\left\{ \begin{array}{l} 
\PP_1[\Delta\phi^0,\Delta\psi^0](z) =a_{11}(z)\Delta\phi^0(z)-\Delta\psi^0(z),
\vspace{0.2cm} \\
\PP_2[\Delta\phi^0,\Delta\psi^0](z) = a_{21}(z)\Delta\phi^0(z)+a_{22}(z)\Delta\psi^0(z)+a_{23}(z)\pa_z\Delta\phi^0(z), \end{array} \right.
\end{align}
where, introducing $\Phi^{0,\star}=(\phi^{0,\star},\psi^{0,\star})$, $\star=\uns,\sta$ and defining 
$N$ as the functional such that the operator $\Nin_2 [\phi,\psi]$ in~\eqref{eq:inner:RHSoperator} can be written as
\[
\Nin_2[\phi,\psi]=N[\phi,\psi, \pa_z\phi],
\]
$a_{i,j}$ is defined as
\begin{equation*}
\begin{aligned}
a_{11}(z)= &-\frac{6}{z^2}+\int_{0}^1D_1 \Nin_1[\Phi^{0,\sta}(z)+\sigma (\Phi^{0,\uns}(z)-\Phi^{0,\sta}(z)) ]d\sigma,\\
a_{2j}(z)=& \int_{0}^1D_j N\big [\Phi^{0,\sta}(z)  +\sigma (\Phi^{0,\uns}(z)-\Phi^{0,\sta}(z)),\\ & \qquad \qquad \pa_z\phi^{0,\sta}(z)+\sigma(\pa_z\phi^{0,\uns}(z)-\pa_z\phi^{0,\sta}(z))\big ]d\sigma.
\end{aligned}
\end{equation*}
Using the norm introduced in \eqref{def:innernorm}, these functions satisfy
\begin{equation}\label{def:innerdiff:aestimates}
\|a_{11}\|_2\lesssim 1,\quad \|a_{21}\|_4\lesssim 1,\quad \|a_{22}\|_2\lesssim 1,\quad \|a_{23}\|_3\lesssim 1.
\end{equation}

We now write equation \eqref{eq:inner:diff} as an integral fixed point equation. 
On the one hand,
\[
\partial_z \Delta \phi^0 (z)=  C_1 - \int_{z_1}^z  \mathcal{P}_1[\Delta \phi^0,\Delta \psi^0](s) ds
\]
with $C_1= \partial_z \Delta \phi^0(z_1)$. Since $\lim_{\Im z \to -\infty} \partial_z \Delta \phi^0(z)=0$, we conclude that 
\[
\partial_z \Delta \phi^0(z) = - \int_{-i\infty}^z  \mathcal{P}_1[\Delta \phi^0,\Delta \psi^0](s)  ds 
\]
and as a consequence, reasoning analogously,
\begin{equation}\label{redef:deltaphi0}
\Delta \phi^0(z) = \int_{-i\infty}^{z} \int_{-i\infty}^s \mathcal{P}_1[\Delta \phi^0,\Delta \psi^0](\sigma) d\sigma.
\end{equation}
On the other hand, recalling that $\mathcal{L}^\inn_2 [\Delta \psi^0]= \partial_z^2 \Delta \psi^0 + \Delta\psi^0$, we have 
\begin{align*}
\Delta\psi^0(z) & = e^{iz} \left (C_1 + \frac{1}{2i
}\int_{z_1}^z e^{-is} h(s) ds \right )  + e^{-iz}\left ( C_2- \int_{z_2}^z e^{is}h(s) ds\right ) 
\end{align*}
with 
\[
2i e^{iz_1} C_1 = i \Delta \psi^0(z_1) + \partial_z \Delta \psi^0(z_1), \qquad 
2i e^{-iz_2} C_2 = i \Delta \psi^2(z_2) - \partial_z \Delta \psi^0(z_2),
\]
Using~\eqref{def:innerdiff:aestimates}, taking $z_2 =-i\kappa$ and imposing that $\lim_{\Im z \to -\infty} \Delta \psi^0(z) =0$, we obtain 
\begin{equation}\label{redef:deltapsi0}
\begin{aligned}
\Delta \psi^0(z)=    &\int_{-i\infty}^z \frac{e^{-i(s-z)}}{2i} \PP_2[\Delta \phi^0, \Delta \psi^0](s) ds   + \Theta_0e^{-iz} \\ & -\int_{-i\kappa}^z \frac{e^{i(s-z)}}{2i}\PP_2[\Delta \phi^0, \Delta \psi^0](s) ds 
\end{aligned}
\end{equation}
with 
\begin{equation}\label{def:Theta0:inner}
\Theta_0 = \Theta_0(\kappa)=\frac{1}{2i} e^{\kappa} \big ( i \Delta \psi^0 (-i\kappa)- \partial_z \Delta \psi^0 (-i\kappa) \big ).
\end{equation}
We emphasize that, from item~\ref{firstitem:thm:inner} of Theorem~\ref{thm:inner}, $|\Delta \phi^0 (z)| \lesssim |z|^{-3}$, $|\Delta \psi^0 (z)|\lesssim |z|^{-5}$ uniformly on the domain $\mathcal{R}_{\theta,\kappa}^\inn$ and hence, using also bounds~\eqref{def:innerdiff:aestimates} of $a_{ij}$, the improper integrals in~\eqref{redef:deltaphi0} and~\eqref{redef:deltapsi0} are well defined. 
Therefore, $(\Delta \phi^0, \Delta \psi^0)$ satisfies the fixed point equation 
\begin{equation}\label{eq:IntegralForInnerDif}
\left\{ \begin{array}{l}
\Delta\phi^0(z) =\Gindif_1 \circ \PP_1[\Delta\phi^0,\Delta\psi^0](z),\\
\Delta\psi^0(z) =\Theta_0 e^{-iz}+ \Gindif_2\circ \PP_2[\Delta\phi^0,\Delta\psi^0](z).
\end{array} \right.
\end{equation}
where the constant $\Theta_0 = \Theta_0(\kappa)$ is defined in~\eqref{def:Theta0:inner}, $\PP$ in~\eqref{def:innerdiff:operatorRHS} and $\Gindif=(\Gindif_!,\Gindif_2)$
is the integral linear operator defined on functions $h:\mathcal{R}^{\inn,+}_{\theta,\kappa}\rightarrow \CC$, as
\[
\begin{aligned}    
\Gindif_{1}[h](z)&=-\int_{-i\infty}^z\int_{-i\infty}^s h(\sigma)d\sigma\, d s,\\
\Gindif_{2}[h](z)&=\displaystyle\int_{-i\infty}^z\dfrac{e^{-i(s-z)}h(s)}{2i}ds-\displaystyle\int_{-i\kappa}^z\dfrac{e^{i(s-z)}h(s)}{2i}ds.
\end{aligned}
\]

Denoting $\Delta \Phi^0 = (\Delta \phi^0, \Delta \psi^0)$, equation~\eqref{eq:IntegralForInnerDif} can be rewritten as
\[  
\Delta \Phi^0  = \Delta \Phi^0_0+ \wt \PP [ \Delta \Phi^0] , \qquad \Delta \Phi^0_0(z) = \begin{pmatrix}
    0 \\ \Theta_0 e^{-iz} 
\end{pmatrix}, 
\] 
where $\wt \PP$ is the linear operator defined by
\begin{equation}
	\label{Btilop}
	\wt{\mathcal{P}}=\big (\wt\PP_1, 
\wt\PP_2\big )=\big (\Gindif_1\circ\PP_1, 
\Gindif_2\circ\mathcal{P}_2\big ).
	\end{equation}
Notice that, if the operator $\mathrm{Id} - \wt \PP$ were invertible, then we could write $\Delta \Phi^0 = \big ( \mathrm{Id} - \wt \PP \big )^{-1} [\Delta \Phi^0_0]$ and study $\Delta \Phi^0$ through $\wt \PP$ and $\Delta \Phi^0$. 

The following lemma specifies properties of the linear operator $\wt{\mathcal{P}}$. Its proof is straightforward using the estimates in~\eqref{def:innerdiff:aestimates} and the definition of the operators in \eqref{Btilop}, where we also recall that $\mathcal{R}_{\theta,\kappa}^\inn$ is a subset of $i\mathbb{R}$. 

\begin{lemma}\label{lemma:IntOpInnerDiff}
	The linear operator $\wt{\mathcal{P}}:\mathcal{Z}_{\times,\exp}\rightarrow\mathcal{Z}_{\times,\exp}$ given in~\eqref{Btilop},
  is well defined. Moreover, there exists a constant $M$ such that for each $\kappa\geq 1$,
	\begin{enumerate}
		\item The linear operators $\wt{\PP}_1, \pa_z\wt{\PP}_1:\mathcal{Z}_{\times,\exp}\rightarrow\mathcal{Z}_{0,\exp}$ satisfy	
		\begin{equation*}
\begin{split}
\|\wt{\PP}_1[\Delta\phi^0,\Delta\psi^0]  \|_{0,\exp} &\leq\frac{M}{\kk^2}\|\Delta\phi^0\|_{0,\exp} +M\|\Delta\psi^0\|_{0,\exp},  \\ 
 \|\partial_z \wt{\PP}_1[\Delta\phi^0,\Delta\psi^0]  \|_{0,\exp} &\leq\frac{M}{\kk^2}\|\Delta\phi^0\|_{0,\exp} +M\|\Delta\psi^0\|_{0,\exp}.
\end{split}		
  \end{equation*}
  \item The linear operator $\wt{\PP}_2:\mathcal{Z}_{\times,\exp}\rightarrow\mathcal{Z}_{0,\exp}$ satisfy	
  \begin{equation*}
\|\wt{\PP}_2[\Delta\phi^0,\Delta\psi^0]\|_{0,\exp}\leq \dfrac{M}{ \kappa}\|(\Delta\phi^0,\Delta\psi^0)  \|_{0,\exp}.	
		\end{equation*}
\end{enumerate}
\end{lemma}

This result of Lemma \ref{lemma:IntOpInnerDiff} does not lead to check that $\wt \PP$ has small norm so that $\mathrm{Id} - \wt \PP$ is invertible. Hence we proceed in a similar way as in the proof of Theorem~\ref{thm:fixedpoint}. We emphasize that $\Delta \Phi_0$ is also a solution of
\begin{equation}
\label{eq:IntegralForInnerDif:new}
\Delta \Phi^0 = \wh{\Delta \Phi^0_0}+ \wh \PP [ \Delta \Phi^0] , \qquad \wh{\Delta \Phi^0_0}(z) =  \Delta \Phi_0^0 (z)+ \begin{pmatrix}
    \wt \PP_1[\Delta \Phi^0_0] \\0
\end{pmatrix},
\end{equation}
where $\wh \PP$ is the linear operator defined by 
\begin{equation*}
\left\{ \begin{array}{l}
\wh\PP_1[\Delta\phi^0, \Delta \psi^0] =\wt\PP_1 \big [\Delta \phi^0, \wt\PP_2[\Delta\phi^0, \Delta\psi^0 ]\big ], \vspace{0.2cm}\\
\wh\PP_2[\Delta\phi^0, \Delta \psi^0 ] =  \wt\PP_2[\Delta\phi^0  \Delta \psi^0 ].
\end{array} \right. 
\end{equation*}
Lemma \ref{lemma:IntOpInnerDiff} implies that $\wh \PP$ satisfies
\[
\left\|\wh \PP[\Delta\phi^0,\Delta\psi^0]\right\|_{\times,\exp}\lesssim\frac{1}{\kk}\left\|\Delta\phi^0,\Delta\psi^0\right\|_{\times,\exp}.
\]
Then we conclude that, taking $\kappa$ big enough, $\mathrm{Id}-\wh \PP $ is invertible in $\ZZZ_{\times,\exp}$. On the other hand, using that $\Delta_0^0(z) = (0, \Theta_0 e^{-iz})^\top$, formula~\eqref{def:innerdiff:operatorRHS} of $\PP_1$ and that $\wt \PP_1 = \Gindif_1\circ \PP_1$, we obtain that 
\begin{equation}\label{first:approximation:difference}
\wh{\Delta \Phi_0^0}(z)=\begin{pmatrix}
    \wt\PP_1[\Delta \Phi_0^0](z) \\ \Theta_0 e^{-iz} 
\end{pmatrix}= \begin{pmatrix}
    -\Theta_0 e^{-iz} \\ \Theta_0 e^{-iz}
\end{pmatrix} \in \ZZZ_{\times,\exp}.
\end{equation}
As a consequence, it follows from equation~\eqref{eq:IntegralForInnerDif:new} that 
$\big ( \mathrm{Id} - \wh\PP \big ) \Delta\Phi^0 = \wh{\Delta\Phi_0^0} \in \ZZZ_{\times,\exp}$ and we conclude
\[
\Delta \Phi^0 = \big ( \mathrm{Id} - \wh\PP \big )^{-1}[ \wh{\Delta \Phi^0_0}] \in \ZZZ_{\times,\exp}. 
\]
In addition, this implies that, for $z\in \mathcal{R}^{\inn,+}_{\theta,\kappa}$,
\[
\begin{pmatrix}
\Delta\phi^0(z)\\ \Delta\psi^0(z)\end{pmatrix}
=
\Theta_0 e^{-iz}\begin{pmatrix}-1+\OO\left(\frac{1}{ {\kk}}\right)\\ 1+\OO\left(\frac{1}{\kk}\right)\end{pmatrix}.
\]
Note that this asymptotic formula is not the one given in Proposition~\ref{prop:difinner}. Indeed, the asymptotics here is given with respect to $\kk^{-1}$ whereas the one in Proposition~\ref{prop:difinner}  is given in terms of $z^{-1}$. To improve the asymptotics, we need to define a new constant $\Theta$ which is $\kk^{-1}$ close to $\Theta_0$.

We define the constant
\begin{equation}\label{def:StokesContantDef}
\Theta=\Theta_0-\int_{-i\kappa}^{-i\infty}\dfrac{e^{iz}\PP_2[\Delta\phi^0,\Delta\psi^0] (z)}{2i}dz.
\end{equation}
Note that the fact that  $(\Delta\phi^0,\Delta\psi^0)\in\ZZZ_{\times,\exp}$ implies that the integral is convergent and the constant $\Theta$ is well-defined.

Proposition~\ref{prop:difinner} (and hence the second statement of Theorem~\ref{thm:inner}) is a direct consequence of the following lemma.

\begin{lemma}\label{lemma:FinalAsymptotics}
The functions $(\Delta\phi^0,\Delta\psi^0)$ satisfy that, for $z\in \mathcal{R}^{\inn,+}_{\theta,\kappa}$,
\[
\begin{pmatrix}
\Delta\phi^0(z)\\ \Delta\psi^0(z)\end{pmatrix}
=
\Theta e^{-iz}\begin{pmatrix}-1+\OO\left(\frac{1}{z}\right)\\ 1+\OO\left(\frac{1}{z}\right)\end{pmatrix},
\]
for some constant $\Theta \in \mathbb{R}$.  
\end{lemma}

\begin{proof}
We exploit the fact that we already have proven that $(\Delta \phi^0, \Delta \psi^0) \in \ZZZ_{\times ,\exp}$.
We obtain the asymptotic formula for each component. From~\eqref{eq:IntegralForInnerDif} and using definition~\eqref{def:StokesContantDef} of $\Theta$, we note that, the second component can be written as   
\[
\Delta\psi^0(z)=\Theta e^{-iz}+\check{\GG}_2^\inn\big [\PP_2[\Delta\phi^0,\Delta\psi^0]\big ](z),  
\]
with 
\begin{equation*}
\check{\GG}^{\inn}_2[h](z)=\displaystyle\int_{-i\infty}^z\dfrac{e^{-i(s-z)}h(s)}{2i}ds-\displaystyle\int_{-i\infty}^z\dfrac{e^{i(s-z)}h(s)}{2i}ds.
\end{equation*}
Since $(\Delta\phi^0,\Delta\psi^0)\in\ZZZ_{\times,\exp}$, estimates~\eqref{def:innerdiff:aestimates} imply that $\PP_2[\Delta\phi^0,\Delta\psi^0] \in \ZZZ_{2, \exp}$ and 
\[
\left\|\PP_2[\Delta\phi^0,\Delta\psi^0] \right\|_{2,\exp}\lesssim 1.
\]
Then, it is a straightforward computation to see that $\Delta\psi^0-\Theta e^{-iz}\in  \ZZZ_{1,\exp}$ and 
\[
\left\|\Delta\psi^0-\Theta e^{-iz} \right\|_{1,\exp}=\left\|\check{\GG}_2^\inn \big [\PP_2[\Delta\phi^0,\Delta\psi^0]\big ]\right\|_{1,\exp}\lesssim 1.
\]  
This completes the proof of the asymptotic formula for $\Delta\psi^0$. Analogous computations lead to the asymptotic formula for $\partial_z \Delta \psi^0$. 

Now we prove the asymptotic formula for the first component. To this end, using that we rewrite the identity (see~\eqref{eq:IntegralForInnerDif:new} and~\eqref{first:approximation:difference})
\[
\Delta\phi^0(z)=\wt \PP_1[\Delta \Phi_0^0](z) + \wh\PP_1[\Delta\phi^0,\Delta\psi^0](z)=-\Theta_0e^{-iz}+\wt\PP_1\big [\Delta\phi^0,\wt\PP_2[\Delta\phi^0,\Delta\psi^0]\big ](z)
\]
as  
\[
\Delta\phi^0(z)=-\Theta e^{-iz}+\wt\PP_1\big [\Delta\phi^0,\check{\GG}_2^\inn\big [\PP_2[\Delta\phi^0,\Delta\psi^0]\big ]\big ](z),
\]
where we have used 
\[
\Delta\psi_0 (z) = \Theta_0 e^{-iz} + \wt \PP_2[\Delta\phi^0,\Delta \psi^0 ] (z)=  \Theta e^{-iz} + \check{\GG}_2^\inn \big [\PP_2[\Delta\phi^0,\Delta \psi^0] \big ] (z).
\]
Then, it can be easily seen that 
\[
\Delta\phi^0(z)+\Theta e^{-iz}=\wt\PP_1\big [\Delta\phi^0,\check{\GG}_2^\inn \big [\PP_2[\Delta\phi^0,\Delta\psi^0]\big ]\big ]\in \ZZZ_{1,\exp}
\] 
and
\[
\left\|\Delta\phi^0+\Theta e^{-iz}\right\|_{1,\exp}\lesssim 1.
\]
This completes the asymptotic formula for the first component and analogously we have the one for its derivative.

It only remains to show that the constant $\Theta$ is real. This is a direct consequence of the fact that the solutions $(\phi^{0,\star}, \psi^{0,\star})$, $\star=\uns,\sta$ are real-analytic and satisfy \eqref{def:symmetryinnersol}. Indeed these two properties  imply that, for $z\in \mathcal{R}^{\inn,+}_{\theta,\kappa}$ (recall that $\mathcal{R}^{\inn,+}_{\theta,\kappa}\subset i\RR$), 
\[
\Delta \psi^0(z)\in\RR.
\]
This implies that  $e^{iz}\Delta \psi^0(z)\in\RR$ and therefore $\Theta\in\RR$ since it can be defined as
\[
\lim_{\Im z\to -\infty, z\in i\RR}e^{iz}\Delta \psi^0(z).
\]
This completes the proof of Lemma \ref{lemma:FinalAsymptotics}.  
\end{proof}

Finally, the fact that $\Theta\neq 0$ if and only if $\Delta\phi^0$ does not vanish at one point is  a direct consequence of the asymptotic formula. This proves the third item of Theorem~\ref{thm:inner}.

\section{Matching around  singularities}\label{sec:matching}

Here we prove Theorem~\ref{thm:matching:intro}. We will give the proof only for the $-$ case, being the $+$ case is analogous. Due to this reason, we  omit the sign $\pm$ in our notation and we provide  estimates for $(\xi^{\uns},\eta^{\uns})$ and $(\xi^\aux,\eta^\aux)$ around the singularity $x_-$.

It is convenient to work with inner variables, see~\eqref{def:innervariable} and~\eqref{def:innerfunctions}, namely,  
\begin{equation}\label{inner_variables_matching}
z=\varepsilon^{-1} (x-x_-), \quad \phi(z)=\frac{\varepsilon}{c_{- 1}} \xi(x_-+\varepsilon z ),\quad
\psi(z)=\frac{\varepsilon^3}{c_{- 1}} \eta (x_-+ \varepsilon z).
\end{equation}
We define now the matching domain $D^{-,\match}_{\theta_1,\theta_2,\nu}$  by \eqref{def:domainsmatching} in the inner variable. We fix $0<\nu <1$ and  $0<\theta_2 < \theta <\theta_1<\frac{\pi}{2}$,  where $\theta$ is the angle introduced in \eqref{def:defdomainouter}, and we define 
\[
\mathcal{D}_{\theta_1,\theta_2,\nu}^{\match} = \widehat{-i\kappa, z_1,z_2},
\]
the triangle with vertices $-i\kappa, z_1,z_2$, with 
\begin{equation*}
z_1=-i\kappa +\frac{1}{\varepsilon^{1-\nu}} e^{-i \theta_1}, \qquad z_2 = -i\kappa - \frac{1}{\varepsilon^{1-\nu}} e^{-i\theta_2}.
\end{equation*}
In addition, if we define 
\[
\hat{u}_0(z) = u_0(x_-+\eps z),
\]
we notice that, if $z \in \mathcal{D}_{\theta_1,\theta_2}^{\nu, \match}$, then $|\varepsilon z| \lesssim \varepsilon^{\nu}$ and therefore
\begin{equation}\label{exp_u0_inner} 
    \begin{split}
       \varepsilon c_{-1}^{-1} \hat{u}_0(z)   & = \frac{1}{z} + \varepsilon \sum_{k\geq 0} c_k (\varepsilon z)^k=
\frac{1}{z} + \mathcal{O}(\varepsilon), \\ 
\varepsilon c_{-1}^{-1} \hat{u}'_0(z) &= - \frac{1}{z^2 }  +\mathcal{O}(\varepsilon^2).
    \end{split}  
\end{equation}  
Moreover, defining 
\begin{equation}\label{unstainner_variables}
\phi^{\us}(z) = \frac{\varepsilon}{c_{-1}} \xi^{\us} (x_-+\varepsilon z )  , \qquad 
\psi^{\us}(z) = \frac{\varepsilon^3}{c_{-1}} \eta^{\us} (x_-+\varepsilon z )  , \qquad \us=\uns,\aux
\end{equation}
with $(\xi^{\uns}, \eta^{\uns})$  and $(\xi^\aux, \eta^\aux)$, given in Theorems~\ref{thm:outer:intro} and~\ref{thm:aux:intro} respectively, we have that  
\begin{equation}\label{recall:phipsi}
\big |\phi^\us(z) \big |   \lesssim \frac{1}{|z|^3}, \qquad   \big |\partial_z  \phi^\us(z) \big | \lesssim \frac{1}{|z|^4},
\qquad \big | \psi^\us (z) \big |\lesssim \frac{1}{|z|^5}. 
\end{equation}  
Now we rephrase Theorem~\ref{thm:matching:intro}  in the inner variables as follows. 

\begin{theorem}\label{thm:matching}
    Let $\theta>0, \kappa_0$ be fixed as in Theorems~\ref{thm:outer:intro},  \ref{thm:aux:intro} and  \ref{thm:inner}.
Take $0<\theta_2<\theta <\theta_1< \frac{\pi}{2}$ and $\nu\in (0,1)$. 
We introduce the functions 
\[
\delta \phi^{\us}(z) = \frac{\eps }{ c_{-1}} \delta \xi^\us_-( x_-+\eps z ), \quad 
\delta \psi^{\us}(z) = \frac{\eps^3 }{ c_{-1}} \delta \eta^\us_-( x_-+\eps z), \quad \us=\uns,\aux,
\]
with $\delta\xi^\us_-,\delta \eta^\us_-$ defined in Theorem~\ref{thm:matching:intro}. 
Then there exist $\kappa_1\geq \kappa_0$ and a constant $M>0$ such that for all $\kappa \geq \kappa_1$ and $z \in \mathcal{D}_{\theta_1,\theta_2, \nu}^{\match}$
\begin{align*}
& \big |\delta \phi^\us (z)\big |    \leq   M |\log \eps|\frac{\eps^{1-\nu}}{|z|^2},  & \big | \partial_ z \delta \phi^\us (x)\big | \leq   M |\log \eps|\frac{\eps^{1-\nu}}{|z|^3}, \\
& \big |\delta \eta^\us (x)\big | \leq   M |\log \eps|\frac{\eps^{1-\nu}}{|z|^4},  & \big | \partial_z \delta \eta^\us_- (x)\big | \leq   M |\log \eps|\frac{\eps^{1-\nu}}{|z|^4}.
\end{align*}
\end{theorem}

\begin{remark}\label{rmk:matching1}
    We emphasize that we already know the existence of $\delta \phi^\us, \delta \psi^\us$ in the matching domain and that, using~\eqref{recall:phipsi} and Theorem~\ref{thm:inner}
    \[
    \big |\delta \phi^\us (z) \big | \leq \big | \phi^\us(z)\big | + \big | \phi^{0,\us}(z) \big | \lesssim \frac{1}{|z|^3}, \qquad \big | \delta \psi^\us(z) \big | \leq \big | \phi^\us(z)\big | + \big | \phi^{0,\us}(z) \big |\lesssim \frac{1}{|z|^5},
\]
and also $\big | \partial_z \delta \phi^\us\big | \lesssim |z|^{-4}$. 
    However, these estimates do not imply that, when $\eps=0$, $\delta \phi^\us, \delta \psi^\us=0$.
\end{remark}
The remaining part of this section is devoted to prove Theorem~\ref{thm:matching}. The prove for $\us=\uns, \aux$ are identical and, therefore, we only present the first one.

\subsection{Reformulation of the problem}

To prove Theorem \ref{thm:matching} we look for differential  equations which have  ($\delta \xi^\uns, \delta\eta^\uns)$,  as a solutions. To this end, let ($\xi^\uns, \eta^\uns)$  be the solution of equation~\eqref{eq:perturb} provided in Theorem~\ref{thm:outer:intro} and consider the function $(\phi^\uns, \psi^\uns)$ defined in~\eqref{unstainner_variables}. Applying the change of coordinates to  equation~\eqref{eq:perturb} we have that 

\begin{align*}
\left\{ \begin{array}{l} 
\mathcal{L}^{\inn}_1[\phi^\uns] = \Nch_1[\phi^\uns,\psi^\uns;\eps]:=\Nin_1[\phi^\uns,\psi^\uns] + \mathcal{A}_1[\phi^\uns, \psi^\uns;\eps], \vspace{0.2cm} \\  
\mathcal{L}^{\inn}_2[\psi^\uns] = \Nch_2[\phi^\uns,\psi^\uns;\eps]:=\Nin_2[\phi^\uns,\psi^\uns] + \mathcal{A}_2[\phi^\uns, \psi^\uns;\eps],
\end{array} \right.
\end{align*}
where $\LL_j^\inn$ and $\Nin_j$, $j=1,2$ are introduced in \eqref{def:diffoperators:inner} and \eqref{eq:inner:RHSoperator}.

 We introduce the notation $\Phi=(\phi,\psi)$, $\mathcal{A}[\Phi;\varepsilon]=(\mathcal{A}_1[\Phi;\varepsilon], \mathcal{A}_2[\Phi;\varepsilon])$, 
\[
\mathcal{L}^{\inn} [\Phi] = (\mathcal{L}^{\inn}_1[\phi], \mathcal{L}^{\inn}_2[\psi]), 
\qquad 
\Nin[\Phi] = (\Nin_1[\Phi], \Nin_2[\Phi]),
\]
and
\[
\Nch[\Phi;\varepsilon] = (\Nch_1[\Phi;\varepsilon], \Nch_2[\Phi;\varepsilon]) = \Nin[\Phi]+ \mathcal{A}[\Phi ;\varepsilon].
\]
Since, by Theorem \ref{thm:inner}, $\Phi^{0,\uns}=(\phi^{0,\uns},\psi^{0,\uns})$ is a solution of $\mathcal{L}^{\inn} [\Phi^{0,\uns}] = \Nin[\Phi^{0,\uns}]$ and $\Phi^{\uns}$ satisfies $\mathcal{L}^{\inn}[\Phi^{\uns}] = \Nin[\Phi^{\uns}] + \mathcal{A}[\Phi^{\uns};\varepsilon]$, using the mean value theorem, we have that $\delta \Phi^\uns=\Phi^{\uns} -   \Phi^{0,\uns}$ satisfies
\begin{align*}
\mathcal{L}^{\inn}[\delta \Phi^\uns]=&
\mathcal{L}^{\inn}[\Phi^{\uns}](z) - \mathcal{L}^{\inn}[\Phi^{0,\uns}](z)
\\  = & \int_{0}^1 D_{\Phi}\Nin [\Phi^{0,\uns} + \lambda (\Phi^{\uns} - \Phi^{0,\uns})] (z) \cdot (\Phi^{\uns}(z) - \Phi^{0,\uns}(z)) \, d\lambda + \mathcal{A}[\Phi^{\uns};\varepsilon](z) \\ 
&+ \int_{0}^1 D_{\partial_z \phi} \Nin[\Phi^{0,\uns} + \lambda (\Phi^{\uns} - \Phi^{0,\uns})](z)\cdot\left(\partial_z \phi^{\uns} (z) - \partial_z \phi^{0,\uns} (z) \right)\, d\lambda.
\end{align*} 
We denote 
\begin{equation}\label{def:calB:matching}
\begin{aligned}
\mathcal{B}_1^\uns(z) & = \int_{0}^1 D_{\Phi} \Nin[\Phi^{0,\uns} + \lambda (\Phi^{\uns} - \Phi^{0,\uns})] (z)   \, d\lambda - \left (\begin{array}{cc} 0 & -1 \\ 0 & 0 \end{array}\right ), \\ 
\mathcal{B}_2^\uns(z) & =  \int_{0}^1 D_{\partial_z \phi} \Nin[\Phi^{0,\uns} + \lambda (\Phi^{\uns} - \Phi^{0,\uns})] (z)   \, d\lambda,  \\
\mathcal{B}_3(z) &= \left ( \begin{array}{cc} 0 &  -1 \\ 0 & 0 \end{array}\right ),
\end{aligned}
\end{equation}
and $A^\uns(z) = \mathcal{A}[\Phi^\uns;\eps](z)$. 
We emphasize that $\mathcal{B}_1^\uns, \mathcal{B}_2^\uns$ and $A^\uns$ are known functions that depend on the solutions $\Phi^\uns=(\phi^\uns,\psi^\uns)$ and $\Phi^{0,\uns}=(\phi^{0,\uns}, \psi^{0,\uns})$, which have already been constructed above. We then obtain that   $\delta \Phi^\uns= (\delta \phi^\uns ,\delta \psi^\uns) $ satisfies the non-homogeneous linear equation
\begin{equation}\label{eq_matching_first}
\mathcal{L}^{\inn} [\delta \Phi^\uns] (z)= \mathcal{B}_1^\uns(z) \delta \Phi^\uns(z) + \mathcal{B}_2^\uns(z) \partial_z \delta 
\phi^\uns (z) + \mathcal{B}_3(z) \delta \Phi^\uns(z) + A^\uns(z) .
\end{equation}

The following lemma characterizes the solutions of $\mathcal{L}^\inn [\Phi] = h$ with  given initial conditions. Its proof is straightforward and is omitted. 

\begin{lemma}\label{lem:linear_matching} 
Let $\Phi$ be a solution of $\mathcal{L}^\inn [\Phi]= h$ defined in $\mathcal{D}^{\match}_{\theta_1,\theta_2,\nu}$. Then, $\Phi=(\phi,\psi)$ 
is given by 
$$
\Phi(z) = 
      \left (\begin{array}{cc}
    z^3 a_\phi + \frac{1}{z^2} b_\phi \\    
    e^{i(z-z_1)} a_\psi  + e^{-i(z-z_2)} b_\psi  
    \end{array}\right ) + \GG^\match[h],
    $$ 
where 
\begin{equation}\label{def:lem:abphi}
\begin{aligned}
a_\phi &= \frac{1}{ 5 z_1^3 } \big ( 2 \delta \phi  (z_1)+ \partial_z \delta \phi  (z_1) z_1\big ),  & b_\phi &= \frac{z_1^2}{5} \big (3 \delta \phi (z_1)- \partial_z \delta \phi (z_1) z_1\big ), \\ 
a_\psi &= \frac{1}{2} \big (\delta \psi (z_1) - i \partial_z \delta \psi (z_1)\big ), 
&b_\psi &= \frac{1}{2} \big (\delta \psi (z_2) + i \partial_z \delta \psi (z_2)\big ),
\end{aligned}
\end{equation}
and $\GG^\match[h]= (\GG^\match_1[h_1], \GG^\match_2[h_2] )$ is the linear operator (compare with~\eqref{def:J}) defined by 
\begin{equation}\label{def:lem:mathcalI}
\begin{aligned}
\GG^\match_1 [h](z) &= \frac{z^3}{5} \int^{z}_{z_1} \frac{h(s)}{s^2} ds- \frac{1}{5z^2} \int_{z_1}^z s^3 h(s) ds, \\
\GG^\match_2 [h](z) &=\frac{1}{2i} \int_{z_1}^z e^{-i(s-z)}h(s) ds - \frac{1}{2i}\int_{z_2}^z e^{i(s-z)} h(s) ds.
\end{aligned}
\end{equation}
\end{lemma}

Since $\delta \Phi^\uns$ is a solution of~\eqref{eq_matching_first},  Lemma~\ref{lem:linear_matching} implies that $\delta\Phi^\uns$ satisfies the following fixed point (affine) equation
\begin{align}\label{eq_matching_second}
    \delta \Phi^\uns(z) = &\left (\begin{array}{cc}
    z^3 a_{\phi^\uns} + \frac{1}{z^2} b_{\phi^\uns} \\    
    e^{i(z-z_1)} a_{\psi^\uns}  + e^{-i(z-z_2)} b_{\psi^\uns} 
    \end{array}\right ) +  \GG^\match[A^\uns](z) - \left (\begin{array}{c} \GG^\match_1[\delta \psi^\uns](z) \\ 0\end{array}\right )
    \\ &+ \GG^\match [\mathcal{B}_1^\uns \cdot \delta \Phi^\uns] (z) + \GG^\match[\mathcal{B}_2^\uns \cdot \partial_z \delta \Phi^\uns](z),    \notag 
\end{align} 
where $a_{\phi^\uns},b_{\phi^\uns },a_{\psi^\uns}, b_{\psi^\uns}$ are defined by~\eqref{def:lem:abphi} and  
we have used definition~\eqref{def:calB:matching} of $\mathcal{B}_3$. To shorten the notation we introduce 
\begin{equation}\label{def:Phi0:Gmatch}
\begin{aligned}
\delta \Phi_0^\uns(z) &= \begin{pmatrix}\delta \phi_0^\uns(z)\\ \delta \psi_0^\uns(z)\end{pmatrix} =  \left (\begin{array}{cc}
    z^3 a_{\phi^\uns} + \frac{1}{z^2} b_{\phi^\uns} \\    
    e^{i(z-z_1)} a_{\psi^\uns}  + e^{-i(z-z_2)} b_{\psi^\uns} 
    \end{array}\right ) +  \GG^\match[A^\uns](z), \\
    \Gmatch[\delta\Phi]&=    \begin{pmatrix}\Gmatch_1 [\delta\Phi]\\ \Gmatch_2 [\delta\Phi]\end{pmatrix} = \GG^\match [\mathcal{B}_1^\uns \cdot \delta \Phi] (z) + \GG^\match[\mathcal{B}_2^\uns \cdot \partial_z \delta \Phi](z),
\end{aligned}
\end{equation} 
after which we rewrite equation~\eqref{eq_matching_second} as
\begin{equation}\label{eq_matching_third}
\delta \Phi^\uns = \delta\Phi^\uns_0 - \left (\begin{array}{c} \GG^\match_1[\delta \psi^\uns](z) \\ 0\end{array}\right ) + \Gmatch[\delta \Phi^\uns].
\end{equation}
Using that $\delta \Phi^\uns$ is a solution of~\eqref{eq_matching_third}, we observe that $\delta \Phi^\uns$ must be also a solution of 
\begin{equation}\label{eq_matching_forth}
\delta \Phi^\uns = \widehat{\delta \Phi^\uns_0} + \Ghmatch [\delta \Phi^\uns],
\end{equation}
with 
\begin{equation}\label{def:Gmatchhat}
\begin{split}
     \widehat{\delta \Phi^\uns_0}&= \delta\Phi^\uns_0 - \left (\begin{array}{c} \GG^\match_1[\delta \psi^\uns_0](z) \\ 0\end{array}\right ),\\
      \Ghmatch [\delta \Phi]&=- \left (\begin{array}{c} \GG^\match_1 \big [\Gmatch_2[\delta \Phi] \big ](z) \\ 0\end{array}\right )+
\Gmatch[\delta \Phi].
\end{split}
\end{equation}

\subsection{The matching error}
 
For fixed $\ell\in\RR$, we introduce the norm
$$
\|f\|_\ell = \sup_{z\in D_{\theta_1,\theta_2,\nu}^{\match}} \left | z^\ell f(z)\right|
$$
and the Banach spaces 
\begin{equation*}
    \begin{split}
        \mathcal{Y}_{\ell}&=\{f : \mathcal{D}
        _{\theta_1,\theta_2,\nu }^{ \match} \to \mathbb{C}; \, f\, \text{is an analytic function and }\|f\|_{\ell}<\infty\}, \\ 
        \mathcal{D}\mathcal{Y}_{\ell}&=\{f : \mathcal{D}_{\theta_1,\theta_2,\nu}^{ \match} \to \mathbb{C}; \, f\, \text{is an analytic function and }\|f\|_{\ell} + \|f'\|_{\ell+1} <\infty\}.
    \end{split}
\end{equation*}
These Banach spaces satisfy the following properties.

\begin{lemma}\label{lem:banach:matching}
Let $\ell_1,\ell_2 \in \mathbb{R}$. Then 
\begin{enumerate}
    \item If $f\in \mathcal{Y}_{\ell_1}$, then $f\in \mathcal{Y}_{\ell_2}$, for all $\ell_2\in \mathbb{R}$. Moreover for $\ell_1 >\ell_2$ 
    $$
    \|f\|_{\ell_2} \lesssim \kappa^{\ell_2-\ell_1} \| f\|_{\ell_1}
    $$
    and for $\ell_1<\ell_2$, 
    $$
    \|f\|_{\ell_2}\lesssim \varepsilon^{(\ell_1-\ell_2)(1-\nu)}.
    $$
    \item If $f\in \mathcal{Y}_{\ell_1}$ and $g\in \mathcal{Y}_{\ell_2}$, then 
    $\|fg\|_{\ell_1 + \ell_2} \leq \|f\|_{\ell_1}\|g\|_{\ell_2}$.
\end{enumerate}
\end{lemma}

We define the product Banach space $ \mathcal{Y}_\times=\mathcal{D}\mathcal{Y}_2 \times \mathcal{Y}_4$ endowed with the product norm 
\begin{equation}\label{def:norm:matchin}
\| (\phi,\psi) \|_\times= \max\{\| \phi \|_2 + \| \partial_z \phi \|_3 ,\|\psi \|_4 \}.
\end{equation}
We note that, as claimed in Remark~\ref{rmk:matching1}, $\delta \phi^\uns \in \mathcal{D}\mathcal{Y}_3$, $\delta \psi^\uns \in \mathcal{Y}_5$ with $\|\delta \phi^\uns\|_3 + \|\partial_z \delta \phi^\uns \|_4 , \|\delta \psi^\uns \|_5 \lesssim 1$ and therefore, by Lemma~\ref{lem:banach:matching},
\begin{equation}\label{Banachdeltaphi}
\|\delta \Phi^\uns\|_\times= \max\{|\delta \phi^\uns\|_2 + \|\partial_z \delta \phi^\uns\|_3, \|\delta \psi^\uns\|_4 \}\lesssim \frac{1}{\kappa}. 
\end{equation}

We now start estimating all the elements in the fixed point equation~\eqref{eq_matching_third}. 
The following lemma, whose proof is given in  Section~\ref{sec:matching:operator}, deals with the operators $\GG^\match$ and $\Gmatch$ defined in~\eqref{def:lem:mathcalI} and~\eqref{def:Phi0:Gmatch} respectively. 
 
\begin{lemma}\label{lem:lin_operator_matching} If $\kappa $ is big enough, the following statements are satisfied:
\begin{enumerate}
\item \label{firstitem}If $h\in \mathcal{Y}_\ell$ with $\ell>4$, then 
$\GG^\match_1 [h] \in \mathcal{Y}_{\ell-2}$ and 
$$
\|\GG^\match_1[h]\|_{\ell -2} \lesssim \|h \|_{\ell}, \qquad \|\partial_z \GG^\match_1[h]\|_{\ell -1} \lesssim \|h \|_{\ell}.
$$
\item\label{seconditem} If $h \in \mathcal{Y}_{\ell}$ with $\ell>0$, then
$\GG^\match_2[h] \in \mathcal{Y}_{\ell}$ and $\|\GG^\match_2[h]\|_{\ell} \lesssim\|h\|_{\ell}$.
\item If $h\in \mathcal{Y}_4$, then 
$\GG^\match_1 [h] \in \mathcal{Y}_{2}$ and $\|\GG^\match_1[h]\|_{2} \lesssim|\log \varepsilon| \|h \|_{2}$.
\item \label{itemforth} If $h\in \mathcal{Y}_\times =\mathcal{D}\mathcal{Y}_2 \times \mathcal{Y}_4$, then $\Gmatch[h] =\big (\Gmatch_1[h],\Gmatch_2 [h]\big )\in \mathcal{D}\mathcal{Y}_4 \times \mathcal{Y}_6$ with
$$
\|\Gmatch_1 [h]\|_4+ \|\partial_z \big (\Gmatch_1[h]\big )\|_5+\|\Gmatch_2 [h]\|_6 \lesssim  \|h\|_\times.
$$
As a consequence, by definition~\eqref{def:norm:matchin} of $\|\cdot \|_\times$, we have $
\|\mathcal{G}^{\match}[h]\|_\times  \lesssim \frac{1}{\kappa^2} \|h\|_\times.
$
\end{enumerate}
\end{lemma}

We claim now that the operator $\Ghmatch: \mathcal{Y}_\times \to \mathcal{Y}_\times$ defined in~\eqref{def:Gmatchhat} satisfies that, for $\kappa$ big enough,  
\begin{equation}\label{eq:Gmatchcontractive}
\| \Ghmatch[h] \|_\times \lesssim \frac{1}{\kappa^2}\|h\|_\times.
\end{equation}
Indeed, by item~\ref{itemforth} in Lemma~\ref{lem:lin_operator_matching}, if $h\in \mathcal{Y}_\times$, then $\Gmatch_2[h] \in \mathcal{Y}_6$. Therefore, by item~\ref{firstitem} in Lemma~\ref{lem:lin_operator_matching},
$\GG^\match_1 [\Gmatch_2[h]] \in \mathcal{D}\mathcal{Y}_4$  and the estimates in item~\ref{firstitem} apply. By Lemma~\ref{lem:banach:matching}, we have 
$$
\|\GG^\match_1 [\Gmatch_2[h]] \|_2+ \|\partial_z \GG^\match_1 [\Gmatch_2[h]] \|_3 \lesssim \frac{1}{\kappa^2}\|h\|_\times.
$$
Then, the claim follows from item~\ref{itemforth} of Lemma~\ref{lem:lin_operator_matching} and definition~\eqref{def:Phi0:Gmatch} of $\Ghmatch$.   

It follows from~\eqref{eq_matching_forth} that 
$$
\big (\mathrm{Id} -  \Ghmatch\big )\delta \Phi^\uns = \widehat{\delta \Phi^\uns_0}. 
$$
Therefore,  using that $\delta \Phi^* \in \YY_\times$ (see~\eqref{Banachdeltaphi}) and that, by \eqref{eq:Gmatchcontractive}, $\mathrm{Id} - \Ghmatch:\YY_\times\to \YY_\times$ is invertible, we obtain that
$$
\delta \Phi^\uns = \big (\mathrm{Id}- \Ghmatch\big )^{-1} [\widehat{\delta \Phi^\uns_0}]\quad \text{and}\quad \| \delta \Phi^\uns\|_\times \lesssim \|\widehat{\delta \Phi^\uns_0}\|_\times. 
$$
Theorem~\ref{thm:matching} is then a consequence of the following lemma whose proof is given in Section~\ref{sec:proofdelphi0}.

\begin{lemma}\label{lem:deltaphi0} Let $\nu \in (0,1)$. 
If $\kappa$ is big enough, then $\|\widehat{\delta \Phi^\uns_0} \|_\times \lesssim |\log \eps|\eps^{1-\nu}$. 
\end{lemma}

It remains to prove Lemmas~\ref{lem:lin_operator_matching} and \ref{lem:deltaphi0}.

\subsection{Proof of Lemma~\ref{lem:lin_operator_matching}} \label{sec:matching:operator}

The proof of the three first items of Lemma~\ref{lem:lin_operator_matching} can be found in the proof of Lemma 6.2 in \cite{GomideGSZ22} (see also~\cite{BaldomaCS13}).

Now we prove item~\ref{itemforth}. We first note that, from definition~\eqref{eq:inner:RHSoperator} of   $\Nin_1, \Nin_2$, 
\begin{align*}
D_{\Phi} \Nin[\Phi] (z)= \left ( 
\begin{array}{cc} 
-\frac{12}{z} \phi - 6 \phi^2 & -1 \\
 \mathfrak{g}[\Phi]& 
- 6 \left (\frac{1}{z} + \phi \right )^2 
\end{array}\right ), \\
D_{\partial_z \phi} \Nin[\Phi] (z) = \left (0, -24 \left (\frac{1}{z} + \phi  \right ) \left (-\frac{1}{z^2} + \partial_z \phi\right ) \right )^{\top},
\end{align*}
where 
$$
\mathfrak{g}[\Phi](z) = -12 \left (\frac{1}{z} + \phi \right ) \left (\psi + 2 \left (\frac{1}{z} + \phi\right )^3 \right )- 36 \left (\frac{1}{z} + \phi \right )^4 - 12 \left (-\frac{1}{z^2} + \partial_z \phi \right )^2.
$$
Let us denote 
\[
P(z)=D_\Phi \Nin[\Phi ](z) - \left (\begin{array}{cc} 0 &-1 \\ 0 & 0 \end{array}\right ), \qquad Q(z)= D_{\partial_z \phi} \Nin[\Phi] (z).
\]
Then, $P=(P_{ij})_{i,j}$ is a $2\times 2$ matrix and, for $\Phi\in \mathcal{Y}_3\times \mathcal{Y}_3$,
its coefficients satisfy 
\[
|P_{11}(z)|\lesssim \frac{1}{|z|^4},\quad P_{12}(z)=0,\quad|P_{21}(z)|\lesssim \frac{1}{|z|^4},\quad|P_{22}(z)|\lesssim \frac{1}{|z|^2},\quad
\]
whereas $Q$ is a $2$-dimensional vector which, for $\Phi\in \mathcal{Y}_3\times \mathcal{Y}_3$, satisfies
\[
Q_1(z)=0,\qquad |Q_{2}(z)|\lesssim \frac{1}{|z|^3}.
\]
Finally, by definition~\eqref{def:calB:matching} of $\mathcal{B}_1^\uns(z), \mathcal{B}_2^\uns(z)$, if $h\in \mathcal{D}\mathcal{Y}_2 \times \mathcal{Y}_4$, 
then we have 
$$
\|\mathcal{B}_1^\uns \cdot h \|_6 , \,   \| \mathcal{B}_2^\uns \cdot h \|_6 \lesssim \|h\|_\times,
$$ 
and by item~\ref{firstitem} and item~~\ref{seconditem} of Lemma~\ref{lem:lin_operator_matching}, $\Gmatch[h] \in \mathcal{D}\mathcal{Y}_4 \times \mathcal{Y}_6$ with bounded norm. This completes the proof of Lemma \ref{lem:lin_operator_matching}. 

\subsection{Proof of Lemma~\ref{lem:deltaphi0}}  \label{sec:proofdelphi0}
We introduce 
$$ 
\widetilde{\delta \phi_0^\uns} = z^3 a_{\phi^\uns} + \frac{1}{z^2 } b_{\phi^\uns},  \qquad 
\widetilde{\delta \psi_0^\uns} = e^{i(z-z_1) }a_{\psi^\uns} + e^{-i(z-z_2) }b_{\psi^\uns},
$$
where $a_{\phi^\uns}, b_{\phi^\uns}, a_{\psi^\uns}$ and $b_{\psi^\uns}$ are defined by~\eqref{def:lem:abphi} with $\phi = \phi^\uns$ and $\psi=\psi^\uns$.
From~\eqref{def:Gmatchhat}, we have that 
$\widehat{\delta\Phi^\uns_0} = \big ( \widehat{\delta \phi^\uns_0}, \widehat{\delta\psi^\uns_0}\big ) $
is defined by 
\begin{align*}
 \widehat{\delta \phi_0^\uns}(z) & = \delta \phi_0^\uns(z)  - \GG^\match_1[\delta \psi_0^\uns] = \widetilde{\delta \phi_0^\uns}(z)   + \GG^\match_1[A_1^\uns](z) - \GG^\match_1[\delta \psi_0^\uns],\\ 
 \widehat{\delta \psi_0^\uns}(z)& = \delta \psi_0^\uns(z) =  \widetilde{\delta \psi_0^\uns}(z) + \GG^\match_2[A_2^\uns](z),
 \end{align*}
 where $A^\uns=\big (A^\uns_1,A^\uns_2\big)$ is defined by 
 \begin{equation}\label{recallAast}
 A^\uns(z) = \mathcal{A}[\Phi^\uns](z) = \Nch_1[\Phi^\uns](z) - \Nin_1 [\Phi^\uns](z). 
 \end{equation}
We recall that $\phi^\uns \in D\mathcal{Y}_3$ and $\psi^\uns\in \mathcal{Y}_5$, see \eqref{recall:phipsi}. The following lemma estimates $\widetilde{\delta \phi_0^\uns}(z)$.

\begin{lemma} \label{lem:deltaphi0bound} 
	Fix $\nu\in (0,1)$. If $\eps>0$ is small enough, then 
we have for all $z\in D_{\theta_1,\theta_2,\nu}^\match$,
$$
\big |z^2 \widetilde{\delta \phi_0^\uns}(z) \big | + \big |z^3\partial_z \widetilde{\delta \phi_0^\uns}(z) \big |+ \big |z^4 \widetilde{\delta \psi_0^\uns}(z) \big | \lesssim \eps^{1-\nu}.
$$
\end{lemma}

\begin{proof}
By definition~\eqref{def:lem:abphi}, we have
$$
|a_\phi^\uns|\lesssim\frac{1}{|z_1|^6} \lesssim \varepsilon^{6(1-\nu)}, \qquad |b_\phi^\uns| \lesssim \frac{1}{|z_1|} \lesssim\varepsilon^{1-\nu}, \qquad |a_\psi^\uns|, |b_\psi^\uns| \lesssim  \varepsilon^{5(1-\nu)}. 
$$
Then, for $z\in D_{\theta_1,\theta_2}^{\nu, \match}$, using that $|z| \lesssim \min \{ |z_1|, |z_2|\} \lesssim \varepsilon^{-(1-\nu)}$, we obtain 
\[
\begin{split}
\big |z^2 \delta \phi_0^\uns(z) \big |& = \left | z^5 a_{\phi} + b_{\phi} \right |\lesssim    
|z|^5 \eps^{6(1-\nu)} + \eps^{1-\nu} \lesssim \eps^{1-\nu}, \\
\big |z^4 \delta \psi_0^\uns(z) \big | &  \lesssim \varepsilon^{5(1-\nu)} |z|^4 \big (e^{-\Im (z-z_1)} + e^{\Im (z-z_2)} \big ) \lesssim \eps^{1-\nu},
\end{split}
\]
where in the last inequality we have used that $\Im z_2 >\Im z >\Im z_1$. 
\end{proof}

Next we analyze $\GG^\match[A^\uns]$. To do so, we look for an explicit expression of $\Nch$. 

\begin{lemma}\label{lem:1:matching}
    The fixed point equation~\eqref{eq:perturb} in the inner variables~\eqref{inner_variables_matching} can be written as 
    \begin{equation*}
\left\{  \begin{array}{l}
        \mathcal{L}_1^{\inn} \phi = \Nch_1[\phi,\psi;\varepsilon], \\ 
        \mathcal{L}_2^{\inn} \psi =
        \Nch_2[\phi,\psi;\varepsilon],
    \end{array} \right.
\end{equation*}
with 
 \begin{equation*}
\left\{  \begin{array}{l}
    \Nch_1 [\phi,\psi;\varepsilon](z)=  
    \Nin_1[\phi,\psi](z) + \mathcal{A}_1[\phi,\psi;\varepsilon](z), \vspace{0.2cm}\\ 
     \Nch_2[\phi,\psi;\varepsilon](z)= \Nin_1[\phi,\psi](z) + \mathcal{A}_2[\phi,\psi;\varepsilon](z),
    \end{array} \right.
\end{equation*}
where, for $z \in \mathcal{D}_{\theta_1,\theta_2,\nu}^{\match}$, 
\begin{equation}\label{boundsA1A2:matching}
\big | \mathcal{A}_1[\phi^\uns,\psi^\uns;\varepsilon](z) \big | \lesssim
  \frac{\eps}{|z|^4}  , \qquad 
\big | \mathcal{A}_2[\phi^\uns,\psi^\uns;\varepsilon](z) \big | \lesssim \frac{\eps}{|z|^4}.  
\end{equation}
\end{lemma}

\begin{proof}
An straightforward computation shows that in the inner variables, the fixed point equation~\eqref{eq:perturb} can be expressed as 
\begin{equation*}
\left\{  \begin{array}{l}
        \mathcal{L}_1^{\inn} \phi = \Nch_1[\phi,\psi;\varepsilon], \vspace{0.2cm}\\ 
        \mathcal{L}_2^{\inn} \psi =
        \Nch_2[\phi,\psi;\varepsilon],
    \end{array} \right.
\end{equation*}
 with
 \begin{equation*}
\left\{  \begin{array}{l}
     \Nch_1[\phi,\psi;\varepsilon](z)  = \varepsilon^2 \phi(z) \left [-1 + 2u_0(x_-+\eps z )\right ]   + \phi(z) \left [ 6\gamma \varepsilon^2 u_0^2(x_-+\eps z ) + \frac{6}{z^2}\right ] \vspace{0.2cm} \\      \hspace{5cm} + 
     \varepsilon^3 c_{- 1}^{-1}\mathcal{F}_1[\varepsilon^{-1} c_{- 1} \phi, \varepsilon^{-3} c_{- 1} \psi](x_-+\eps z), \vspace{0.2cm}\\ 
     \Nch_2[\phi,\psi;\varepsilon](z) = \varepsilon^{5} c_{- 1}^{-1}
     \mathcal{F}_2[\varepsilon^{-1} c_{- 1} \phi, \varepsilon^{-3} c_{- 1} \psi] (x_-+\eps z).
    \end{array} \right.
 \end{equation*}
Using the expression~\eqref{eq:outer:RHSoperator} of $\mathcal{F}=(\FF_1,\FF_2)$ we obtain 
 \begin{align*}
    \Nch_1[\phi,\psi;\varepsilon](z)=  
    & -\psi  -\frac{6}{z} \phi^2 -2 \phi^3 + \mathcal{A}_1[\phi;\varepsilon](z)
    \\ =& \Nin_1[\phi,\psi](z) + \mathcal{A}_1[\phi,\psi;\varepsilon](z),
\end{align*}
with 
\begin{align*}
\mathcal{A}_1[\phi,\psi;\varepsilon](z)=& \varepsilon^2 \phi(z) \left [-1 + 2\hat{u}_0(z)\right ] + \phi(z) \left [ 6\gamma \varepsilon^2 \hat{u}_0^2(z)  + \frac{6}{z^2}\right ] \\ 
&+
\left ({c_{- 1} }\varepsilon + 6 \varepsilon \gamma c_{- 1} \hat{u}_0(z)  + \frac{6}{ z} \right ) \phi^2.
\end{align*}
Analogously, tedious but easy computations lead to
\begin{align*}
    \Nch_2[\phi,\psi;\varepsilon](z)=& -6 \left (\frac{1}{z} +\phi \right )^2  \left (\psi + 2 \left (\frac{1}{z} + \phi\right )^3 \right )-12 \left (\frac{1}{z} + \phi \right ) \left (-\frac{1}{z^2} + \partial_z \phi\right )^2 \\ & - 6  \left (\frac{1}{z} + \phi\right )^2 C[\phi,\psi;\varepsilon](z) +  \left (\psi + 2 \left (\frac{1}{z} + \phi\right )^3 \right ) B[\phi;\varepsilon](z) \\ &+   B[\phi;\varepsilon](z) \cdot C[\phi,\psi;\varepsilon](z)  + D[\phi,\psi;\varepsilon](z)
    \\ =& \Nin_2[\phi,\psi](z) + \mathcal{A}_2[\phi,\psi;\varepsilon](z)
\end{align*}
with
\begin{align*}
          B[\phi;\varepsilon](z)=&-6\left (\varepsilon c_{- 1}^{-1} \hat{u}_0  - \frac{1}{z} \right ) \left ( \frac{1}{z} + 2 \phi + \varepsilon c_{- 1}^{-1} \hat{u}_0   \right )  + 2 \varepsilon c_{- 1} (\varepsilon c_{- 1}^{-1} \hat{u}_0 + \phi),
\\ 
C[\phi,\psi;\varepsilon](z) = & 
\varepsilon^2 ( \varepsilon c_{- 1}^{-1} \hat{u}_0 + \phi) - \varepsilon c_{- 1} (\varepsilon c_{- 1}^{-1} \hat{u}_0 + \phi)^2 \\
  &+2 \left (\varepsilon c_{- 1}^{-1} \hat{u}_0 - \frac{1}{z}\right )\left[ \left  (\varepsilon c_{- 1}^{-1}\hat{u}_0  + \phi   \right )^2 +
\left  (\varepsilon c_{- 1}^{-1} \hat{u}_0 + \phi  \right ) \left (\frac{1}{z} + \phi\right ) \right] \\ &+
2 \left (\varepsilon c_{- 1}^{-1} \hat{u}_0 - \frac{1}{z}\right ) \left(\frac{1}{z} + \phi \right )^2  ,
\\  
D[\phi,\psi;\varepsilon](z) =& 2 c_{- 1} \varepsilon (\varepsilon c_{- 1}^{-1} \hat{u}_0' + \partial_z \phi)^2  - 12 \left (\varepsilon c_{- 1}^{-1} \hat{u}_0 -\frac{1}{z}\right ) \left (\varepsilon c_{- 1}^{-1} \hat{u}_0' +\partial_z \phi\right )^2 \\ 
&-12 \left (\frac{1}{z} +\phi \right ) \left (\varepsilon c_{- 1}^{-1}  \hat{u}_0'+ \frac{1}{z^2} \right ) \left (2\partial_z \phi + \varepsilon  c_{- 1}^{-1}\hat{u}_0' - \frac{1}{z^2}\right ).
\end{align*} 
To prove the bounds for  $\mathcal{A}_1[\phi^\uns,\psi^\uns;\eps], \mathcal{A}_2[\phi^\uns, \psi^\uns; \eps]$, we recall that $c_{-1}^{-1}= \sqrt{|\gamma|}$ with $\gamma<0$ and take into account~\eqref{exp_u0_inner} and~\eqref{recall:phipsi}, to obtain
$$
\left | \frac{1}{z} + \phi^\uns(z) \right | \lesssim \frac{1}{|z|}, \qquad \left |\eps c_{-1}^{-1} \hat{u_0}- \frac{1}{z}\right | \lesssim \eps, \qquad 
\left |\varepsilon  c_{- 1}^{-1}\hat{u}_0' + \frac{1}{z^2} \right | \lesssim \eps^2.
$$
The proof of~\eqref{boundsA1A2:matching} follows from these bounds and the explicit expressions of the functions involved. 
\end{proof}
 
Lemma~\ref{lem:1:matching}, together with items~\ref{firstitem} and~\ref{seconditem} of Lemma~\ref{lem:lin_operator_matching}, implies that,
for all $z\in \mathcal{D}_{\theta_1,\theta_2,\nu}^\match$, we have
$$
\big |z^2 \GG^\match_1[A_1^\uns] (z) \big |+ \big |z^3 \partial_z \GG^\match_1[A_1^\uns] (z) \big | + |z^4 \GG^\match_2[A_2^\uns](z) \big | \lesssim \eps |\log \eps|, 
$$
where we recall that $A^\uns(z)=\mathcal{A}[\phi^\uns,\psi^\uns]$ (see~\eqref{recallAast}). This estimate and Lemma \ref{lem:deltaphi0bound} imply that for all $z\in \mathcal{D}_{\theta_1,\theta_2,\nu}^\match$,  we have 
\begin{equation*}
\big |z^2 \delta \phi_0^\uns(z) \big | + \big |z^3\partial_z \delta \phi_0^\uns(z) \big |+ \big |z^4\delta \psi_0^\uns(z) \big | \lesssim \eps^{1-\nu}.
\end{equation*}
To estimate $\widehat{\delta \phi_0^\uns}(z)$, it only remains to analyze $\GG^\match_1 [ \delta\psi_0^\uns ]$. To this end, it is enough to recall that $\big |z^4\delta \psi_0^\uns(z) \big | \lesssim \eps^{1-\nu}$ and Lemma~\ref{lem:lin_operator_matching} imply 
\begin{equation*}
    \big | z^2 \GG^\match_1 [ \delta\psi_0^\uns ] (z) \big | \lesssim |\log \eps| \eps^{1-\nu}. 
\end{equation*}
Therefore, recalling that $\widehat{\delta \psi_0^{\uns}} = \delta \psi_0^\uns$, we conclude that, for all $z\in \mathcal{D}_{\theta_1,\theta_2,\nu}^\match$, we have
$$
\big |z^2 \widehat{\delta\phi_0^\uns} (z) \big |+ \big |z^3\partial_z \widehat{\delta  \phi_0^\uns}(z)\big | + |z^4 \widehat{\psi_0^\uns}(z) \big | \lesssim \eps^{1-\nu} |\log \eps|.
$$
This completes the proof of Lemma \ref{lem:deltaphi0}.

\section{The difference between the invariant manifolds}
\label{sec:difference}

Here we prove Proposition \ref{prop:difference:remainder} for $\Delta\eta^{\uns}$. The proof for $\Delta\eta^{\sta}$ is analogous.
We define first the following Banach spaces with norms with exponential weights
\[
\EE_\ell=\{h:E^{\out,\uns}_\kk\to\CC; \, h\,\, \text{real-analytic},\, \|h\|_{\ell,\exp}<\infty\},
\]
where
\begin{equation}\label{def:expnorm}
\|h\|_{\ell,\exp}=\sup_{x\in E^{\out,\uns}_\kk}\left|(x-x_-)^\ell(x-x_+)^\ell(x-\bar x_-)^\ell(x-\bar x_+)^\ell e^{\frac{1}{\eps}\left(\pi-|\Im x|\right)}h(x)\right|.
\end{equation}
We also consider the  Banach space
\[
\EE_\times=\{h=(h_1,h_2):E^{\out,\uns}_\kk\to\CC^2; \, h\,\, \text{real-analytic},\, \|h\|_{\times}<\infty\},
\]
where
\begin{equation}\label{def:norm:expvector}
\|h\|_{\times}= \max \big \{\eps^{-1}\|h_1\|_{0,\exp},\|h_2\|_{0,\exp}+\eps\|\pa_x h_2\|_{0,\exp} \big \}.
\end{equation}
We look for an integral equation in these Banach spaces which has as a unique solution $(\Delta\zeta^{\uns},\Delta\eta^{\uns})$. The following lemma presents suitable inverses of the operators $\wh\LL_1$ and $\LL_2$ defined by \eqref{def:L1hat} and  \eqref{def:diffoperators} respectively. Its proof follows the same lines as the proof of Lemma  7.1 in \cite{GomideGSZ22}.
 
\begin{lemma}\label{lemma:diff:operator:2}
The operators 
\begin{equation*}
\label{def:intoperators-diff:1}
\wh{\GG}_1[h](x)= u_0''(x)\int_0^x\frac{1}{u_0''(s)}h(s)ds\\
\end{equation*}
and
\begin{equation*}
\begin{split}
\wh{\GG}_2[h] (x)=&
-\frac{i\eps}{2} e^{i\eps^{-1}x}\int_{\rminus}^x  e^{-i\eps^{-1}s} h(s)ds+\frac{i\eps}{2} e^{-i\eps^{-1}x}\int_{\overline{\rminus}}^x  e^{i\eps^{-1}s} h(s)ds\\
&+\frac{i\eps\sin\left(\frac{\rr_--x}{\eps}\right)}{2\sin\left(\frac{\rr_--\ol{\rr_-}}{\eps}\right)}e^{i\eps^{-1}\ol{\rr_-}}\int_{\rr_-}^{\ol{\rr_-}}  e^{-i\eps^{-1}s} h(s)ds\\
&-\frac{i\eps\sin\left(\frac{\ol{\rr_-}-x}{\eps}\right)}{2\sin\left(\frac{\ol{\rr_-}-\rr_-}{\eps}\right)}e^{i\eps^{-1}\rr_-}\int^{\rr_-}_{\ol{\rr_-}}  e^{i\eps^{-1}s} h(s)ds,
\end{split}
\end{equation*}
with $\rminus=x_--i\kappa\eps$, have the following properties.
\begin{itemize}
\item Fix $\ell\in\RR$. The operator $\wh{\GG}_1$ is well defined from $\EE_\ell$ to $\EE_\ell$ and satisfies
\[
\left\|\wh{\GG}_1[h]\right\|_{\ell,\exp}\leq M\eps\|h\|_{\ell,\exp}.
\]
It is also well-defined from $\EE_\ell$ to $\EE_0$ and satisfies
\[
\left\|\wh{\GG}_1[h]\right\|_{0,\exp}\leq \frac{M\eps}{(\kappa\eps)^{\ell}}\|h\|_{\ell,\exp}.
\]
Furthermore, $\wh\LL_1\circ \wh{\GG}_1=\mathrm{Id}$ and, for $h\in \EE_\ell$, 
\[
\wh{\GG}_1(h)(0)=0.
\]
\item Fix $\ell>1$. The operator $\wh{\GG}_2$ is well defined from $\EE_\ell$ to $\EE_0$ and satisfies
\[
\begin{split}
\left\|\wh{\GG}_2[h]\right\|_{0,\exp}&\leq \frac{M\eps}{(\kappa\eps)^{\ell-1}}\|h\|_{\ell,\exp},\\
\left\|\pa_x\wh{\GG}_2[h]\right\|_{0,\exp}&\leq \frac{M}{(\kappa\eps)^{\ell-1}}\|h\|_{\ell,\exp}.
\end{split}
\]
Furthermore, $\LL_2\circ \wh{\GG}_2=\mathrm{Id}$ and, for $h\in \EE_\ell$ 
\[
\wh{\GG}_2[h](\rr_-)=0\quad\text{and}\quad\wh{\GG}_2[h](\ol{\rr_-})=0.
\]
\end{itemize}
\end{lemma}

The functions $(\Delta\zeta^\uns,\Delta\eta^\uns)$ introduced in Lemma \ref{lemma:differenceequation:2} satisfy equation \eqref{eq:lineardiff:2}. Now, by the properties of the operators $\wh\GG_1$ and $\wh \GG_2$ introduced in Lemma \ref{lemma:diff:operator:2}, the functions $(\Delta\zeta^\uns,\Delta\eta^\uns)$ must be a fixed point of the operator
\begin{equation}\label{eq:diff:operator}
\PP \big [\Delta\zeta,\Delta\eta\big ](x)=\begin{pmatrix}
   \wh \GG_1\circ\wh\NNN_1 \big [\Delta\zeta,\Delta\eta, \Delta\eta' \big ](x)\\   
 C_1^\uns e^{\frac{ix}{\eps}}+C_2^\uns e^{-\frac{ix}{\eps}}+\wh \GG_2\circ\wh\NNN_2 \big [\Delta\zeta,\Delta\eta,\Delta\eta' \big ](x)
\end{pmatrix}
\end{equation}
for some constants $C_1^\uns$, $C_2^\uns$ satisfying \eqref{eq:linearsystemCs}.

Note that by Lemma \ref{lemma:diff:operator:2}, the function $\RRR^\uns$ introduced in Lemma \ref{prop:difference:remainder} is given by 
\begin{equation}\label{def:Runs}
\RRR^\uns=\wh \GG_2\circ\wh\NNN_2 \big [\Delta\zeta^\uns,\Delta\eta^\uns,\pa_x\Delta\eta^\uns\big ].
\end{equation}
and it satisfies the properties in \eqref{def:R:zeros}. Therefore, it only remains to obtain the estimates in \eqref{def:R:estimates}. 

To this end, we use a fixed point argument relying on \eqref{eq:diff:operator}. However, the operator $\PP$ is not contractive and, therefore, proceeding as in Section \ref{sec:outer}, we consider the operator
\begin{equation*}
\wh \PP \big [\Delta\zeta,\Delta\eta \big ]=\begin{pmatrix}
 \PP_1 \big [\Delta\zeta,\PP_2\big [\Delta\zeta,\Delta\eta \big ] \big ]\\   
\PP_2 \big [\Delta\zeta,\Delta\eta \big ],
\end{pmatrix}
\end{equation*}
which has the same fixed points as $\PP$ and is contractive. Note that both operators $\PP$ and $\wh\PP$ are affine. The following lemma gives the Lipschitz constant of the operator $\PP$.  Its proof  is a direct consequence of Lemmas \ref{lemma:differenceequation:2} and 
\ref{lemma:diff:operator:2}. 

\begin{lemma}\label{lemma:operatorPLipschtiz}
There exists $M>0$ such that, for any $(\Delta\zeta_1,\Delta\eta_1), (\Delta\zeta_2,\Delta\eta_2) \in \EE_\times$, the operator $\PP$ satisfies
\[
\begin{split}
\left\|\PP_1 \big [\Delta\zeta_1,\Delta\eta_1 \big ]-\PP_1 \big [\Delta\zeta_2,\Delta\eta_2 \big ]\right\|_{0,\exp}\leq& M\eps \left\|\Delta\eta_1-\Delta\eta_2\right\|_{0,\exp}
\\&+\frac{M\eps}{\kappa}\left\|(\Delta\zeta_1,\Delta\eta_1)-(\Delta\zeta_2,\Delta\eta_2)\right\|_\times, \\
\left\|\PP_2 \big [\Delta\zeta_1,\Delta\eta_1 \big ]-\PP_2 \big [\Delta\zeta_2,\Delta\eta_2 \big ]\right\|_{0,\exp}\leq &\frac{M}{\kappa}\left\|(\Delta\zeta_1,\Delta\eta_1)-(\Delta\zeta_2,\Delta\eta_2)\right\|_\times, \\
\left\|\pa_x\PP_2 \big [\Delta\zeta_1,\Delta\eta_1 \big ]-\pa_x\PP_2 \big [\Delta\zeta_2,\Delta\eta_2 \big ]\right\|_{0,\exp}\leq &\frac{M}{\eps\kappa}\left\|(\Delta\zeta_1,\Delta\eta_1)-(\Delta\zeta_2,\Delta\eta_2)\right\|_\times.
\end{split}
\]
\end{lemma}

Lemma \ref{lemma:operatorPLipschtiz} implies that $\wh\PP$ satisfies 
\[
\left\|\wh\PP_1 \big [\Delta\zeta_1,\Delta\eta_1 \big ]-\wh\PP \big [\Delta\zeta_2,\Delta\eta_2 \big ]\right\|_\times\leq \frac{M}{\kappa}\left\|(\Delta\zeta_1,\Delta\eta_1)-(\Delta\zeta_2,\Delta\eta_2)\right\|_\times.
\]
Therefore, taking $\kk>0$ large enough, $\wh \PP$ is contractive and has the unique fixed point $(\Delta\zeta^\uns,\Delta\eta^\uns)$. 

We use $\wh \PP$ to obtain estimates of the fixed point with respect to the norm introduced in \eqref{def:norm:expvector}. Indeed, since it is a fixed point, it can be written as 
\[
(\Delta\zeta^\uns,\Delta\eta^\uns)=\wh\PP [0,0]+\left[\wh\PP \big [\Delta\zeta^\uns,\Delta\eta^\uns \big ]-\wh\PP[0,0]\right]
\]
and, therefore, 
\[
\begin{split}
\left\|(\Delta\zeta^\uns,\Delta\eta^\uns)\right\|_\times\leq&\left\|\wh\PP(0,0)\right\|_\times+\left\|\wh\PP(\Delta\zeta^\uns,\Delta\eta^\uns)-\wh\PP(0,0)\right\|_\times\\
\leq&\left\|\wh\PP(0,0)\right\|_\times+\frac{M}{\kappa}\left\|(\Delta\zeta^\uns,\Delta\eta^\uns)\right\|_\times.
\end{split}
\]
Taking $\kk$ large enough implies that 
\[
\left\|(\Delta\zeta^\uns,\Delta\eta^\uns)\right\|_\times\leq2\left\|\wh\PP(0,0)\right\|_\times.
\]
Therefore, it only remains to estimate 
\[
\wh\PP[0,0] (x) = \begin{pmatrix}\wh{\PP}_1[0,0](x) \\ 
\PP_2[0,0](x)\end{pmatrix} =
\begin{pmatrix}\PP_1 [0,  {\PP}_2[0,0]](x)\\
 C_1^\uns e^{\frac{ix}{\eps}}+C_2^\uns e^{-\frac{ix}{\eps}}
\end{pmatrix},
\]
where $C_1^\uns$, $C_2^\uns$ are constants satisfying \eqref{eq:linearsystemCs}. 

By the definition of the norm \eqref{def:expnorm}, we have
\[
\left\|\PP_2[0,0] \right\|_{0,\exp}\leq \left(|C_1^\uns|+|C_2^\uns|\right)e^{\frac{\pi}{\eps}},
\]
which by Lemma \ref{lemma:operatorPLipschtiz}, implies
\[
\left\|\PP_1[0,\PP_2[0,0]  ]\right\|_{0,\exp}\lesssim \left(|C_1^\uns|+|C_2^\uns|\right)e^{\frac{\pi}{\eps}}.
\]
Therefore,  
\[
\left\|(\Delta\zeta^\uns,\Delta\eta^\uns)\right\|_\times\leq2\left\|\wh\PP [0,0]\right\|_\times\lesssim \left(|C_1^\uns|+|C_2^\uns|\right)e^{\frac{\pi}{\eps}}.
\]
Finally,  by definition~\eqref{def:Runs} of $\RRR^\uns$
\[
\RRR^\uns=\wh\PP_2 \big [\Delta\zeta^\uns,\Delta\eta^\uns \big ]-\wh\PP_2  [0,0  ],
\]
we obtain
\[
\left\|\RRR^\uns\right\|_{0,\exp}\leq \frac{M}{\kappa}\left\|(\Delta\zeta^\uns,\Delta\eta^\uns)\right\|_\times\lesssim
\frac{1}{\kappa}\left(|C_1^\uns|+|C_2^\uns|\right)e^{\frac{\pi}{\eps}},
\]
which concludes the proof of Proposition \ref{prop:difference:remainder}.

\appendix
\section{Proof of Lemma \ref{lemma:propertiesu0}}
\label{app-A}
\renewcommand\theequation{\Alph{section}.\arabic{equation}}
\setcounter{equation}{0}

We take $\beta=\sqrt{1+9\gamma} \in (0,1)$. It is straightforward to check that $u_0''(x) =0$ if and only if
$$
\beta \cosh^2 x - \cosh x - 2 \beta=0 
$$
so that 
$$
\cosh x = \frac{1}{2\beta} \left (1 \pm \sqrt{1+ 8\beta^2}\right )\in \mathbb{R}.
$$
Writing $x= a+ib$, we have that 
$$
\cosh a \cos b + i \sinh a \sin b=  \frac{1}{2\beta} \left (1 \pm \sqrt{1+ 8\beta^2}\right ).
$$
Therefore, $\sinh a \sin b=0$. If $a=0$, then 
$$
\cos b =  g_\pm(\beta):=\frac{1}{2\beta} \left (1 \pm \sqrt{1+ 8\beta^2}\right ).
$$
We impose 
$$
| 1 \pm \sqrt{1+ 8 \beta^2 } | \leq 2 \beta
$$ 
and obtain the condition 
$$
\pm \sqrt{1+ 8 \beta^2} \leq -1 -2 \beta^2 
$$
that it is always true, taking the negative sign and $\beta \in (0,1)$. This implies that, for $\beta \in (0,1)$, 
$$
-1 < \frac{1}{2\beta} \left (1 \pm \sqrt{1+ 8\beta^2}\right ) < 0
$$
and therefore $b=\mathrm{acos}(g_-\beta)) \in \left (\frac{\pi}{2}, \pi\right )$. Then $u''_0(\pm i b)=0$.

On the other hand, if $b=0$, then
$$
\cosh a = g_{\pm }(\beta) = \frac{1}{2\beta} \left (1 \pm \sqrt{1+ 8\beta^2}\right ).
$$
Since $g_-(\beta)<-1$, we need to study the zeros of $\cosh a = g_+(\beta)$. 
We notice that, since $\beta \in (0,1)$,  
$$
\cosh a = g_+ (\beta) > \frac{1}{\beta}>1
$$
and that implies that $a=\mathrm{acosh} ( g_+(\beta)) > \alpha$ and $u''_0(\pm a)=0$. 

Finally, when $b= \pm i \pi$, then 
$$
\cosh a= - g_{\pm}(\beta) =  \frac{1}{2\beta} \left (\pm \sqrt{1+ 8\beta^2}-1 \right )
$$
so that
$$
\cosh a = -g_+(\beta)= \frac{1}{2\beta} \left (\sqrt{1+ 8\beta^2}-1 \right ) < \frac{1}{\beta}.
$$

\section{Proof of Proposition \ref{prop:nonvanishing}}
\label{app:nonvanishingStokes}
\setcounter{equation}{0}

Here we prove that the constant $\Theta$ is not zero. To this end, it is convenient to work with just one function instead of two, as in the inner equation~\eqref{eq:inner}. Indeed, note that it is easy to check that if one defines
\[
\Phi=\frac{1}{z}+\phi, 
\]
it satisfies the fourth order equation
\begin{equation}\label{eq:inner:fourth}
\pa_z^4\Phi+\pa_z^2\Phi=2\Phi^3.
\end{equation}
We have the following lemma.

\begin{lemma}\label{lemma:series}
The functions 
\[
\Phi^\star(z)=\frac{1}{z}+\phi^{0,\star}(z),
\]
where $\phi^{0,\star}$ are the functions obtained in Theorem \ref{thm:inner}, are asymptotic to the same series at $z=\infty$ (within their domain of definition), which is of the form 
\[
\hat\Phi(z)=\sum_{n\geq 0}\frac{a_n}{z^{2n+1}},
\]
with coefficients satisfying that $a_n\in\RR$,
\begin{equation}\label{def:paritya_n}
a_n(-1)^n>0
\end{equation}
and 
\begin{equation}\label{def:angrowth}
|a_n|\geq  (2n)!.
\end{equation}
\end{lemma}

\begin{proof}
 To prove the lemma, we look for a recurrence to define the coefficients of $\Phi$. First note that by Theorem \ref{thm:inner} it must be of the form 
 \[
 \hat\Phi(z)=\frac{1}{z}+\OO\left(\frac{1}{z^3}\right)
 \]
It is straightforward to see from \eqref{eq:inner:fourth} that the series has only odd powers. We obtain that 
 \[
 \begin{split}
a_{n+1}=&\frac{1}{(2n+3)(2n+4)-6}\Bigg[-(2n+1)(2n+2)(2n+3)(2n+4)a_n\\
&+6\sum_{\substack{k_1,k_2\geq 1\\k_1+k_2=n+1}}a_{k_1}a_{k_2}+2\sum_{\substack{k_1,k_2,k_3\geq 1\\k_1+k_2+k_3=n+1}}a_{k_1}a_{k_2}a_{k_3}\Bigg],
 \end{split}
 \]
 which, by induction, implies  $a_n\in\RR$ and \eqref{def:paritya_n}.

 Moreover, for all $n\geq 0$,
 \[
 \begin{split}
|a_{n+1}|\geq& \frac{(2n+1)(2n+2)(2n+3)(2n+4)}{(2n+3)(2n+4)-6}|a_n|\\
 \geq &(2n+1)(2n+2)|a_n|,
\end{split}
\]
which implies \eqref{def:angrowth}.
\end{proof}

The fact that $\Theta\neq 0$ is a direct consequence of Lemma \ref{lemma:series}. By the third statement of Theorem \ref{thm:inner}, it is enough to prove that there exists $z_0\in \RRR_{\theta,\kk}^\inn$ such that $\Delta\phi^0(z_0)\neq 0$, or equivalently
\[
\Phi^\uns(z_0)-\Phi^\sta(z_0)\neq 0.
\]
We argue by contradiction. Assume that $\Phi^\uns(z)=\Phi^\sta(z)$ for all $z\in \RRR_{\theta,\kk}^\inn$. Since, by Theorem \ref{thm:inner}, $\Phi^\uns$, $\Phi^\sta$ are real-analytic, they must coincide also in 
\[
\overline{\RRR}_{\theta,\kk}^\inn=\left\{z:\overline{z}\in \RRR_{\theta,\kk}^\inn\right\}.
\]
Therefore, the functions $\Phi^\uns$, $\Phi^\sta$ can be analytically extended to the neighborhood of infinity $|z|\geq \kk$ and, thus, are analytic at infinity. This contradicts the fact that the asymptotic series of these functions at infinity have coefficients growing  faster than a factorial.

\section{The right inverses of $\LL_1$}
\label{app:g1:outer}
\setcounter{equation}{0}

Here we prove Lemmas~\ref{lem:zeta2:welldefined}, \ref{prop_operators}, \ref{lem:zeta2:welldefined:aux}, 
and~\ref{lemma:operators:aux}. 

\subsection{Proof of Lemmas~\ref{lem:zeta2:welldefined} and~\ref{lem:zeta2:welldefined:aux}}\label{appendix:zeta2}

We first prove Lemma~\ref{lem:zeta2:welldefined} in Section~\ref{appendix:zeta2:out}. Then, we prove Lemma~\ref{lem:zeta2:welldefined:aux} in Section~\ref{appendix:zeta2:aux} as an straightforward consequence of Lemma~\ref{lem:zeta2:welldefined}.

\subsubsection{Proof of Lemma~\ref{lem:zeta2:welldefined}}\label{appendix:zeta2:out}
Let $\zeta_1 (x)= u_0'(x)$. In $D^{\out,\uns}_\kappa$, it only vanishes at $x=0$ 
(see~\ref{def:soliton}). We rewrite~\eqref{wronskian} as 
\[
\left (\frac{\zeta_2}{\zeta_1} \right )' = \frac{1}{\zeta_1^2},
\]
which is equivalent at the domain $D^{\out,\uns}_\kappa \backslash \{0\}$. For $x\in B_r \subset \mathbb{C}$, the open ball centered at the origin of radius $r$,   
\[
\zeta_1(x) = \sum_{k=1}^\infty c_k x^{2k-1}, \qquad c_1 \neq 0.
\]
Therefore, writing $\wh \zeta_2 = \zeta_2 \zeta_1^{-1}$ we have that 
\[
\wh \zeta_2'(x) =  \frac{1}{\zeta_1^2(x)}=\frac{1}{c_1 x^2} \sum_{k=0}^{\infty} d_k x^{2k}
\]
which implies that 
\begin{equation}\label{defzeta2hat:appendix}
\wh \zeta_2 (x) = -\frac{1}{c_1 x} + c_0+ \sum_{k=1}^{\infty} \frac{d_k}{2k} x^{2k-1}, \qquad x\in B_r. 
\end{equation}
As a consequence, taking $c_0=0$ yields 
\begin{equation}\label{zeta2:appendix}
\zeta_2 (x) = \wh \zeta_2(x) \zeta_1(x) = -1  + \sum_{k=1} \hat c_k x^{2k} , \qquad x \in B_r,
\end{equation}
which defines an even real analytic function in $B_r$. Notice that $\zeta_2(0)=-1\neq 0$. For $x\in D_{\kappa}^{\out,\uns} \backslash B_r$, we define $\zeta_2(x)$ as
\begin{equation}\label{definitionzeta2:appendix}
\zeta_2(x) = \begin{cases} 
\zeta_1(x) \left [ \wh \zeta_2(r) + \int_{r}^x \frac{1}{\zeta_1^2(s)}\, ds \right ] & \text{if }  \Re x\geq 0, \\ 
\zeta_1(x) \left [ \wh \zeta_2(-r) + \int_{-r}^x \frac{1}{\zeta_1^2(s)}\, ds \right ] & \text{if }   \Re x< 0,
\end{cases}
\end{equation}
with $\wh \zeta_2$ defined in~\eqref{defzeta2hat:appendix},
which is the even analytic extension at $D_{\kappa}^{\out,\uns}$ of $\zeta_2$ defined in~\eqref{zeta2:appendix}.  

We notice that since $\zeta_1 =u_0'\in \mathcal{E}_{1,2}$, 
then for $x\in D^{\out,\uns}_\kappa \cap \{\Re x \leq -10\}$, 
\[
|\zeta_2 (x)| \lesssim \frac{1}{|\cosh x|} \left [1 +  \int_{\Re x}^{-r} \cosh^2 s \, ds  \right ] \lesssim \cosh \Re x \lesssim |\cosh x|,
\]
where we have used $\cosh \Re s \lesssim |\cosh s|\lesssim \cosh \Re s$. 

When $x\in D^{\out,\uns}_\kappa \cap \{\Re x \geq -10\}$, 
$
|\zeta_2(x)| \lesssim |\zeta_1(x)|
$ and we conclude that $\zeta_2 \in \mathcal{E}_{-1,2}$.

\subsubsection{Proof of Lemma~\ref{lem:zeta2:welldefined:aux}}\label{appendix:zeta2:aux}

On $D_\kappa^\aux$, see~\eqref{def:domainaux} and Figure~\ref{fig:aux},  $\zeta_1$ has simple zeroes at $0,i \pi, -i \pi$. Then, denoting $x_0=0,i\pi,-i\pi$, one has $\zeta_1(x)=\zeta_1'(x_0) (x-x_0) + \mathcal{O}(x-x_0)^2$ with $\zeta_1'(x_0)\neq 0$, and, as a consequence, when $x$ goes to $x_0$ in definition~\eqref{definitionzeta2:appendix} of $\zeta_2$, 
we have 
\[
\lim_{x\to x_0} \zeta_2(x)=\lim_{x\to x_0} \zeta_1(x) \int_{\pm r}^x \frac{1}{\zeta_1^2(s)} \, ds = -\frac{1}{\zeta_1'(x_0)}.
\]
In addition, $x_0$ do not belong to the segment between $x\in D_\kappa^\aux$ and $\pm r$ and then we conclude that $\zeta_2$ defined in~\eqref{definitionzeta2:appendix} is, in fact, well defined and real analytic also at $D_\kappa^\aux$. Finally, using that $\zeta_1=u_0' \in \D\mathcal{Y}_2^1$, where $D\mathcal{Y}_\ell^1)$ is defined by  ~\eqref{def:Banach:aux}, we obtain the result.

\subsection{Fundamental solutions of $\mathcal{L}_1[\zeta]=0$}

Here we provide new sets of fundamental solutions of the linear second order differential equation $\mathcal{L}_1[\zeta]=0$, where $\mathcal{L}_1$ is defined in~\eqref{def:diffoperators}. We mainly follow the strategy in~\cite{GomideGSZ22}, being the first result below an adaptation of Lemma A.1 
in~\cite{GomideGSZ22}.

We fix the complex rectangle 
\begin{equation}\label{def:rectangle:appendix}
    R=\{x\in \mathbb{C}: \quad -10 \leq \Re x \leq 0,\, \;\; |\Im x| \leq  2\pi\}
\end{equation}
and we emphasize that, by Lemma~\ref{lemma:propertiesu0} $\zeta_1 =u_0'$ is analytic in $R\backslash\{x_-,\ol{x_-}\}$. 
\begin{lemma}\label{lemma:G1:appendix}
    Let \[
\zeta_{+} (x) = \zeta_1(x) \int_{x_{-}}^x \frac{1}{\zeta_1^2(s)} \, ds, \qquad \zeta_{-} (x) = \zeta_1(x) \int_{\ol{x_{-}}}^x \frac{1}{\zeta_1^2(s)} \, ds.
\]
Then, 
\begin{itemize}
    \item $\zeta_{\pm}$ are analytic solutions of $\LL_1[\zeta]=0$ in the domain $R\backslash\{x_-,\ol{x_-}\}$ satisfying
    \[
    W(\zeta_+,\zeta_-) = \zeta_+ \zeta_-' - \zeta_+'\zeta_-=\int_{x_-}^{\ol{x_-}} \frac{1}{\zeta_1^2 (s)} \, ds \neq 0.
    \]
    \item They satisfy, for $x\in R$ with $R$ defined in~\eqref{def:rectangle:appendix},
    \begin{equation}\label{def:zetapm}
    \zeta_{+}(x)= \frac{ (x- x_- )^3}{(x-\ol{x_-})^2} \wh \zeta_+(x), \qquad  \zeta_{-}(x)= \frac{ (x- \ol{x_-} )^3}{(x-x_-)^2} \wh \zeta_-(x)
    \end{equation}
    where $\wh \zeta_{\pm}$ are analytic functions in $R$ and $|\wh \zeta_{\pm}(x)|\leq M $ for some constant $M$ (independent on $x$). 
    \item For some constant $c$, we have 
    \begin{equation}\label{relationzeta12zetapm}
\zeta_1(x) = \frac{1}{W(\zeta_+,\zeta_-)} \big (\zeta_+(x) - \zeta_-(x) \big ),\qquad \zeta_2(x) = c \zeta_1(x) + \zeta_-(x).
\end{equation} 
\end{itemize}
\end{lemma}

\begin{proof}
On the rectangle $R$ in \eqref{def:rectangle:appendix}, the function  $\zeta_1(x)=u_0'(x)$, see ~\eqref{def:soliton},  has simple zeroes only at $x=0,\pm i \pi, \pm i 2\pi $, that is, writing $x_0=0,i\pi,-\pi$,
$\zeta_1(x) = \zeta_1'(x_0) (x-x_0) + \mathcal{O}(x-x_0)^2$ when 
$x$ is close to $x_0$. Moreover, for all $x\in R$, the segments $\overline{x, x_-}$ and $\overline{x,\ol{x_-}}$ do not cross $x_0$. Then, since $\zeta_1'(x_0)\neq 0$, 
\[
\lim_{x\to x_0} \zeta_\pm (x) = -\frac{1}{\zeta_1'(x_0)}
\]
that implies that $\zeta_\pm$ are well defined at the set $R$. In addition, the fact that $\zeta_1^{-2}$ has zeroes of order $4$ at $x_-,\ol{x_-}$ and it is uniformly bounded at $R$, implies that the estimates in~\eqref{def:zetapm} follow immediately and hence the second item of Lemma \ref{lemma:G1:appendix} is already proven.  

From the definition of $\zeta_\pm$, one can easily compute $W(\zeta_+,\zeta_-)$. We check that it is not zero. Indeed, we define 
\[
\wt{u}_0(t) = u_0(-\alpha + i t) = \frac{3}{\cos t +1 -   3 \sqrt{|\gamma|} i\sin t }
\]
and, after some tedious computations, we have that 
\[
\frac{1}{(u_0'(-\alpha+it))^2} = -\frac{1}{(\wt{u}'_0(t))^2 }=  \frac{(\cos t +1 -  3 \sqrt{|\gamma|} i \sin t)^4}{9 (\sin t+   3\sqrt{|\gamma|}i \cos t)^2}.
\]
Then, again performing some tedious but straightforward computations, we obtain  
\[
\int_{x_-}^{\ol{x_-}} \frac{1}{\zeta_1^2(s)} \, ds = -i \int_{\pi}^{-\pi} \frac{1}{(\wt u_0(t))^2 }\, dt = 3\pi i 
 \left (  |\gamma| - \frac{5}{9} \right ).
 \] 
This ends the proof of the first item of Lemma \ref{lemma:G1:appendix}. 

Finally, we prove the third item of Lemma \ref{lemma:G1:appendix}.  
By the first item, $\zeta_+,\zeta_-$ are independent solutions of $\mathcal{L}_1[\zeta]=0$, so that $\zeta_1= c_1   \zeta_- + c_2 \zeta_+$. Evaluating at $x_-,\ol{x_-}$ we obtain the coefficients $c_1,c_2$ and the formula for $\zeta_1$. On the other hand, $\zeta_2$ is a linear combination of $\zeta_+,\zeta_-$, which yields~\eqref{relationzeta12zetapm} since  $W(\zeta_1,\zeta_2) \neq 0$.
\end{proof}

Now we study 
  \begin{equation}\label{defJ:appendix}
            J_\pm(x):=\left | \zeta_\pm (x) \int_{0}^{x} \zeta_\mp(s) h(s) \, ds \right | ,
    \end{equation}
which play a key role when bounding the norm of the linear operators $\GG_1, \wt{\GG}_1$ defined in~\eqref{def:intoperators-1} and~\eqref{def:intoperators-aux} respectively. Since these operators are defined over analytic functions in different domains, we introduce a new class of domains that posses the minimal properties we need to be able to bound $J_\pm$. 

\begin{definition}\label{def:newdomain:appendix}
  Let $D \subset  R$, with $R$ defined in~\eqref{def:rectangle:appendix}, be a closed bounded domain satisfying that 
\begin{itemize}
    \item $0\in \mathrm{int}(D)$, $x_-, \ol{x_-} \notin  D $,
    \item if $x \in D$, then $\Re x \in D$ and the segments
    $\overline{0,x} \in D$, $\overline{x, \Re x }\subset D $,
    \item there exists a constant $\vartheta\in (0,1)$ such that
    if $x\in D$ either $|\Im x |< \pi$, or 
    \[
    |\Re x + \alpha| \geq \vartheta \min\{ |x- x_-|, |x-\ol{x_-}|\}.
    \]
\end{itemize}  
\end{definition}

\begin{remark}\label{domains:appendix}
Notice that $D_{\kappa}^{\out,\uns}\cap \{-10<\Re < 0\}$ in ~\eqref{def:defdomainouter} and $D_\kappa^{\aux}$ in ~\eqref{def:domainaux} satisfy the conditions in Definition~\ref{def:newdomain:appendix}.   
\end{remark}

\begin{lemma}\label{boundingx0:appendix}
Let $D$ be a domain satisfying the conditions in Definition~\ref{def:newdomain:appendix} and fix $\ell \geq 5$. If $h:D\to \mathbb{C}$, then 
    \begin{equation*}
           \left| J_\pm(x) \right| \lesssim \frac{\lfloor h\rfloor_{\ell}}{|x-x_-|^{\ell -2 } |x- \ol{x_-}|^{\ell -2}}, \qquad x\in D,
    \end{equation*}
where $J_\pm$ has been introduced in~\eqref{defJ:appendix} and 
\[
\lfloor h \rfloor_\ell = \sup_{x\in D} |h(x)| |x-x_-|^{\ell} |x- \ol{x_-}|^\ell.
\] 
\end{lemma}
\begin{proof}
We recall that $x_-=-\alpha + i \pi$ with $\alpha>0$. We only provide the details for $J_+$ being the corresponding for $J_-$ analogous. 
When $x\in D \cap\{ x\in \mathbb{C}: \Re x\geq -\frac{\alpha}{2}\}$, then, using the second item in Lemma~\ref{lemma:G1:appendix},  
\begin{equation*}
\begin{aligned}
\left | \zeta_+(x) \int_{0}^{x} \zeta_-(s) h(s) \, ds \right | & \lesssim \frac{|x-x_-|^3}{|x-\ol{x_-}|^2 }  \int_{0}^x \frac{\lfloor h \rfloor_{ \ell}}{|s-x_-|^{\ell +2 } |s-\ol{x_-}|^{\ell -3}} \\
&  
\lesssim\frac{\lfloor h \rfloor_{ \ell}}{|x-x_-|^{\ell -2 } |x- \ol{x_-}|^{\ell -2}}.
\end{aligned}
\end{equation*}
Now we deal with $x\in D\cap \{ x\in \mathbb{C}: \Re x<-\frac{\alpha}{2}\}$. Since, by Lemma~\ref{lemma:G1:appendix}, $\zeta_\pm$ and $h$ are analytic functions in $D\subset R$, we write
    \begin{align*}
    \zeta_+(x)\int_{0}^x \zeta_- (s) h(s) \, ds   & =
    \zeta_+(x)\left [\int_{\gamma_1} \zeta_-(s) h(s)\, ds + \int_{\gamma_2} \zeta_-(s) h(s)\, ds
    \right ]\\ & =: G_1(x) + G_2(x),
    \end{align*}
    with $\gamma_1(t)= -t$, for $t\in [0, -\Re x]$ and $\gamma_2(t) = \Re x +it$, for $t\in \overline{0, \Im x}$. Notice that, by Definition~\ref{def:newdomain:appendix} of $D$, the paths $\gamma_1,\gamma_2 \subset D$. Then, we obtain 
    \begin{align*}
         |G_1(x)|=\left |\zeta_+(x) \int_{\gamma_1} \zeta_-(s) h(s) \, ds \right | & \lesssim  \frac{|x-x_-|^3} {|x-\ol{x_-}|^2 }\left |\int_{0}^{|\Re x|} \frac{\lfloor h \rfloor_{ \ell}}{|t +x_-|^{\ell +2 } |t+\ol{x_-}|^{\ell -3}} \right | \\ &  \lesssim
         \frac{\lfloor h \rfloor_{ \ell}} {|x-x_-|^{\ell -2 } |x- \ol{x_-}|^{\ell -2}},
    \end{align*}
where we have used that $|t + x_-|, |t + \ol{x_-}| \geq \pi$ and that $|x|\lesssim 1$. 

With respect to $G_2$,   we have that 
\[
|G_2(x)| \lesssim \lfloor h \rfloor_\ell \frac{|x-x_-|^3} {|x-\ol{x_-}|^2 } \left | \int_{0}^{\Im x} \frac{1}{ |\Re x +i t - x_-|^{\ell +2} |\Re x +i t-\ol{x_-}|^{\ell -3} }\, dt \right |.
\]
Then, if $\Im x\geq 0$, since $ |\Re x + i t - \ol{x_-}|\geq \pi$, for $t\in [0,\Im x]$, we have that 
\[
|G_2(x)| \lesssim \lfloor h \rfloor_\ell  \frac{|x-x_-|^3} {|x-\ol{x_-}|^2 }  \int_{0}^{\Im x} \frac{1}{\big ( (\Re x + \alpha)^2 + (t-\pi)^2 \big )^{\frac{\ell+2}{2}} 
  }\, dt. 
\]
In the case $|\Re x + \alpha| \geq \vartheta |x- x_-|$, 
\begin{align*}
|G_2(x) | & \lesssim \lfloor h \rfloor_\ell \frac{|x-x_-|^3} {|x-\ol{x_-}|^2 } \frac{1}{|\Re x + \alpha|^{\ell +1}} \int_{-\infty}^{+\infty} \frac{1}{ (1 + t^2)^{\frac{\ell +2}{2}}} \, dt \\ & \lesssim \lfloor h \rfloor_\ell  \frac{1}{|x-\ol{x_-}|^2 
|x-x_-|^{\ell -2}},
\end{align*}
and the result follows provided $\ell \geq 5$. If 
$|\Re x + \alpha|\leq \vartheta |x-x_-|$, then $0\leq \Im x <\pi$ and $\pi-\Im x  \geq \sqrt{1-\vartheta^2} |x-x_-|$.  We obtain
\[
|G_2(x)|\lesssim \lfloor h \rfloor_\ell \frac{|x-x_-|^3} {|x-\ol{x_-}|^2 } \int_{0}^{\Im x} \frac{1}{(\pi -t)^{\ell +2}} \, dt \lesssim 
\lfloor h \rfloor_\ell  \frac{|x-x_-|^3} {|x-\ol{x_-}|^2 (\pi - \Im x)^{\ell +1}},
\]
and the result follows trivially also in this case. 

The details in the case $\Im x\leq 0$ are left to the reader. 
\end{proof}
\subsection{Proof of Lemma~\ref{prop_operators}}

The result related to $\GG_2$ defined in~\eqref{def:intoperators-2} is a straightforward consequence of Lemma 5.5 in~\cite{GOS10}. 

We focus now on proving the results related to $\GG_1$. To do so we follow the main ingredients in the proof of Proposition 4.3 in~\cite{GomideGSZ22}. 
When  $x\in D^{\out,\uns}_\kappa \cap \{\Re x \leq -10\}$, by Lemma~\ref{lem:zeta2:welldefined},  $|\zeta_2 (x)| \lesssim |\cosh x|$.  From here, using also that $|\zeta_1(x) | \lesssim |\cosh x|^{-1}$ and following exactly the same steps as the ones in~\cite{GomideGSZ22}, we prove that
\[
| \cosh x|^{m} \big |\GG_1[h](x) \big | \lesssim \| h\|_{m,\ell}, \qquad x\in D^{\out,\uns}_\kappa \cap \{\Re x \leq -10\}.
\]   

The case  $x\in D^{\out,\uns}_\kappa \cap \{\Re x \geq -10\}$ is more involved. Indeed, the main obstacle to overcome is that $\zeta_1,\zeta_2$ have poles of order $2$ at $x=x_-, \ol{x_-}$. Following~\cite{GomideGSZ22}  we rewrite $\GG_1$ in~\eqref{def:intoperators-1} in terms of $\zeta_+,\zeta_-$ in Lemma~\ref{lemma:G1:appendix}. Using the third item of this result, we obtain that 
    \begin{align*}
    \GG_1[h] (x) =  & \frac{1}{W(\zeta_+,\zeta_-)} \left[ \zeta_+(x) \int_{0}^x \zeta_-(s) h(s) \, ds - \zeta_-(x) \int_{0}^x \zeta_+(s) h(s)\, ds\right ] \\ & - \zeta_2(x) \int_{-\infty}^0 \zeta_1(s) h(s)\, ds.
    \end{align*}
By Remark~\ref{domains:appendix}, we can use the results in Lemma~\ref{boundingx0:appendix} to bound the two first integrals defining $\mathcal{G}_1[h]$. To bound the third integral, we claim that is a convergent real integral and that $\|\zeta_2\|_{-1,2}\lesssim 1$. Then, 
\[
\left |\zeta_2 (x) \int_{-\infty}^0 \zeta_1(s) h(s)\, ds \right | \lesssim \left |\zeta_2 (x)\right|\|h\|_{m,\ell} \lesssim \frac{\|h\|_{m,\ell}}{|x-x_-|^2|x-\ol{x_-}|^2}.
\]
Again, using that $\ell\geq 5$, the first bound in Lemma~\ref{prop_operators} is proven. To prove $\|\partial_x \mathcal{G}_1[h]\|_{1,\ell -1}$ we proceed analogously. Indeed, we have that
    \begin{align*}
    \partial_x \GG_1[h] (x) =  & \frac{1}{W(\zeta_+,\zeta_-)} \left[ \zeta'_+(x)  \int_{0}^x \zeta_-(s) h(s) \, ds - \zeta_-'(x) \int_{0}^x \zeta_+(s) h(s)\, ds\right ] \\ & - \zeta_2'(x) \int_{-\infty}^0 \zeta_1(s) h(s)\, ds,
    \end{align*}
where 
\[
\zeta_+'(x)=\frac{(x-x_-)^2}{(x-\ol{x_-})^3} \wt{\zeta}_+(x), \qquad \zeta_-'(x)=\frac{(x-\ol{x_-})^2}{(x-x_-)^3} \wt{\zeta}_-(x)
\]
for $\wt{\zeta}_{\pm}$ are analytic functions uniformly bounded at $R$. 

To complete  the proof of Lemma~\ref{prop_operators}, we just recall that, by Lemma~\ref{lem:zeta2:welldefined}, $\zeta_2$ is an even function. 

\subsection{Proof of Lemma~\ref{lemma:operators:aux}}

We first notice that using relations~\eqref{relationzeta12zetapm} between $\zeta_1,\zeta_2$ and $\zeta_{+},\zeta_-$ we have that 
\begin{equation}\label{appendixwtg1}
\wt{\GG}_1[h](x) = \frac{1}{W(\zeta_+,\zeta_-)} \left ( \zeta_+(x)\int_{0}^x \zeta_-(s) h(s)\, ds - \zeta_-(x) \int_{0}^x \zeta_+(x) h(s) \, ds\right ) 
\end{equation}
and that by Remark~\ref{domains:appendix}, we can apply the results in Lemma~\ref{lemma:G1:appendix} for 
 $x\in D^{\aux}_{\kappa} \cap \{ x \in \mathbb{C} : \Re x \leq 0\}$. Then,  
 we have 
\[
|\wt{\GG}_1[h](x) | \lesssim\frac{\|h\|_{ \ell}}{|x-x_-|^{\ell -2 } |x- \ol{x_-}|^{\ell -2}},
\]
so that, since $1 \lesssim |x-x_+|, |x-\ol{x_+}|$, for 
$x\in D^{\aux}_{\kappa} \cap \{ x \in \mathbb{C} : \Re x \leq 0\}$, we obtain 
\begin{equation}\label{boundwtG1}
|\wt{\GG}_1[h](x) | |x-x_-|^{\ell -2 } |x- \ol{x_-}|^{\ell -2} |x_-x_+|^{\ell -2 } |x- \ol{x_+}|^{\ell -2} \lesssim \|h \|_\ell .
\end{equation}
When $x\in D^{\aux}_{\kappa} \cap \{ x \in \mathbb{C} : \Re x \geq 0\}$,  we only need to define the new set of fundamental solutions of $\mathcal{L}_1[h]$ given by 
\[
\wt{\zeta}_+(x) = \zeta_1(x) \int_{x_+}^x \frac{1}{\zeta_1^2(s)}\, ds , \qquad \wt{\zeta}_-(x)= \zeta_1(x)\int_{\ol{x_+}}^x \frac{1}{\zeta_1^2(s)} \, ds
\]
and proceeding in an analogous way as for $x\in D^{\aux}_{\kappa} \cap \{ x \in \mathbb{C} : \Re x \leq 0\}$ to obtain the bound~\eqref{boundwtG1} for $x\in D_\kappa^\aux$. By definition~\eqref{normcosh:aux} of the norm,
$\| \wt{\GG}_1[h]\|_{\ell -2 } \lesssim \|h\|_{\ell}$. 

Differentiating~\eqref{appendixwtg1} with respect to $x$ and performing similar bounds as the previous one, we prove the result for $\partial_x \GG_1[h]$.

For the operator $\wt \GG_2$ in~\eqref{def:intoperators-aux}, we take $x\in D_\kappa^\aux$ be such that $\Re x\leq 0$ since the case $\Re x\geq 0$ is analogous. In this case $1\lesssim |x-x_+|, |x-\ol{x_+}|$ and hence we have to prove
\[
|\GG_2(x)| \lesssim \eps^2 \frac{\|h\|_\ell} {|x-x_-|^\ell |x-\ol{x_-}|^\ell }.
\]
By  definition~\eqref{def:intoperators-aux} of $\GG_2$ it is enough to prove that for $\Re x\leq 0$, 
\begin{equation}\label{boundG21:appendix}
\left | e^{ \pm i\eps^{-1} x}\int_{\mp i\rho}^x e^{\mp i\eps^{-1} s }h(s) \, ds \right |   \lesssim \eps \frac{\|h\|_\ell} {|x-x_-|^\ell |x-\ol{x_-}|^\ell }.
\end{equation}
We deal with the  bound for the integral from $-i\rho$. To prove the second one is analogous. We write
\begin{align*}
e^{i\eps^{-1} x}\int_{-i\rho}^x e^{-i\eps^{-1} s }h(s) \, ds &  =e^{i\eps^{-1} x} \int_{\gamma_1} e^{-i\eps^{-1} s }h(s) \, ds + e^{i\eps^{-1} x}\int_{\gamma_2} e^{-i\eps^{-1} s }h(s) \, ds 
\\ & =:  G_1(x) + G_2(x),
\end{align*}
where the paths $\gamma_1,\gamma_2$ are defined by
\[
\gamma_1(t)=x+te^{ -i\vartheta}, \quad t\in \overline{0,-\mathrm{sec} \vartheta\,  \Re x }, \qquad 
\gamma_2(t) = it,\quad t \in \overline{\tan \vartheta \, \Re x,-\rho} 
\]
with $\vartheta>0$ such that $\gamma_1(t) \in D_\kappa^\aux$. We recall that $\Re x\leq 0$ and hence $1\lesssim |x-x_+|, |x-\ol{x_+}|$. Therefore, 
\[
|G_1(x)| \lesssim \|h \|_\ell  \int_{0}^{\mathrm{sec} \vartheta |\Re x|}
 \frac{e^{-\eps^{-1} t \sin \vartheta}}{|x - x_-+ t e^{-i\vartheta }|^{\ell}|x - \ol{x_-}+ t e^{-i\vartheta }|^{\ell}} \, dt.
\]
The geometry of the set $D_\kappa^\aux$ implies that 
\[
|x - x_- + t e^{-\i \vartheta}| \gtrsim |x-x_-|, \qquad 
|x - \ol{x_-} + t e^{-\i \vartheta}| \gtrsim |x-\ol{x_-}|,
\]
hence
\[
|G_1(x)| \lesssim \frac{\| h \|_\ell}{|x-x_-|^\ell |x-\ol{x_-}|^\ell } \int_{0}^\infty e^{-\eps^{-1} t \sin \vartheta} \, dt \lesssim  \eps \frac{\| h \|_\ell}{|x-x_-|^\ell |x-\ol{x_-}|^\ell }.
\]

The bound for $G_2$ follows using the same arguments. It is clear that $|x-x_-+it|\gtrsim |x- x_-|$ and $|x-\ol{x_-}| \gtrsim |x-\ol{x_-}|$. Hence,
\begin{align*}
|G_2(x)| & \lesssim \|h \|_\ell 
 \int_{-\rho}^{ -\tan \vartheta |\Re x|}  \frac{e^{\eps^{-1} t}}{|x-x_-+ it|^{\ell}|x - \ol{x_-}+ it|^{\ell}} \, dt \\ 
& \lesssim \frac{\| h \|_\ell}{|x-x_-|^\ell |x-\ol{x_-}|^\ell } \int_{-\infty}^0 e^{\eps^{-1} t}\, dt \lesssim \eps \frac{\| h \|_\ell}{|x-x_-|^\ell |x-\ol{x_-}|^\ell }.
\end{align*}
As a consequence,~\eqref{boundG21:appendix} is proven. 

\bibliography{references}
\bibliographystyle{abbrv}

\end{document}